\documentclass{ws-m3as}

\usepackage{tabularx}
\usepackage{booktabs}
\usepackage{graphicx}
\usepackage{subcaption}
\usepackage[numbers]{natbib}
\usepackage{hyperref}
\usepackage{xcolor}
\usepackage{tikz}
\usepackage{graphicx} 
\usepackage{lipsum}
\usepackage{tikz} 
\usepackage{bm}
\allowdisplaybreaks

\begin{document}

\markboth{L. Zhang et al.}{A novel phase-field model for $N$-phase problems}

%
\catchline{}{}{}{}{}
%

\title{A novel phase-field model for $N$-phase problems: modeling, asymptotic analysis and numerical simulations}

\author{Lun Zhang}

\address{Department of Mathematics, Southern University of Science and Technology,\\
Shenzhen 518055, P. R. China\\
12231273@mail.sustech.edu.cn}

\author{Chenxi Wang}

\address{School of Science, Harbin Institute of Technology,\\
Shenzhen 518055, P. R. China\\
wangcx2017@mail.sustech.edu.cn}

\author{Nan Lu}

\address{Department of Mathematics, Southern University of Science and Technology,\\
Shenzhen 518055, P. R. China\\
lun@sustech.edu.cn}

\author{Zhen Zhang\footnote{Corresponding author.}}

\address{Department of Mathematics, National Center for Applied Mathematics (Shenzhen),\\
Southern University of Science and Technology, Shenzhen 518055, P. R. China\\
zhangz@sustech.edu.cn}

\maketitle

\begin{history}
\received{(Day Month Year)}
\revised{(Day Month Year)}
\comby{(xxxxxxxxxx)}
\end{history}

\begin{abstract}
Multiphase problems involving complex interfaces have found widespread applications in fluid mechanics, materials science, biology, and image processing, and thus have attracted increasing research interests. 
The classical phase-field  modeling approaches for multiphase problems represent each phase using a regularized characteristic function, which necessarily introduces a simplex constraint for the phase-field variables. Additionally, the consistency requirement for phase-field modeling brings difficulties to the construction of nonlinear potentials in the energy functionals, posing significant challenges for classical phase-field modeling and its numerical methods for problems involving many phases. In this work, by adopting a dichotomic approach to represent multiphase, we propose a novel phase-field modeling framework without simplex constraint,
in which the free energy is interpolated from the classical two-phase Ginzburg-Landau free energies. We systematically establish the interpolation rules and explicitly construct the interpolation functions, rendering the consistency properties of the model. 
The proposed model enjoys an energy dissipation property and is shown to be asymptotically consistent with its sharp interface limit, with the Neumann triangle condition recovered at the triple junction.
Based on a mobility operator splitting technique, we develop a linear, decoupled, and energy stable scheme for efficiently solving the system of phase-field equations. 
The numerical stability and accuracy, as well as the consistency properties of the model, are validated through a large number of numerical examples. 
In particular, the model demonstrates its success in several benchmark simulations for multiphase problems, such as the formation of liquid lenses between two stratified fluids,  the generation of double emulsions and Janus emulsions, showing good agreement with experimental observations.
\end{abstract}

\keywords{phase-field; multiphase system; sharp interface limit; operator splitting; energy stability}

\ccode{AMS Subject Classification: 22E46, 53C35, 57S20}

\section{Introduction}
Phase-field methods have become powerful and widely adopted tools for modeling multiphase systems in various scientific areas. In fluid dynamics, phase-field variables are introduced to capture interfacial evolution and phase interactions in multiphase flows and bubble dynamics \cite{boyer2006study,boyer2011numerical,Analysisof,XuThree2019,AiharaHighly2023,ZhangA2022,ZhangMulti2024}. In materials science, phase-field dynamics enable the simulation of phase transitions and microstructural evolution to optimize material properties \cite{Diffuseinterface,LiAnisotropic2021,ZhangCavitation2023,DondlPhase2024,MaHigh2023}. In image processing, phase-field functions are employed to label multiple classes, and therefore image segmentation and edge detection can be achieved by minimizing proper  energy functionals with data fidelity \cite{YangImage2019,QiaoTwo2022,LiuTwo2022,LiuMulti2023,WangMulti2022}. In mathematical biology problems, such as cell behavior, tumor growth, and tissue development, phase-field variables are usually used to capture the distribution and interaction of distinct cellular components and simulate dynamic processes such as vesicle formation, fusion, fission, and phase separation \cite{GuA2016,TangPhase2023,AshourPhase2023,YangPhase2023,WenHydrodynamics2024}. Phase-field methods have demonstrated their strong capability to study the dynamics of complex multiphase systems.

In a typical phase-field model for phase separation problems, a smooth order parameter (namely the phase-field function) is usually introduced as an approximation of the indicator function of one bulk phase. Then the interface between two phases is characterized by the smooth transition of the phase-field function. The dynamics of the whole system can be variationally derived from a proper energy functional of this phase-field function. This idea can be generalized to model multiphase systems. For instance, to model a ternary phase system, a characteristic-based phase-field (CBPF) vector $(c_1,c_2,c_3)$ is introduced as approximation of the characteristic functions  of the three phases (sketched in Fig. \ref{three phase1}), subjecting to the hyperplane constraint $c_1+c_2+c_3=1$ for volume conservation and non-negative constraints $c_i\geqslant0$. However, such constraints bring difficulties in designing structure preserving numerical methods for the resulting model. In addition, it is not easy to construct energies in the classical ternary CBPF models due to the consistency constraints arising from physical considerations. For example, the algebraic consistency is defined to impose the coincidence of the ternary phase-field model with the two-phase model in the absence of one phase, preventing the model from having unphysical nucleation at the interface between any two phases. This property must be preserved under small perturbations, leading to dynamic consistency. In order to maintain these two consistencies, Boyer et al. \cite{boyer2006study} suggested that both the free energy and the mobility coefficients should be carefully constructed up to crucial mathematical constraints. 
In some case, the energy functional may not be bounded from below because of the negative capillary energy part. 
Significant challenges in both modeling and analysis remain in the classical ternary CBPF model.
\begin{figure}[htb]
   \centering
   \begin{subfigure}{0.35\linewidth}
   \centering
   \includegraphics[trim=0cm 0.3cm 0cm 0.2cm,clip,width=0.9\linewidth]{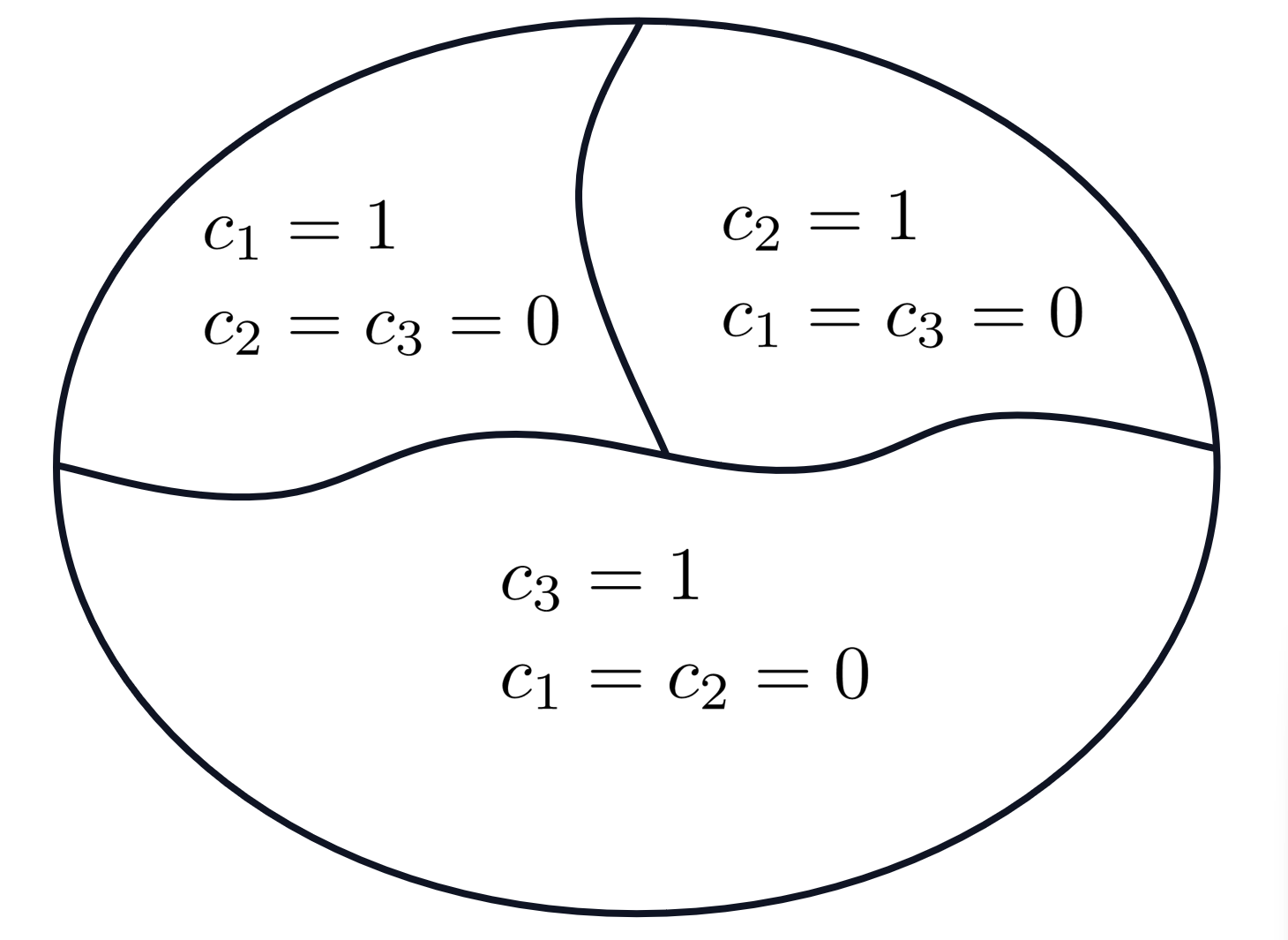}
   \subcaption{}\label{three phase1}
   \end{subfigure}\hspace{5mm}
   \begin{subfigure}{0.35\linewidth}
   \centering
   \includegraphics[trim=0cm 0.3cm 0cm 0.2cm,clip,width=0.9\linewidth]{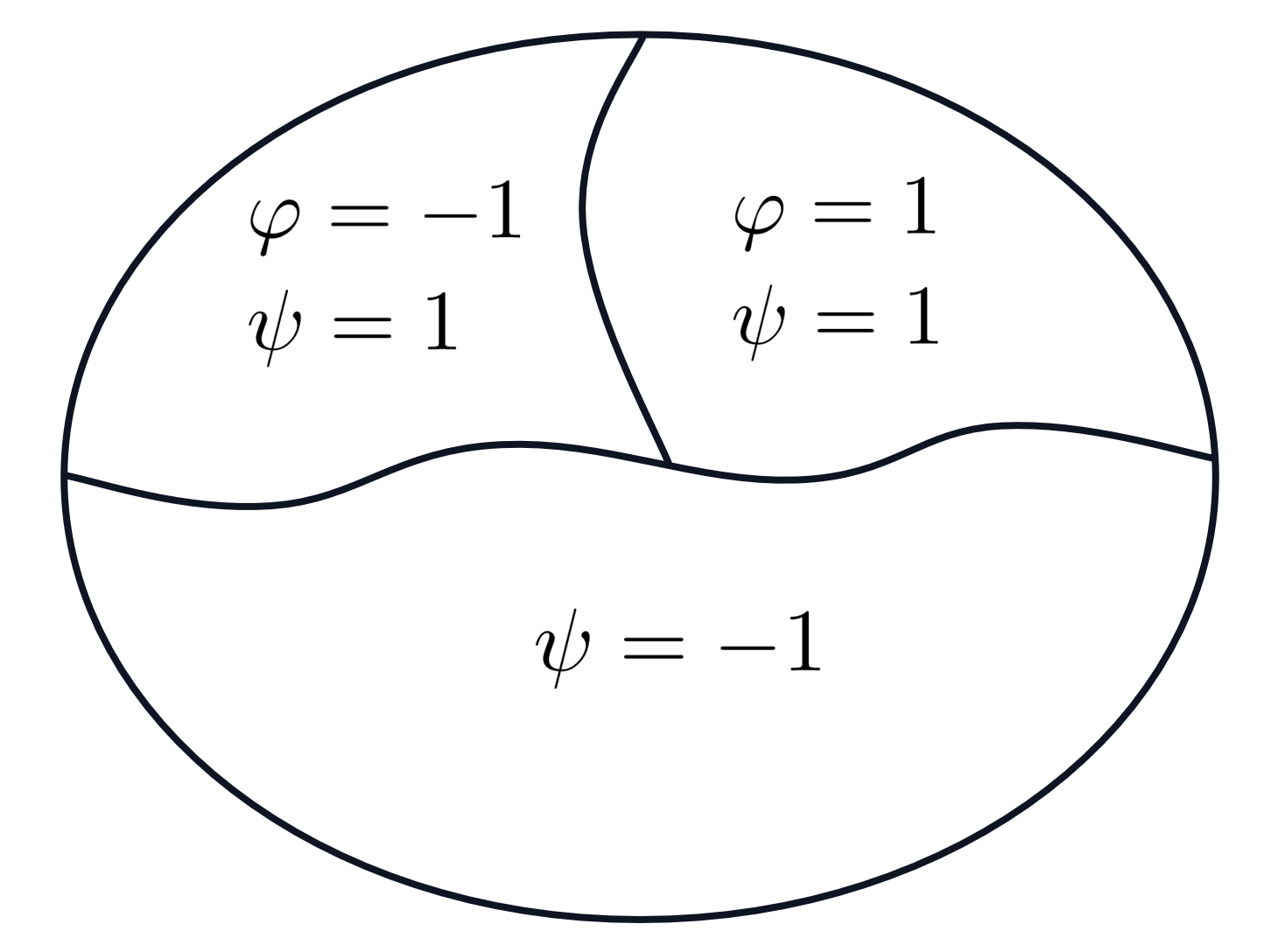}
   \subcaption{}\label{three phase2}
   \end{subfigure}
   \caption{(a) Configuration of a phase-field vector $(c_1,c_2,c_3)$ representing ternary phase. The bulks of the three phases in CBPF model are characterized by $\bar{\mathcal{B}}_1:(1,0,0)$, $\bar{\mathcal{B}}_2:(0,1,0)$ and $\bar{\mathcal{B}}_3:(0,0,1)$.  (b) Configuration of a phase-field vector $(\psi,\varphi)$ representing ternary phase. The bulks of the three phases in DBPF model are characterized by ${\mathcal{B}}_1:(-1,\cdot)$, ${\mathcal{B}}_2:(1,-1)$ and ${\mathcal{B}}_3:(1,1)$. }\label{three phase}
   \end{figure}
   
In \cite{Diffuseinterface,Analysisof}, an alternative dichotomy-based phase-field (DBPF) model for ternary system is proposed, offering both simplicity and effectiveness in capturing the complex interfacial structures between phases. The ternary DBPF model employs a pair of phase-field variables $(\psi, \varphi)$ to identify each phase from  the rest in a nested way (sketched in Fig. \ref{three phase2}). In contrast to the hyperplane constraint imposed in the CBPF model, the DBPF model automatically satisfies the volume conservation. 
Moreover, the energy of the ternary DBPF model is always bounded from below, and the algebraic and dynamic consistencies can be easily preserved in the DBPF model. These strengths potentially expand the scope of the DBPF model's applications in both analytical and numerical domains. 


To validate the phase-field models, sharp-interface analysis is usually conducted to study their consistency with existing sharp-interface models as well as physical requirements.
For binary phase-field models, their sharp interface limits have been well studied using both matched asymptotic analysis and $\Gamma$-convergence analysis \cite{dai2012motion, dai2014coarsening, chen2014analysis, cahn1996cahn, lee2016sharp, colli2017asymptotic, tang2022asymptotic}. For the ternary CBPF model, formal asymptotic analysis was employed to study its sharp-interface limit\cite{Onthreephase, Thedynamics}, in which the Neumann triangle condition at the triple junction was obtained. 
However, there is still a lack of sharp-interface analysis for the DBPF model.

In the aspect of numerical simulations, developing high-order and unconditionally energy-stable schemes for the ternary CBPF model is still challenging because of the strong nonlinearities and couplings in the system. To preserve the mathematical structures of such phase-field systems, tremendous numerical techniques for gradient flow systems have been developed recently \cite{ Unconditionallygradient, zhu1999coarsening, shen2010numerical, shen2018scalar, shen2019new, li2022stability, li2022stabilityand}, some of which are generalized to solve the CBPF model. These include the convex-concave splitting schemes \cite{boyer2011numerical}, IEQ approaches \cite{yang2017numerical}, SAV methods \cite{zhang2020decoupled}. The energy stable methods for simulating multiphase flows using the CBPF model are also widely investigated \cite{yang2021new, Zhang2016Phase, Wu2023Highly, Tan2023An, Yang2021Numerical, Yang2024Efficiently}. In contrast, there are fewer works on the numerical study for the DBPF model, which are either fully implicit schemes \cite{Diffuseinterface} or only first-order schemes \cite{zhao2017decoupled}. Developing high-order linear, decoupled, and energy-stable numerical schemes for the original energy is in high demand for the DBPF models.

This paper aims to generalize the DBPF models for general $N$-phase systems, and investigate their sharp-interface properties and numerical methods. By introducing $N\!-\!1$ completely independent phase-field functions and a degenerate energy, the novel DBPF model is derived systematically using the Onsager principle \cite{Reciprocal1931OnsagerI, Reciprocal1931OnsagerII}, guaranteeing the thermodynamic consistency of the proposed model. Compared with the classical CBPF model, the novel DBPF model shows both analytical and numerical advantages. In particular, it avoids the hyperplane constraint and ensures mechanic, energetic, algebraic, and dynamic consistency conditions without imposing restrictions on physical parameters. Moreover, we develop a framework for recursively constructing the free energies, so that the DBPF approach is easily extended to model systems with arbitrary numbers of components while preserving these four consistency conditions. The novel DBPF model shows good sharp-interface consistency with existing physical requirements such as the Neumann triangle condition. Using the mobility operator splitting (MOS) technique \cite{Lu2025Decoupled}, we easily construct a second-order, linear, and energy stable scheme for the DBPF model which decouples the numerical integration of the $N-1$ phase-field functions. Numerical results demonstrate that the DBPF model succeeds in several benchmark simulations, reproducing the formation of liquid lenses between two stratified fluids and the generation of double emulsions and Janus emulsions. These are in agreement with experimental observations, showing the potential of the DBPF model in capturing complex interfacial phenomena.

The rest of this paper is organized as follows. In Section \ref{Sec:N_phase_dynamics}, we present the new DBPF model for the $N$-phase system by using the Onsager principle. In particular, we propose a framework for constructing the free energy functional of $N$-phase systems satisfying the four consistency conditions. In Section \ref{Sec:Sharp_interface_limit}, using matched asymptotic expansions, we investigate the sharp interface limit of the ternary DBPF model and derive the Neumann triangle law. In Section \ref{Sec:Numerical_schemes}, by using a mobility operator splitting technique, we develop a linear and decoupled scheme for the $N$-phase DBPF model which preserves the dissipation of the original energy. In Section \ref{Sec:Numerical_simulations}, various numerical results are presented to validate the theoretical derivation. In particular, benchmark simulations for multiphase problems are also provided.

\section{Novel $N$-phase Cahn-Hilliard dynamics}\label{Sec:N_phase_dynamics}

\subsection{$N$-phase Cahn-Hilliard dynamics with dichotomic representation}
In a two-phase phase-field model, an order parameter (also named phase-field function) $\phi(\mathbf{x},t): \mathbb{R}^d\times\mathbb{R}\!\rightarrow\!\mathbb{R}$ is introduced to describe the distinct phases, that is,
   \begin{equation*}
   \begin{cases}
   \phi=-1 \quad &\text{in phase 1},\\  
   \phi=1 \quad &\text{in phase 2},
   \end{cases}
   \end{equation*}
and $\phi$ has a smooth transition between $\pm1$ within a thin region across the interface. The mixing energy of the two-phase system is given by the classical Ginzburg-Landau potential
   \begin{align}\label{eq:binary-energy}
   W(\phi)=\gamma\int_{\Omega}\Big(\frac{\varepsilon}{2}|\nabla\phi|^{2}+\frac{1}{4\varepsilon}(\phi^2-1)^2\Big)\mathrm{d}\mathbf{x},
   \end{align}
where $\varepsilon$ is the width of the interface,  $\gamma=\frac{3}{2\sqrt 2}\sigma_{12}$ is the surface tension coefficient, and $\Omega\!\in\!\mathbb{R}^d$. The first term of the mixing energy accounts for the hydrophilic interactions between the materials, while the second term models the hydrophobic interactions. The minimizer of this energy functional admits a (soft) binary partition of $\Omega$. In this work, we focus mainly on the two-/three-dimensional space with $d=2,3$.

By iteratively applying binary partitions to $\Omega$ and its subsets through such phase-field functions, we can obtain a $N$-partition of $\Omega$ with $N\!-\!1$ independent phase-field variables, which we call the dichotomic representation of multiphase systems. Mathematically, we can introduce a phase-field vector ${\mathbf{\Phi}}^{N}=(\phi_1,\ldots,\phi_{N-1})$ for the $N$-phase system, with each $\phi_i$ being the standard two-phase phase-field variable. Let $Q_{m}=[-1,1]^{m}$, $m\geqslant 1$ (with the convention that $Q_0=\emptyset$) be the $m$-cube. Define the projection operator along the $j$-th direction ($1\leqslant j\leqslant N\!-\!1$) in an $(N\!-\!1)$-cube as follows:
\begin{align*}
\Pi_{j,b}^{N-1}:\ & Q_{N-1}\rightarrow~~\quad Q_{j-1}~\times~\{b\}~\times~ Q_{N-1-j},
\\
&\mathbf{\Phi}^{N}\quad\mapsto(\phi_1,\ldots,\phi_{j-1},
b,\phi_{j+1},\ldots,\phi_{N-1}).
\end{align*}
For simplicity, we adopt the shorthand notation $\Pi_{i_1,b}^{N-1}\circ\Pi_{i_2,b}^{N-1}\circ\ldots\circ\Pi_{i_k,b}^{N-1}:=\Pi_{\{i_1,\ldots,i_k\},b}^{N-1}$ for the composition of multiple projections and $\Pi_{0,b}^{N-1}, \Pi_{N,b}^{N-1}=\emptyset$ for trivial projections. As summarized in Table \ref{N-phase}, the bulk of the $i$th-phase ($i=1\leqslant i\leqslant N$) can be mathematically represented as ${\mathbf{\Phi}}^{N}\in{\mathcal{B}}^i$ with
\begin{align}\label{bulk}
   {\mathcal{B}}^i:=
   \Pi_{\{1,\ldots,i-1\},1}^{N-1}\circ\Pi_{i,-1}^{N-1} Q_{N-1}, \quad 1\leqslant i\leqslant N.
\end{align}
Consequently, the interface between the $i$th-phase and the $k$th-phase ($1\leqslant i<k\leqslant N$) can be represented by ${\mathbf{\Phi}}^{N}\in\mathcal{I}^{ik}$ with 
\begin{align}\label{interface}
   {\mathcal{I}}^{ik}:=
   \Pi_{\{1,\ldots,k-1\}\setminus\{i\},1}^{N-1}\circ\Pi_{k,-1}^{N-1} Q_{N-1}, \quad 1\leqslant i<k\leqslant N.
\end{align}
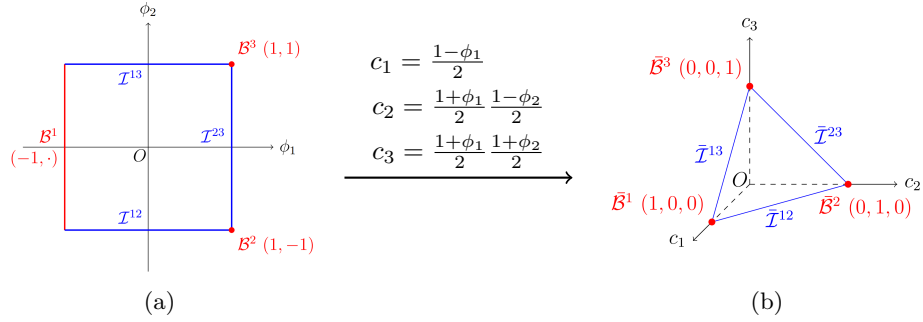
\begin{figure}[htb]
   \begin{subfigure}{0.36\linewidth}
   \centering
   \includegraphics[trim=0cm 0.3cm 0.5cm 0.2cm,clip,width=0.9\linewidth]{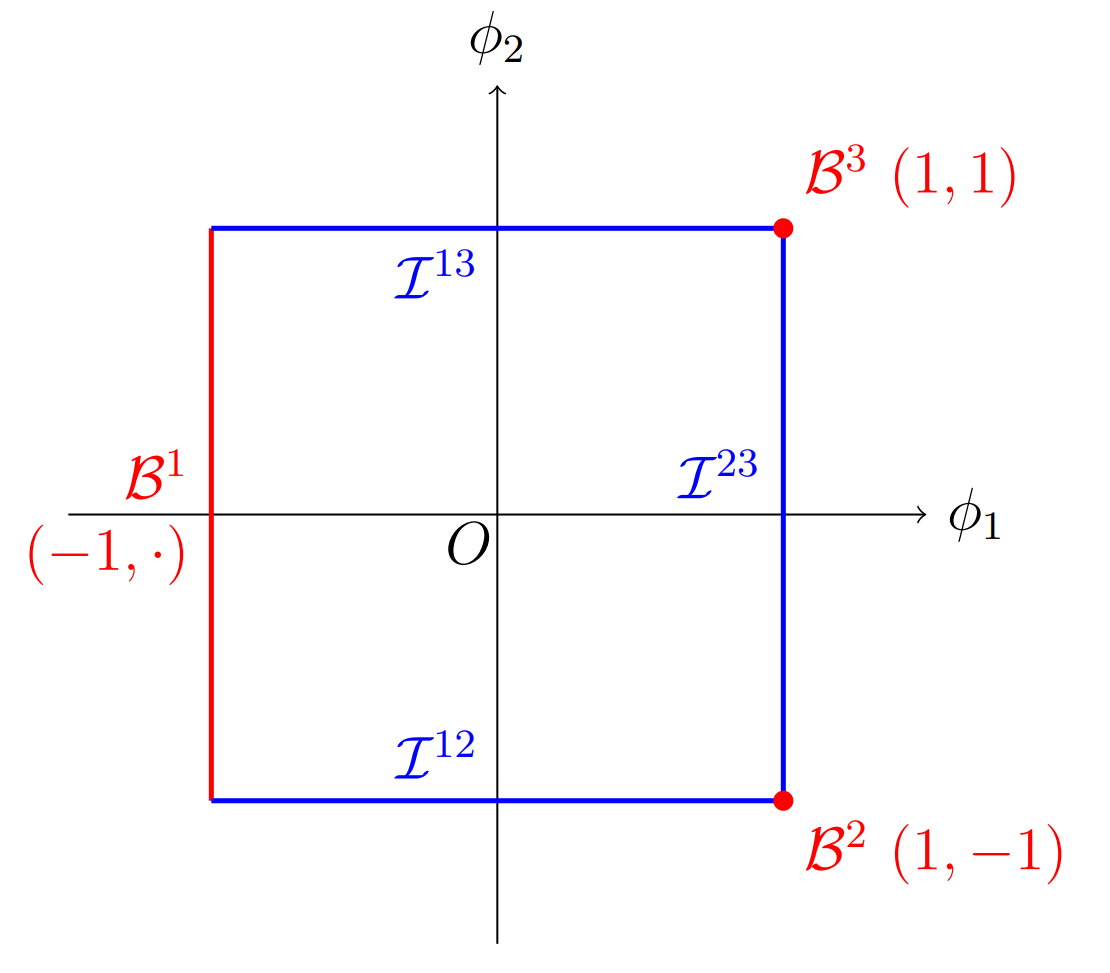}
   \subcaption{}\label{sketch_of_TTPM}
   \end{subfigure}
   \hfill
   \begin{tikzpicture}
   [baseline={(0,-2)}]
   \draw[->,thick] (-1,0) -- (2,0);
   \node[above] at (0.1,1.2) {$c_1=\frac{1-\phi_1}{2}$};
   \node[above] at (0.5,0.6) {$c_2=\frac{1+\phi_1}{2}\frac{1-\phi_2}{2}$};
   \node[above] at (0.5,0) {$c_3=\frac{1+\phi_1}{2}\frac{1+\phi_2}{2}$};  
   \end{tikzpicture}
   \hfill
   \begin{subfigure}{0.38\linewidth}
   \centering
   \includegraphics[trim=0cm 0.3cm 0cm 0cm,clip,width=0.9\linewidth]{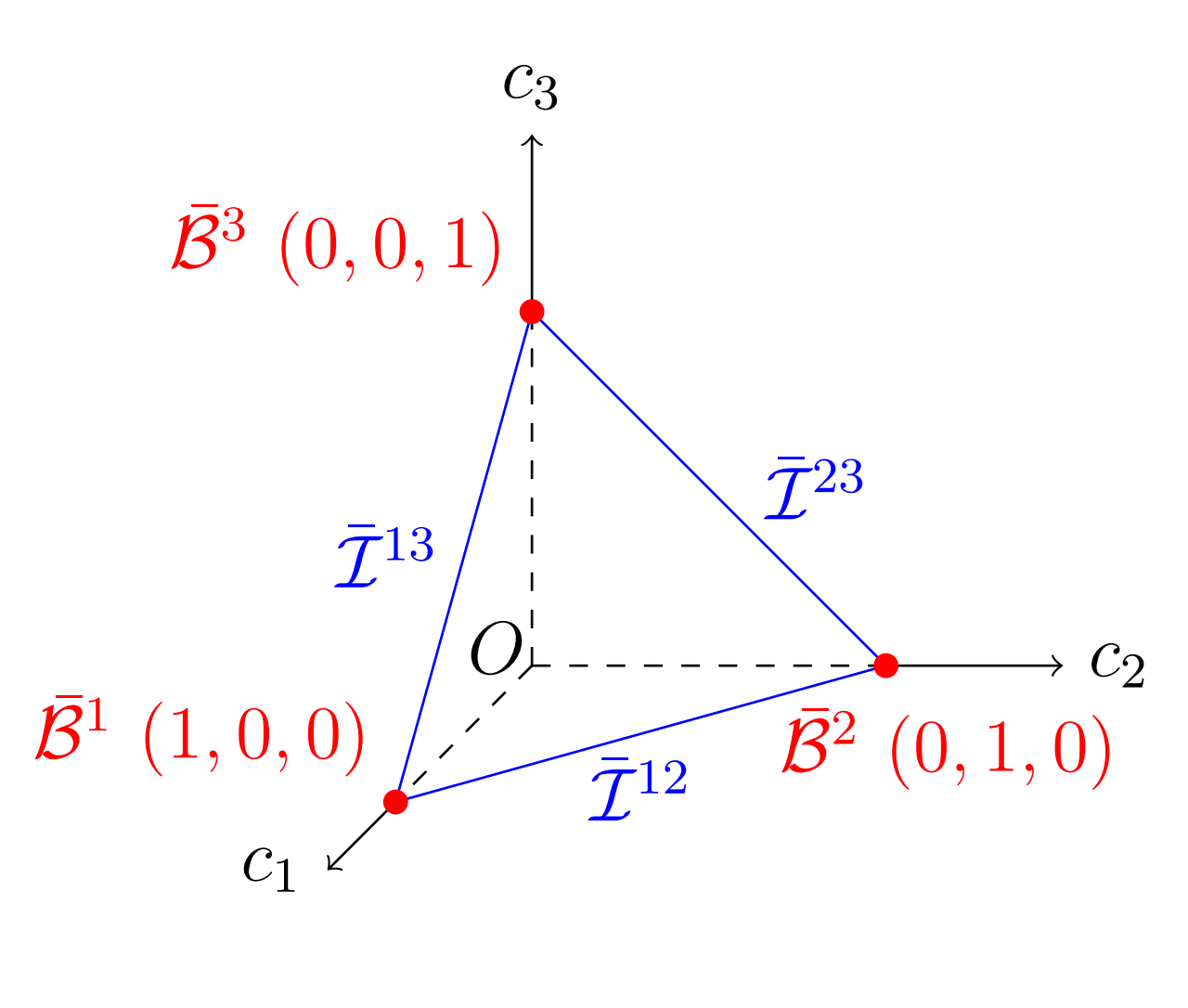}
   \subcaption{}\label{sketch_of_DTPM}
   \end{subfigure}
   \caption{Sketch of dichotomic representation (a) vs. characteristic representation (b) for ternary phases. ${\mathcal{B}}^i$ and $\bar{\mathcal{B}}^i$ $(i=1,2,3)$ represent the bulk phases (the red portion) while ${\mathcal{I}}^{ik}$ and $\bar{\mathcal{I}}^{ik}$ $(1\leqslant i<k\leqslant 3)$ represent the interfaces (the blue portion) in the respective representations. In this case, there is a nonlinear correspondence between the two sets of phase variables $(\phi_1,\phi_2)$and $(c_1,c_2,c_3)$.} \label{sketch_of_TTPM_and_DTPM}
   \end{figure}

\begin{table}[htb]
   \centering
   \caption{The phase-field variables $\phi_1, \phi_2, \cdots, \phi_{N-1}$, describing the N-phase problem.}
   \begin{tabular}{@{}>{\centering\arraybackslash}cccccccc}
   \toprule[1.5pt]
    & $\text{phase}\ 1$ & $\text{phase}\ 2$  & $\text{phase}\ 3$ & $\cdots$ & $\text{phase}\ N\!-\!2$ & $\text{phase}\ N\!-\!1$ & $\text{phase}\ N$\\
   \midrule[0.8pt]
   $\phi_1$ & -1 & 1 &  1 &  $\cdots$ & 1 & 1 & 1\\
   $\phi_2$ &  & -1 & 1 & $\cdots$ & 1 & 1 & 1\\
   $\phi_3$ &  &  & -1 & $\cdots$ & 1 & 1 & 1\\
   $\cdots$ &  &  &  & $\cdots$ & $\cdots$ & $\cdots$ & $\cdots$\\
   $\phi_{N-2}$ &  &  &  &  & -1 & 1 & 1\\
   $\phi_{N-1}$ &  &  &  &  &  & -1 & 1\\
   \bottomrule[1.5pt]
   \end{tabular}
   \label{N-phase}
   \end{table}
   
The two types of representations for ternary phase systems can be graphically
illustrated in Figure \ref{sketch_of_TTPM_and_DTPM}. 
The bulk phases are represented by the red nodes of a simplex ($\bar{\mathcal{B}}^i$, $i=1,2,3$) in CBPF (Figure \ref{sketch_of_TTPM_and_DTPM} (b)), while they are represented by the red nodes ($\mathcal{B}^i$, $i=2,3$) or the red edge ($\mathcal{B}^1$) of a 2-cube in DBPF (Figure \ref{sketch_of_TTPM_and_DTPM} (a)). The two representations are related by a nonlinear transform \cite{Analysisof}. Due to the degeneracy of the nonlinear transform (Figure \ref{sketch_of_TTPM_and_DTPM}), the edge $\mathcal{B}^1$ simply degenerates to the node $\bar{\mathcal{B}}^1$ when the 2-cube  degenerates to a simplex. The blue edges $\mathcal{I}^{ik}$ and $\bar{\mathcal{I}}^{ik} (1\leqslant i<k\leqslant 3)$ connecting the bulk phases in Figures \ref{sketch_of_TTPM_and_DTPM} (a) and (b) consist of all the interfaces. 

In analogy to \eqref{eq:binary-energy}, we could define the mixing free energy in an $N$-phase DBPF model:
   \begin{equation}\label{mixing energy of N-phase model}
   W^N({\mathbf{\Phi}}^{N})=
   \int_{\Omega}\sum_{i=1}^{N-1}\gamma_i^N
   \Big(\frac{\varepsilon_i}{2}|\nabla\phi_i|^{2}+\frac{1}{\varepsilon_i}F(\phi_i) \Big)\mathrm{d}\mathbf{x},
   \end{equation} 
where 
\begin{align}\label{gamma_i^N}
\gamma_i^N=\gamma_i^N({\mathbf{\Phi}}^{N})
:\mathbb{R}^{N-1}\rightarrow\mathbb{R},\ i=1,\ldots,N-1,
\end{align}
is the surface tension function on the corresponding interface, and $F(\phi_i)=\frac{1}{4}(\phi_i^2-1)^2$ is the Ginzburg-Landau potential. \eqref{mixing energy of N-phase model} can be regarded as an interpolation of $N\!-\!1$ binary-phase potentials with some interpolation functions that are well selected in \eqref{gamma_i^N}.

We shall apply the Onsager variational principle to derive the new $N$-phase dynamics induced by the free energy \eqref{mixing energy of N-phase model} and a proper dissipation functional (or equivalently the metric). In this work, we consider conservative dynamics, while nonconservative dynamics can be derived in a similar way. The energy dissipation functional, as half the rate of free energy dissipation, is given by
   \begin{align*}
   \Phi=\sum_{i=1}^{N-1}\int_{\Omega}\frac{|\mathbf{J}_{\phi_i}|^2}{2M_i}\mathrm{d}\mathbf{x},
   \end{align*}
where $M_i=M_i({\mathbf{\Phi}}^{N}) (i=1,2,\cdots,N\!-\!1)$ is the mobility parameter, and $\mathbf{J}_{\phi_i}$ is the flux of $\phi_i$ that satisfies the continuity equation in $\Omega$
   \begin{align}\label{continuity equations}
   \dot{\phi_i}+\nabla\cdot\mathbf{J}_{\phi_i}=0, \quad i=1,2,\ldots,N-1,
   \end{align}
supplemented with boundary conditions
   \begin{align}\label{BC of J}
   \mathbf{J}_{\phi_i}\cdot\mathbf{n}=0,\quad \mbox{on}\;\;\partial\Omega.
   \end{align}
In \eqref{continuity equations} and hereafter, we shall use the shorthand notation $\dot{(\bullet)}=\partial_t(\bullet)$ for the clarity and simplicity of our derivation. 
Now we are ready to derive the governing dynamics. The Onsager principle states that dissipative dynamics can be obtained by minimizing the Rayleighian 
   \begin{align*}
   R=\Phi+\dot{W}^N,
   \end{align*}
with respect to all rate variables
$\{\mathbf{J}_{\phi_1},\cdots,\mathbf{J}_{{\phi}_{N-1}};\dot{\phi}_1,\cdots,\dot{\phi}_{N-1}\}$ under the continuity constraints \eqref{continuity equations}, that is,
   \begin{align}\label{minimizing the Rayleighian}
   \begin{cases}
   \underset{\{\mathbf{J}_{\phi_1},\cdots,\mathbf{J}_{{\phi}_{N-1}};
   \dot{\phi}_1,\cdots,\dot{\phi}_{N-1}\}}{\min}\hspace{1mm}\Phi+\dot{W}^N,\\
   \rm{s.t.}
   \begin{cases}
   \dot{\phi}_i+\nabla\cdot\mathbf{J}_{\phi_i}=0\quad  &\text{in}~\Omega, \ i=1,2,\cdots N-1,\\
   \mathbf{J}_{\phi_i}\cdot\mathbf{n}=0, \quad &\text{on}~\partial\Omega, \ i=1,2,\cdots N-1.
   \end{cases}
   \end{cases}
   \end{align} 
Using integration by parts and the constraints \eqref{continuity equations}--\eqref{BC of J}, we have
   \begin{align*}
   &~\dot{W}^N(\mathbf{J}_{\phi_1},\cdots,\mathbf{J}_{\phi_{N-1}};\dot{\phi}_1,\cdots,\dot{\phi}_{N-1})\\
   =&\int_{\Omega}\sum_{i=1}^{N-1}\bigg(\gamma_i^N\Big({\varepsilon_i}\nabla\phi_i\cdot\nabla\dot{\phi}_i+\frac{1}{\varepsilon_i}f(\phi_i)\dot{\phi}_i
   \Big)+\sum_{j=1}^{N-1}\frac{\partial \gamma_i^N}{\partial \phi_j}\dot{\phi}_j\Big(\frac{\varepsilon_i}{2}|\nabla\phi_i|^2+\frac{1}{\varepsilon_i}F(\phi_i) \Big)\bigg)\mathrm{d}\mathbf{x}\\
   =&\int_{\Omega}\sum_{i=1}^{N-1}\bigg(-\varepsilon_i\nabla\cdot\big(\gamma_i^N\nabla\phi_i\big)+\frac{1}{\varepsilon_i}\gamma_i^N f(\phi_i)+\displaystyle\sum_{j=1}^{N-1}\frac{\partial \gamma_j^N}{\partial \phi_i}\Big(\frac{\varepsilon_j}{2}|\nabla\phi_j|^2+\frac{1}{\varepsilon_j}F(\phi_j) \Big)\bigg)\dot{\phi}_i\mathrm{d}\mathbf{x}\\
   &+\int_{\partial\Omega}\sum_{i=1}^{N-1}\gamma_i^N{\varepsilon_i}(\partial_n\phi_i)\dot{\phi}_i\mathrm{d}S\\
   =&\int_{\Omega}\sum_{i=1}^{N-1}\nabla\mu_{\phi_i}\cdot\mathbf{J}_{\phi_i}
   \mathrm{d}\mathbf{x}+\int_{\partial\Omega}\sum_{i=1}^{N-1}\gamma_i^N{\varepsilon_i}(\partial_n\phi_i)\dot{\phi}_i\mathrm{d}S,
   \end{align*}
where $f(\phi_i)=F'(\phi_i)$ and
   \begin{align*}
   \mu_{\phi_i}=\frac{\delta W^N}{\delta \phi_i}=-\varepsilon_i\nabla\cdot\big(\gamma_i^N\nabla\phi_i\big)+\frac{1}{\varepsilon_i}\gamma_i^Nf(\phi_i)+\displaystyle\sum_{j=1}^{N-1}\frac{\partial \gamma_j^N}{\partial \phi_i}\Big(\frac{\varepsilon_j}{2}|\nabla\phi_j|^2+\frac{1}{\varepsilon_j}F(\phi_j) \Big).
   \end{align*}
 Then the Euler-Lagrange equation of \eqref{minimizing the Rayleighian} with respect to $\mathbf{J}_{\phi_i}$ leads to the following constitutive relations
   \begin{align*}
   \mathbf{J}_{\phi_i}=-M_i\nabla\mu_{\phi_i}\quad i=1,2,\cdots,N-1.
   \end{align*}
Putting these constitutive relations into the continuity equations \eqref{continuity equations}, we arrive at the $N$-phase Cahn-Hilliard type dynamics:
   \begin{align}\label{N-phase Cahn-Hilliard system}
   \begin{cases}
   \dfrac{\partial \phi_i}{\partial t}=
   \nabla\cdot\big(M_i\nabla\mu_{\phi_i}\big),\\
   \mu_{\phi_i}=-\varepsilon_i\nabla\cdot\big(\gamma_i^N\nabla\phi_i\big)+\frac{1}{\varepsilon_i}\gamma_i^Nf(\phi_i)+\displaystyle\sum_{j=1}^{N-1}\frac{\partial \gamma_j^N}{\partial \phi_i}\Big(\frac{\varepsilon_j}{2}|\nabla\phi_j|^2+\frac{1}{\varepsilon_j}F(\phi_j) \Big).
   \end{cases}
   \end{align}
These equations are associated with the boundary conditions 
   \begin{align}\label{BC1}
   \partial_{n}\phi_i=0,\qquad \partial_{n}\mu_{\phi_i}=0 \qquad \text{on}\ \partial\Omega, 
   \end{align}
which can be obtained by considering the optimality condition of \eqref{minimizing the Rayleighian} with respect to $\dot{\phi}_i$ on $\partial\Omega$ and placing the constitutive relation of $\mathbf{J}_{\phi_i}$ into \eqref{BC of J}, respectively.
As a direct consequence of our derivation, the following energy law can be established for the system \eqref{N-phase Cahn-Hilliard system}:
   \begin{equation*}
   \frac{\mathrm{d}}{\mathrm{d} t}W^N({\mathbf{\Phi}}^{N})
   =-\int_{\Omega}\sum_{i=1}^{N-1}M_i|\nabla\mu_{\phi_i}|^2\mathrm{d}\mathbf{x}\leqslant0.
   \end{equation*}

\begin{remark}\label{choose mobility}
   In numerical practice, it is found that the choice of constant mobilities for all $M_i$ may lead to loss of volume conservation for some phase. This is attributed to the unexpected interface motion of some $\phi_i$ in the bulk phase with $\phi_j=-1 (j<i)$ and the deviation of $\phi_i$ from the $\tanh$ profile. This loss of volume conservation can be alleviated when the thickness parameter $\varepsilon$ decreases, as can be illustrated by the asymptotic analysis. Nevertheless, in this work we adopt a choice of degenerate mobilities to restrict such unexpected interface motions in numerical simulations:
   \begin{align}\label{mobility}
   M_i=
   \begin{cases}
    m_1, & i=1,\\
    m_i\displaystyle\prod_{j=1}^{i-1}
    \Big(\frac{1+\phi_j}{2}\Big)^{2a_j}, & i=2,3,\ldots,N-1, 
   \end{cases}
   \end{align}
where $m_i$, $i=1,2,\ldots,N\!-\!1$, is a positive constant and $a_j$, $j=1,2,\ldots,N\!-\!2$, is a positive integer. As will be reflected from our numerical results, these degenerate mobilities effectively preserve volume conservations.
\end{remark}


\subsection{Mechanic, energetic, algebraic and dynamic consistency}
For the $N$-phase model \eqref{N-phase Cahn-Hilliard system} to correctly capture the multiphase dynamics, it is crucial that its free energy functional should satisfy the following consistency properties: namely, the mechanic, the energetic, the algebraic and the dynamic consistencies \cite{boyer2006study, Analysisof}.  It should be noted that an incorrect choice of the surface tension coefficient $\gamma_i^N$ will lead to an unphysical dynamics of the system. For simplicity, we can assume that $\gamma_i^N$ is independent of $\phi_i$.

\vspace{2mm}
\begin{enumerate}
\item\label{Mechanic consistency}
  $\textbf{Mechanic consistency}$: The construction of $\gamma_i^N$ should recover the surface tension coefficient $\sigma_{ij}$ upon integration of the free energy density across the corresponding interface. Moreover, $\gamma_i^N$ should vanish in the bulk of each phase $i$. 
 In summary,
\begin{align}\label{condition i1}
   \gamma_i^N({\mathbf{\Phi}}^{N})
   =
   \begin{cases}
   \frac{3}{2\sqrt 2}\sigma_{ik} ,\quad &\text{if}\ {\mathbf{\Phi}}^{N}\in\mathcal{I}^{ik},\ 1\leqslant i<k\leqslant N,\\
   0, \quad
   &\text{if}\ {\mathbf{\Phi}}^{N}\in\displaystyle\bigcup_{n=1}^{i-1}\mathcal{B}^n,\    2\leqslant i\leqslant N-1,
   \end{cases}
\end{align}
where $\mathcal{I}^{ik}$ and $\mathcal{B}^n$  are defined in \eqref{bulk} and \eqref{interface}. 

\vspace{2mm}
\item\label{Energetic consistency}
\textbf{Energetic consistency}: The energy of the $N$-phase model should exactly coincide with the corresponding $(N\!-\!1)$-phase model when some phase is absent. Denote 
the $(N\!-\!1)$-phase vector in the absence of the $j$th-phase by $\mathbf{\Phi}^{N}_{(-j)}=(\phi_1,\ldots,\phi_{j-1},
\phi_{j+1},\ldots,\phi_{N-1})$, and the parameter sets of surface tension coefficients by
$\mathbf{\Sigma}^{N,i}=(\sigma_{i,i+1},\ldots\sigma_{i,N})$,
$\mathbf{\Sigma}_{(-j)}^{N,i}=(\sigma_{i,i+1},\ldots,\sigma_{i,j-1},
\sigma_{i,j+1},\ldots\sigma_{iN})$, $i<j$,  and $\mathbf{\Sigma}_{(-j)}^{N,i}=\mathbf{\Sigma}^{N,i}$, $i>j$. For convenience, we introduce $\phi_0, \phi_N, \sigma_{i0}=\emptyset$.  Then the energetic consistency can be formulated as
\begin{align}\label{condition iii 1}
   \begin{cases}
   W^N|_{{\mathbf{\Phi}}^{N}\in\Pi_{j,1}^{N-1}Q_{N-1}}
   =W^{N-1}(\mathbf{\Phi}_{(-j)}^{N};\bigcup_{i\neq j}\mathbf{\Sigma}_{(-j)}^{N,i}),\ 1\leqslant j\leqslant N-1,\\
   W^N|_{{\mathbf{\Phi}}^{N}\in\Pi_{N-1,-1}^{N-1}Q_{N-1}}=W^{N-1}
   (\mathbf{\Phi}^{N-1};\bigcup_{i<N-1}\mathbf{\Sigma}^{N-1,i}),
   \end{cases}
\end{align}
where ${\mathbf{\Phi}}^{N}\in\Pi_{j,1}^{N-1}Q_{N-1}$ and ${\mathbf{\Phi}}^{N}\in\Pi_{N-1,-1}^{N-1}Q_{N-1}$ indicate the absences of the $j$th-phase $(1\leqslant j\leqslant N\!-\!1)$ and the $N$th-phase respectively, and the parameter sets are explicitly specified to emphasize the dependence of the energy densities on the corresponding surface tension coefficients.
Hereafter, we introduce the projection on the specific coordinate $(j, b)\in I_a=\{(1,1),\ldots,(N\!-\!1,1)\}\cup\{(N\!-\!1,-1)\}$ to represent the absence of some phase.

\vspace{2mm}
\item\label{Algebraic consistency}
\textbf{Algebraic consistency}: If $\phi_j(x,y,0)\equiv1$, $j=1,\ldots,N\!-\!1$ or $\phi_{N-1}(x,y,0)\equiv-1$,  then we have $\phi_j(x,y,t)\equiv1$ or $\phi_{N-1}(x,y,t)\equiv-1$, $\forall t\geqslant0$, respectively. This can be achieved by setting $\mu_{\phi_j}|_{{\mathbf{\Phi}}^{N}\in\Pi_{j,b}^{N-1}Q_{N-1}}=0$, $(j,b)\in I_a$. By using the second equation of \eqref{N-phase Cahn-Hilliard system}, we have
\begin{align}\label{condition iii 2}
   \frac{\partial \gamma_i^N}{\partial \phi_j}\big|_{{\mathbf{\Phi}}^{N}\in\Pi_{j,b}^{N-1}Q_{N-1}}=0, \quad (j,b)\in I_a, \ 1\leqslant i\leqslant N-1, \ i\neq j,
\end{align}
This means that if we fix $\{\phi_1,\ldots,\phi_{N-1}\}\setminus \{\phi_j\}$, then $\phi_j=1$, $1\leqslant j\leqslant N-2$, and $\phi_j=\pm 1$, $j=N-1$ are the critical points of $\gamma_i^N$, $i=1,\ldots,N\!-\!1$.

\vspace{2mm}
\item\label{Dynamic consistency}
\textbf{Dynamic consistency}: The $N$-phase model is dynamically consistent in the sense that the solutions $ {\mathbf{\Phi}}^{N}\in\Pi_{j,1}^{N-1}Q_{N-1}$, $j=1,\ldots,N\!-\!1$ and $ {\mathbf{\Phi}}^{N}\in\Pi_{N-1,-1}^{N-1}Q_{N-1}$ of the $N$-phase model are stable under small perturbations of the initial data.  To ensure the dynamical consistency, the critical points in \eqref{condition iii 2} need to be minima, i.e.,
\begin{align}\label{condition iv}
   \frac{\partial^2 \gamma_i^N}{\partial \phi_j^2}\big|_{{\mathbf{\Phi}}^{N}\in\Pi_{j,b}^{N-1}Q_{N-1}}>0,\quad (j,b)\in I_a, \ 1\leqslant i\leqslant N-1, \ i\neq j.
\end{align}
\end{enumerate}

\vspace{2mm}
Now, we provide an inductive approach to construct $\gamma_i^N$ that satisfies the consistency conditions \eqref{Mechanic consistency}-\eqref{Dynamic consistency}. The energetic consistency \eqref{condition iii 1} gives a relation between $\{\gamma_i^N\}$ and $\{\gamma_i^{N-1}\}$, which serves as the basis for the inductive inference from $\{\gamma_i^{N-1}\}$ to $\{\gamma_i^N\}$. To facilitate this construction, we need the following lemma for the construction of a $(N\!-\!1)$-dimensional smooth function $f(\mathbf{Z})$ from its boundary projections, where $\mathbf{Z}=(z_1,\ldots,z_{N-1})\in Q_{N-1}$ is a vector-valued variable. We will still adopt the notation $\mathbf{Z}_{(-i)}$ for the set of $(N-2)$-dimensional variables without the $i$th variable $z_i$.
\begin{lemma}\label{construction of f}
Assume that $h_i(\mathbf{Z}_{(-i)})\;(i=1,\ldots,N\!-\!1)$ and $h_N(\mathbf{Z}_{(-(N-1))})$ are $(N-2)$-dimensional smooth functions satisfying the following conditions:
\begin{enumerate}
\item
Continuity condition:
\begin{align}\label{continuity condition}
h_i|_{z_j=1}=h_j|_{z_i=1},\quad h_i|_{z_{N-1}=-1}=h_N|_{z_i=1}
\end{align}
with $1\leqslant i\neq j\leqslant N\!-\!1$.  
 
\item First-order derivative condition:
\begin{align}\label{1st derivative condition of h_i}
\frac{\partial h_i}{\partial z_j}\Big|_{z_j=1}=\frac{\partial h_i}{\partial z_{N-1}}\Big|_{z_{N-1}=-1}=\frac{\partial h_N}{\partial z_i}\Big|_{z_i=1}=0,
\end{align}
with $1\leqslant i\neq j\leqslant N\!-\!1$. 
\end{enumerate}
Then there exists a smooth function $f$ satisfying the boundary constraints \begin{align}\label{boundary of f}
   \begin{cases}
   f(\Pi_{i,1}^{N-1}\mathbf{Z})=h_i,\ 1\leqslant i\leqslant N-1,
   \\
   f(\Pi_{N-1,-1}^{N-1}\mathbf{Z})=h_N,
   \end{cases}
\end{align}
and the first-order derivative conditions
\begin{align}\label{1st_derivative}
   \frac{\partial}{\partial z_j}f(\mathbf{Z})\Big|_{\mathbf{Z}\in\Pi_{j,b}^{N-1}Q_{N-1}}=0,\qquad (j,b)\in I_a.
\end{align}
 
\end{lemma}

\begin{proof}
    The detailed proof is presented in \ref{proof of Lemma}.
\end{proof}

This lemma basically states that the successful reconstruction of a smooth function from its boundary projections relies on the continuities of these projections at their intersections, which may occur at points, lines, surfaces, etc. A direct consequence of this lemma is the construction of a parameter-dependent smooth function from its boundary projections with at most two distinct functions.

\begin{corollary}\label{corollary 1}
Let $\mathbf{X}=(z_1,\ldots,z_{N_x})$,  $\mathbf{Y}=(z_{N_x+1},\ldots,z_{N_x+N_y})$ and $\mathbf{Z}=(\mathbf{X},\mathbf{Y})\in Q_{N_x}\times
Q_{N_y}$ with $N_x, N_y\geqslant0$ and $N_x+N_y=N\!-\!1$. Assume that $u(\mathbf{X} _{(-i)},\mathbf{Y};\mathbf{\Theta}) \ (1\leqslant i\leqslant N_x)$,  $v(\mathbf{X},\mathbf{Y}_{(-j)};\mathbf{\Theta}_{(-j)})\ (1\leqslant j\leqslant N_y)$ and $v(\mathbf{X},\mathbf{Y}_{(-N_y)};\mathbf{\Theta}_{(-(N_y+1))})$
are $(N-2)$-dimensional smooth functions satisfying the similar continuity condition \eqref{continuity condition} and first-order derivative condition \eqref{1st derivative condition of h_i}, where $\mathbf{\Theta}\in\mathbf{R}^{N_y+1}$  is the set of parameters. Then there exists a smooth function $f$ satisfying the boundary constraints
\begin{align}\label{boundary of f uv}
   \begin{cases}
   f(\Pi_{i,1}^{N_x}\mathbf{X},\mathbf{Y};\mathbf{\Theta})=
   u(\mathbf{X}_{(-i)},\mathbf{Y};\mathbf{\Theta}),\ 1\leqslant i\leqslant N_x,
   \\
   f(\mathbf{X} ,\Pi_{j,1}^{N_y}\mathbf{Y};\mathbf{\Theta})=
   v(\mathbf{X},\mathbf{Y}_{(-j)};\mathbf{\Theta}_{(-j)}),\ 1\leqslant j\leqslant N_y,
   \\
   f(\mathbf{X} ,\Pi_{N_y,-1}^{N_y}\mathbf{Y};\mathbf{\Theta})=
   v(\mathbf{X},\mathbf{Y}_{(-N_y)};\mathbf{\Theta}_{(-(N_y+1))}),
   \end{cases}
\end{align}
and the first-order derivative conditions
\begin{align}\label{1st derivative of f}
   \frac{\partial}{\partial z_j}f(\mathbf{Z};\mathbf{\Theta})\Big|_{\mathbf{Z}\in\Pi_{j,b}^{N-1}Q_{N-1}}=0,
\end{align} 
with $(j,b)\in I_a$. 
\end{corollary}



With the help of Corollary \ref{corollary 1}, we can inductively construct $\gamma_i^N$ such that the $N$-phase model \eqref{N-phase Cahn-Hilliard system} satisfies energetic consistency \eqref{Energetic consistency} and algebraic consistency \eqref{Algebraic consistency}. Specifically, the inductive relation between $\{\gamma_i^N\}$ and $\{\gamma_i^{N-1}\}$ when the $j$th phase is absent is summarized as follows: 
\begin{center}
\begin{tikzpicture}
    \node at (-2,1.25) {$N$-phase model:}; 
    \node (A) at (0,1.3) {$\gamma_1^N$};
    \node (dots1) at (1,1.3) {$\ldots$};
    \node (B) at (2,1.3) {$\gamma_{j-1}^N$};
    \node (C) at (3,1.3) {$\gamma_j^N$};
    \node (D) at (4,1.3) {$\gamma_{j+1}^N$};
    \node (dots2) at (5,1.3) {$\ldots$};
    \node (E) at (6,1.3) {$\gamma_{N-1}^N$};
    
    \node at (-2.4,0) {$(N\!-\!1)$-phase model:}; 
    \node (F) at (0,0) {$\gamma_1^{N-1}$};
    \node (dots3) at (1,0) {$\ldots$};
    \node (G) at (2,0) {$\gamma_{j-1}^{N-1}$};
    \node (H) at (3,0) {$\gamma_j^{N-1}$};
    \node (dots4) at (4,0) {$\ldots$};
    \node (I) at (5,0) {$\gamma_{N-2}^{N-1}$};
    
    \draw[->] (A) -- (F);
    \draw[->] (B) -- (G);
    \draw[->] (D) -- (H);
    \draw[->] (E) -- (I);
\end{tikzpicture}
\end{center}
Mathematically, this relation is expressed as
\begin{align*}
   \gamma_i^N(\mathbf{\Phi}^{N}\in\Pi_{j,1}^{N-1}Q_{N-1};
   \mathbf{\Sigma}^{N,i})
   =
   \begin{cases}
   \gamma_{i-1}^{N-1}(\mathbf{\Phi}_{(-j)}^{N};
   \mathbf{\Sigma}^{N,i}), &1\leqslant j<i\leqslant N\!-\!1,\\
   \gamma_{i}^{N-1}(\mathbf{\Phi}_{(-j)}^{N};
   \mathbf{\Sigma}_{(-j)}^{N,i}), &1\leqslant i<j\leqslant N\!-\!1, 
   \end{cases}
\end{align*}
and 
   \begin{align*}
   \gamma_i^N(\mathbf{\Phi}^{N}\in\Pi_{N-1,-1}^{N-1}Q_{N-1};
   \mathbf{\Sigma}^{N,i})
   =&\gamma_i^{N-1}(\mathbf{\Phi}^{N-1}
   ;\mathbf{\Sigma}^{N-1,i}), \quad 1\leqslant i\leqslant N-2. 
   \end{align*}
By Corollary \ref{corollary 1}, we can construct $\gamma_i^N$ from $\gamma_{i-1}^{N-1}$ and $\gamma_i^{N-1}$ (with the possible choice of $N_x=0$ or $N_y=0$). 
The details 
are provided in \ref{Construction of gamma_i^N}.

\begin{example}
For the ternary-phase model with $N=3$, we assume that $\gamma_i^3 (i=1,2)$ is independent of $\phi_i$.  Let $h_2^1=\frac{3}{2\sqrt2}\sigma_{13}$, $h_3^1=\frac{3}{2\sqrt2}\sigma_{12}$ and $h_1^2=\frac{3}{2\sqrt2}\sigma_{23}$. Then, from \eqref{formula of gamma_i^N} in \ref{Construction of gamma_i^N}, we have 
\begin{align*}
   \gamma_1^3(\phi_2)=
   &~a_2h_2^1+a_3h_3^1
   +c_1^3(\phi_2)
   \\
   =&\frac{3}{2\sqrt{2}}\bigg( 
   \sigma_{13}\Big(\frac{1+\phi_2}{2}\Big)^2(2-\phi_2)
   +\sigma_{12}\Big(\frac{1-\phi_2}{2}\Big)^2(2+\phi_2)\bigg)
   +\Lambda(1-\phi_2^2)^2.
   \end{align*}
To determine $\Lambda$, \eqref{condition of c} is used to give a constraint
\begin{align}\label{the condition of d_1^3}
   \frac{3}{2\sqrt2}\big(\pm\frac{3}{2}
   (\sigma_{12}-\sigma_{13})\big)+8\Lambda>0.
\end{align}
The choice of  
\begin{align}
   \Lambda=
   \frac{3}{2\sqrt2}\Big(\frac{\alpha}{16}
   |\sigma_{12}-\sigma_{13}|\Big), \quad \alpha=3+\varepsilon, 
\end{align}
leads to the ternary-phase model in \cite{Analysisof}. Specifically,
\begin{align}\label{formula of gamma_1^3}
   \gamma_1^3(\phi_2)
   =&\frac{3}{2\sqrt{2}}\bigg( \sigma_{12}\Big(\frac{1-\phi_2}{2}\Big)^2(2+\phi_2)+
   \sigma_{13}\Big(\frac{1+\phi_2}{2}\Big)^2(2-\phi_2)\bigg)
   \notag\\
   &+\frac{3}{2\sqrt{2}}\Big(\frac{\alpha}{16}|\sigma_{12}-\sigma_{13}|(1-\phi_2^2)^2\Big),
\end{align}
and 
\begin{align}\label{formula of gamma_2^3}
   \gamma_2^3(\phi_1)
   =&~\frac{3}{2\sqrt{2}}\sigma_{23} \Big(\frac{1+\phi_1}{2}\Big)^2(2-\phi_1)
   +\frac{3}{2\sqrt{2}}\Big(\frac{\alpha}{16}\sigma_{23}(1-\phi_1^2)^2\Big).
\end{align}
\end{example}

\section{Sharp interface limit}\label{Sec:Sharp_interface_limit}
The formation of junctions at the intersections of multiple phases, such as moving contact lines, makes the $N$-phase problems much more complicated than the two-phase problems, which involve only interfaces. Among those junctions, 
triple junctions are most commonly observed in $N$-phase patterns and are usually stable. 
In this section, we will analyze the sharp interface limit of the ternary DBPF model, especially the behavior near the triple junction. Using the matched asymptotic analysis, the Neumann
triangle condition at the triple junction.

\subsection{Ternary DBPF model}
For simplicity and clarity of presentation, we omit the superscript 3 in $W^3$, $\gamma_1^3$, and $\gamma_2^3$, and introduce $\psi$ and $\varphi$ instead of $\phi_1$ and $\phi_2$, respectively, for the phase variables, i.e. 
    \begin{equation*}
   \begin{cases}
   \psi=-1,
   &\text{in phase 1},\\
   \psi=1,\varphi=-1 &\text{in phase 2},\\  
   \psi=1,\varphi=1 &\text{in phase 3},
   \end{cases}
   \end{equation*}
Without loss of generality, we let $\varepsilon_1=\varepsilon_2=\varepsilon$. Then, the mixing free energy \eqref{mixing energy of N-phase model} reduces to
   \begin{equation}\label{eq:ternary_energy}
   W(\psi,\varphi)=
   \int_{\Omega}\gamma_1(\varphi)\Big(\frac{\varepsilon}{2}|\nabla\psi|^{2}+\frac{1}{\varepsilon}F(\psi) \Big)+\gamma_2(\psi)\Big(\frac{\varepsilon}{2}|\nabla\varphi|^{2}+\frac{1}{\varepsilon}F(\varphi) \Big)\mathrm{d}\mathbf{x}.
   \end{equation}
where $\gamma_1$ and $\gamma_2$ are defined in \eqref{formula of gamma_1^3} and \eqref{formula of gamma_2^3}, respectively, and $F(\bullet)=\frac{1}{4}(\bullet^2-1)^2$. Using \eqref{N-phase Cahn-Hilliard system} and \eqref{BC1}, the ternary DBPF model reads
\begin{align}
   \psi_{t}&=\nabla\cdot(M_1\nabla\mu_{\psi}),\label{CH11}\\
   \mu_{\psi}&=-\varepsilon\nabla\cdot\big(\gamma_1(\varphi)\nabla\psi\big)+\frac{1}{\varepsilon}\gamma_1(\varphi)f(\psi)+\gamma_2'(\psi)\Big(\frac{\varepsilon}{2}|\nabla\varphi|^2+\frac{1}{\varepsilon}F(\varphi) \Big),\label{CH22}\\
   \varphi_{t}&=\nabla\cdot\big(M_2(\psi)\nabla\mu_{\varphi}\big),\label{CH33}\\
   \mu_{\varphi}&=-\varepsilon\nabla\cdot\big(\gamma_2(\psi)\nabla\varphi\big)+\frac{1}{\varepsilon}\gamma_2(\psi)f(\varphi)+\gamma_1'(\varphi)\Big(\frac{\varepsilon}{2}|\nabla\psi|^2+\frac{1}{\varepsilon}F(\psi) \Big),\label{CH44}
   \end{align}
where $M_1$ and $M_2(\psi)$ are defined in \eqref{mobility}, and $f=F'$.
The boundary conditions are given by
   \begin{align}
   \label{ternary_BC}
   \partial_{n}\mu_{\psi}=0, \qquad 
   \partial_{n}\varphi=0,\qquad \partial_{n}\mu_{\varphi}=0,\qquad\partial_{n}\psi=0   \qquad \text{on}\ \partial\Omega.
   \end{align}
The system \eqref{CH11}--\eqref{ternary_BC}, yields the following energy dissipation law:
   \begin{equation*}
   \frac{\mathrm{d}}{\mathrm{d} t}W(\psi,\varphi)=-\int_{\Omega}M_1|\nabla\mu_\psi|^2\mathrm{d}\mathbf{x}-\int_{\Omega}M_2(\psi)|\nabla\mu_\varphi|^2\mathrm{d}\mathbf{x}\leqslant0.
   \end{equation*}
       
   
   
To perform asymptotic expansions, we separated the domain into the outer region (away from $\Gamma_i$, $i=1, 2, 3$, shown by the black line in Fig. \ref{three region}), the inner region (around the interface but away from the triple junction) and the triple junction region. Suppose that the three-phase interfaces are defined through the zero level sets of $\varphi$ and $\psi$: $\Gamma_1=\{\mathbf{x}\in\Omega:\varphi=0,~ \psi>0\}$, $\Gamma_2=\{\mathbf{x}\in\Omega:\varphi>0,~ \psi=0\}$, and $\Gamma_3=\{\mathbf{x}\in\Omega:\varphi<0,~ \psi=0\}$, where $\Gamma_1$, $\Gamma_2$ and $\Gamma_3$ are smooth closed surfaces and divide the domain $\Omega$ into three regions: $\Omega_1=\{\mathbf{x}\in\Omega: \psi<0\}$, $\Omega_2=\{\mathbf{x}\in\Omega: \varphi<0,\psi>0\}$ and $\Omega_3=\{\mathbf{x}\in\Omega: \varphi>0,\psi>0\}$. 
At this stage, we perform a two-step matching procedure: First, match the asymptotic solutions between the inner region and outer region; then, match the solutions between the inner region and triple junction region. 

   \begin{figure}[htb]
   \centering
   \includegraphics[trim=0cm 0.1cm 0cm 0.5cm,clip,width=8.6cm]{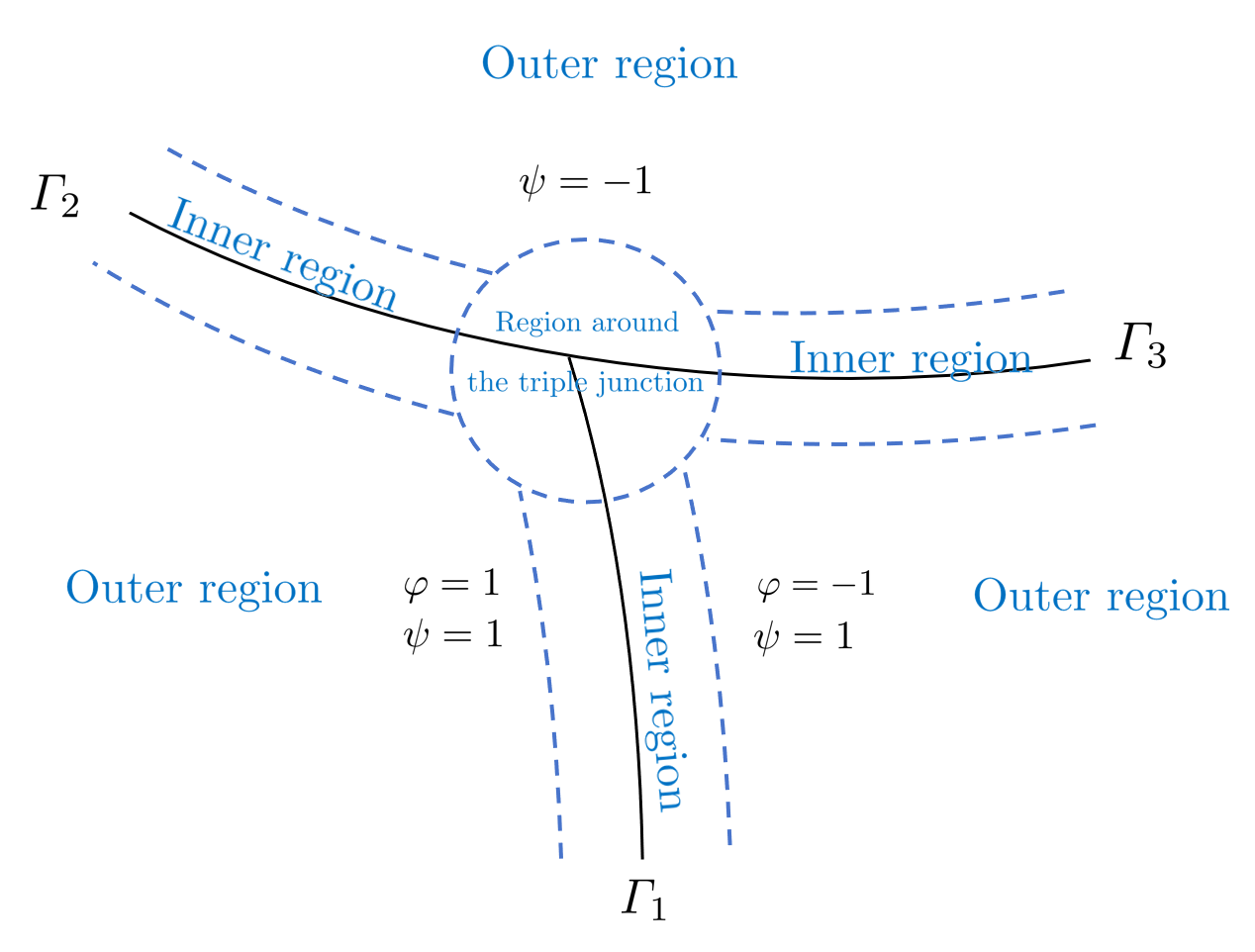}
   \caption{The sketch of the outer region, inner region and  the triple junction region.}\label{three region}
   \end{figure}

\subsection{Analysis in the outer and inner regions}
\vspace{4mm}
\noindent\textbf{Outer variables.}
Away from the interfaces, we consider the outer functions $\psi$, $\mu_{\psi}$, $\varphi$, and $\mu_{\varphi}$ associated with the outer variable $\mathbf{x}$.  The outer expansion for any given outer function $w(\mathbf{x},t)$ (e.g. $\psi$, $\mu_{\psi}$, $\varphi$, $\mu_{\varphi}$) is given as follows
   \begin{align*}
   w(\mathbf{x},t)&=w_{0}(\mathbf{x},t)+\varepsilon w_{1}(\mathbf{x},t)+{\varepsilon^2}w_{2}(\mathbf{x},t)+\cdots.
   \end{align*}
   
\vspace{4mm}
\noindent\textbf{Inner variables.}
For the region around $\Gamma_i$ ($i=1,2,3$) but away from the triple junction, we first introduce an inner variable $z$ as
   \begin{align*}
   z=\frac{d(\mathbf{x},t)}{\varepsilon},
   \end{align*}
where $d(\mathbf{x},t)$ is the signed distance function of the point $\mathbf{x}$ to $\Gamma_i$ 
and $|\nabla{d}|=1$. For any function $w(\mathbf{x},t)$ (e.g.  $\psi$, $\mu_{\psi}$, $\varphi$, $\mu_{\varphi}$), we can rewrite it in terms of inner variable:
   \begin{align*}
    w(\mathbf{x},t)={\tilde w}\Big(\frac{d(\mathbf{x},t)}{\varepsilon},\mathbf{x},t\Big),
   \end{align*}
Following \cite{Frontmigration}, it will be convenient to require that the quantity $\tilde{w}$ does not vary when $\mathbf{x}$ varies in the normal direction to $\Gamma_i$ but $z$ holds, that is, $\tilde{w}(z,\mathbf{x}+\alpha\nabla d,t)=\tilde{w}(z,\mathbf{x},t)$ for small $\alpha$ or $\nabla d(\mathbf{x},t)\cdot\nabla_{\mathbf{x}}{\tilde{w}}=0$. Define $\mathbf{m}=\nabla d(\mathbf{x},t)$ and $\kappa=\Delta d(\mathbf{x},t)=\nabla\cdot\mathbf{m}$. Then, we have 
   \begin{align}\label{inner_derivatives}
   \begin{cases}
   \nabla w=\nabla_{\mathbf{x}}\tilde{w}+\varepsilon^{-1}\mathbf{m}\partial_{z}\tilde{w},\\
   \Delta w=\Delta_\mathbf{x}{\tilde{w}}+\varepsilon^{-1}\kappa\partial_{z}\tilde{w}+\varepsilon^{-2}\partial_{zz}\tilde{w},\\
   \partial_{t}w=\partial_{t}\tilde{w}+\varepsilon^{-1}{d_t}\partial_{z}\tilde{w},
   \end{cases}
   \end{align}
where $\nabla_{\mathbf{x}}$ and $\Delta_{\mathbf{x}}$ represent the gradient and Laplace operator with respect to $\mathbf{x}$, respectively. 
Substituting \eqref{inner_derivatives} into \eqref{CH11}--\eqref{CH44}, we can obtain the inner system.  
For any inner function $\tilde{w}(z,{\mathbf{x}},t)$ (e.g. $\tilde{\psi}$, $\tilde{\mu}_{\psi}$, $\tilde{\varphi}$, $\tilde{\mu}_{\varphi}$), we seek the inner expansions as
   \begin{align*}
   \tilde{w}(z,{\mathbf{x}},t)&={\tilde{w}}_{0}(z,{\mathbf{x}},t)+\varepsilon\tilde{w}_{1}(z,{\mathbf{x}},t)+{\varepsilon^2}\tilde{w}_{2}(z,{\mathbf{x}},t)+\cdots.
   \end{align*}

\vspace{4mm}
\noindent\textbf{Matching conditions between outer and inner expansions.}
For the outer and inner functions associated with the phase-field variables, one requires that 
   \begin{align}\label{matching condition 1 between outer and inner expansions}
   \tilde{\varphi}_0(z,{\mathbf{x}},t)=
   \begin{cases}
   \varphi_0^{\pm}({\mathbf{x}},t), \quad &z\to\pm\infty,\\
   \big(\varphi_0^{+}({\mathbf{x}},t)+\varphi_0^{-}({\mathbf{x}},t)\big)/2,\quad &z\to 0,
   \end{cases} 
   \end{align}
For further use, we develop the matching conditions for the outer and inner expansions of the chemical potential \cite{Caginalp1988Dynamics}. Fixing ${\mathbf{x}}$ on $\Gamma_i$, we seek to match the expansions by requiring that
   \begin{align*}
    (\mu_{\varphi_0}+\varepsilon\mu_{\varphi_1}+\mathcal{O}(\varepsilon^2))|_{({\mathbf{x}}+\varepsilon z\mathbf{m},t)}=(\tilde{\mu}_{\varphi_0}+\varepsilon\tilde{\mu}_{\varphi_1}+\mathcal{O}(\varepsilon^2))|_{(z,{\mathbf{x}},t)},
   \end{align*}
where $\mathbf{m}$ is the unit normal vector to $\Gamma_i$ and $\varepsilon z$ is between $O(\varepsilon)$ and $o(1)$. Expanding the left-hand side in the powers of $\varepsilon$ as $\varepsilon z\to0^+$, we have
   \begin{align}\label{these expansions}
   \mu_{\varphi_0}^{+}+\varepsilon(\mu_{\varphi_1}^{+}+zD_{\mathbf{m}}\mu_{\varphi_0}^{+})+
   \mathcal{O}(\varepsilon^2),
   \end{align}
where $D_{{\mathbf{m}}}$ denotes the directional derivative along $\mathbf{m}$ and
   \begin{align*}
   \mu_{\varphi_i}^{+}({\mathbf{x}},t)=\lim_{z\to0^{+}}\mu_{\varphi_i}({\mathbf{x}}+z{\mathbf{m}},t).
   \end{align*}
Similar results are obtained as $\varepsilon z\to0^-$. To match these expansions in \eqref{these expansions} with the inner expansion, one requires
   \begin{align}
   \mu_{\varphi_0}^{\pm}({\mathbf{x}},t)=&~\tilde{\mu}_{\varphi_0}(z,{\mathbf{x}},t),\quad z\to\pm\infty,\label{matching condition of chemical potential  1}\\
   (\mu_{\varphi_1}^{\pm}+zD_{\mathbf{m}}\mu_{\varphi_0}^{\pm})({\mathbf{x}},t)=&~\tilde{\mu}_{\varphi_1}(z,{\mathbf{x}},t),\quad z\to\pm\infty.\label{matching condition of chemical potential 2}
   \end{align}
Similarly, we can obtain the matching conditions between the outer and inner expansions associated with $\psi$. 

Now we can perform a matched asymptotic analysis for the expansions between the outer region and the inner region. It is worth noting that in the inner region associated with one phase-field function but away from the triple junction, the other phase-field function can be regarded as constant in its outer region. 
As a result, the asymptotic analysis  the standard two-phase model can be employed. We will illustrate the matching process by considering the matching of the outer and inner regions near $\Gamma_1$. Other matching cases near the interfaces (but away from the triple junction) can be done in a similar manner.

\vspace{4mm}
\noindent\textbf{Leading order.} For the outer regions on both sides of $\Gamma_1$, ${\psi}({\mathbf{x}},t)=\tilde{\psi}(z,{\mathbf{x}},t)=1$, which implies $\psi_0=1$ in $\Omega_2\cup\Omega_3$. At order $\mathcal{O}(\varepsilon^{-1})$,
we can derive  $\varphi_0=-1$ in $\Omega_2$, $\varphi_0=1$ in $\Omega_3$.  At order $\mathcal{O}(1)$,  we have $\Delta{\mu_{\varphi}}_0=0$ in $\Omega_2\cup\Omega_3$. Collecting the expansion results for the outer regions on both sides of $\Gamma_2$ and $\Gamma_3$, we have 
\begin{align}
   \begin{cases}\label{outer_results}
   \psi_0=-1 \quad &\text{in}\ \Omega_1,\\
   \psi_0=1,~~ \varphi_0=-1 \quad &\text{in}\ \Omega_2,\\
   \psi_0=1,~~ \varphi_0=1~~  \quad &\text{in}\ \Omega_3,\\
   \end{cases}
   \end{align}
and
   \begin{align}\label{eq:mu_psi0_and_mu_phi0}
   \begin{cases}
   \Delta{{\mu_\psi}_0}=0 \quad &\text{in}\ \Omega_1\cup\Omega_2\cup\Omega_3,\\
   \Delta{{\mu_\varphi}_0}=0 \quad &\text{in}\ \Omega_2\cup\Omega_3,
   \end{cases}
   \end{align}

For the inner expansions around  $\Gamma_1$, at order $\mathcal{O}(\varepsilon^{-2})$,  we have
   \begin{align}\label{independnet of z}
   \tilde{\mu}_{\varphi_0}(z,{\mathbf{x}},t)\equiv\mu_{\varphi_0}({\mathbf{x}},t).
   \end{align}
This follows from the matching condition \eqref{matching condition of chemical potential  1}.  At order $\mathcal{O}(\varepsilon^{-1})$, we have
   \begin{align}\label{Integrating obtain}
   \big(\partial_z\tilde{\varphi}_0\big)^2=2F(\tilde{\varphi}_0).
   \end{align}
Furthermore, using the condition \eqref{matching condition 1 between outer and inner expansions} and \eqref{outer_results}, we obtain    $\lim_{z\to\pm\infty}\tilde{\varphi}_0(z)=\pm1$.
Hence,
   \begin{align}\label{tanh_function}
   \tilde{\varphi}_0(z)=\tanh\big(\frac{z}{\sqrt{2}} \big).
   \end{align}
Similarly, around $\Gamma_2$ and $\Gamma_3$, we have
    \begin{align}\label{tanh_function_psi}
   \tilde{\psi}_0(z)=\tanh\big(\frac{z}{\sqrt{2}} \big).
   \end{align}
\vspace{4mm} 
\noindent\textbf{First order.} 
For the inner system across $\Gamma_1$,  at order $\mathcal{O}(\varepsilon^{-1})$, we have 
\begin{align*}
   d_t=\frac{1}{2}[{\mathbf{m}}\cdot\nabla\mu_{\varphi_0}]_{-}^{+},
   \end{align*}
where $[\bullet]$ denotes the jump from the $-$ side to the $+$ side,
and \eqref{matching condition 1 between outer and inner expansions}, \eqref{matching condition of chemical potential 2} and \eqref{independnet of z} have been used.

At order $\mathcal{O}(1)$, multiplying both sides by $\partial_z\tilde{\varphi}_0$ and integrating  the result with respect to $z$, we obtain
   \begin{align*}
   \tilde{\mu}_{\varphi_0}\tilde{\varphi}_0|_{-\infty}^{+\infty}+\kappa(x,t)\int_{-\infty}^{+\infty}\gamma_2(1)(\partial_z\tilde{\varphi}_0)^2 \mathrm{d}z=0.
   \end{align*}
Since  \eqref{tanh_function} implies 
   \begin{align*}
   \int_{-\infty}^{+\infty}\gamma_2(1)(\partial_z\tilde{\varphi}_0)^2 \mathrm{d}z
   =\sigma_{23},
   \end{align*}
this together with \eqref{independnet of z} and the matching condition \eqref{matching condition 1 between outer and inner expansions} leads to
   \begin{align}\label{mu_phi_gamma1}
   {\mu}_{\varphi_0}=-\frac{1}{2}\sigma_{23}\kappa_{23}(x,t), \quad \text{on}\ \Gamma_1,
   \end{align}
where $\kappa_{23}=\kappa$ is the curvature. Thus

Similarly, we have
   \begin{align}
   {\mu}_{\psi_0}=-\frac{1}{2}\sigma_{13}\kappa_{13}(x,t), \quad \text{on}\ \Gamma_2,\label{mu_phi_gamma2}\\
   {\mu}_{\psi_0}=-\frac{1}{2}\sigma_{12}\kappa_{12}(x,t), \quad \text{on}\ \Gamma_3,\label{mu_phi_gamma3}
   \end{align}
with
   \begin{align}\label{sigma_{12} and sigma_{13}}
   \int_{-\infty}^{+\infty}\gamma_1(1)(\partial_z\tilde{\psi}_0)^2 \mathrm{d}z=\sigma_{13},\qquad
   \int_{-\infty}^{+\infty}\gamma_1(-1)(\partial_z\tilde{\psi}_0)^2 \mathrm{d}z=\sigma_{12}.
   \end{align}
Therefore, the chemical potentials ${\mu}_{\psi_0}$ and ${\mu}_{\varphi_0}$ are determined by the Laplace equations \eqref{eq:mu_psi0_and_mu_phi0}, associated with the Dirichlet boundary conditions \eqref{mu_phi_gamma1}--\eqref{mu_phi_gamma3} on $\Gamma_i$ and the homogeneous Neumann conditions $\partial_{\bm{n}}{\mu_{\psi_0}}=\partial_{\bm{n}}{\mu_{\varphi_0}}=0$ on $\partial\Omega$. In general, ${\mu}_{\psi_0}$ and ${\mu}_{\varphi_0}$ are not constant.

\subsection{Analysis in the triple junction region and inner region}
In this subsection, we study the asymptotic behavior of the phase-field functions around the triple junction. Due to the finite interface thickness, there is a ``mixture'' region around the triple junction where the two phase-field functions vary between $-1$ and $1$. For our convenience of sharp-interface analysis, we shall define a physical triple junction point (line) by the intersection of the three asymptotes of the interfaces $\Gamma_i~(i=1,2,3)$ in the outer region as $\varepsilon\to0$
(shown in Fig. \ref{Sketch of the interfaces}), as one looks closer into the triple junction region from outside. 
The apparent contact angles $\theta_{23}$, $\theta_{12}$, $\theta_{13}$ are defined by the angles between the asymptotes.
\begin{figure}[htb]
   \centering
   \begin{subfigure}{0.48\linewidth}
   \centering
   \includegraphics[trim=0cm 0.5cm 0cm 0.3cm,clip,width=1\linewidth]{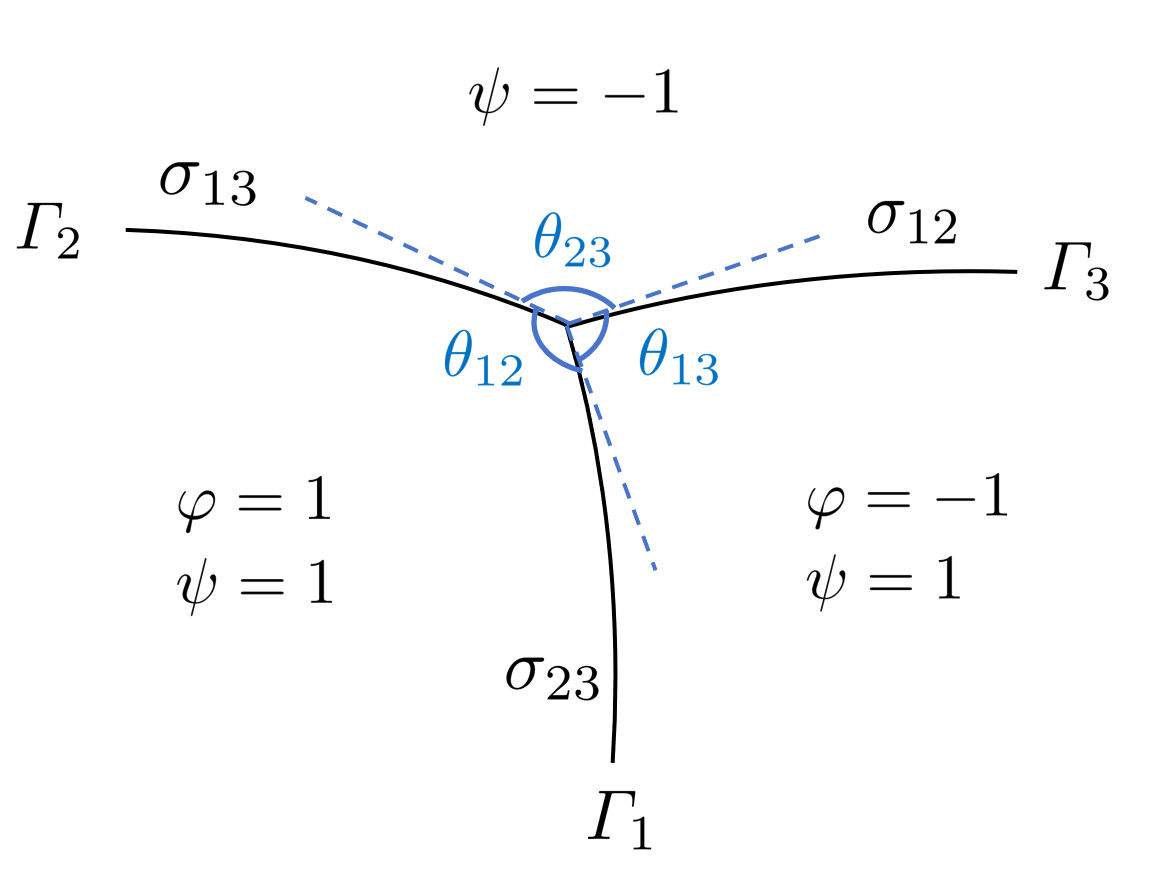}
   \end{subfigure}
   \begin{subfigure}{0.45\linewidth}
   \centering
   \includegraphics[trim=0cm 0.5cm 0cm 0.6cm,clip,width=1\linewidth]{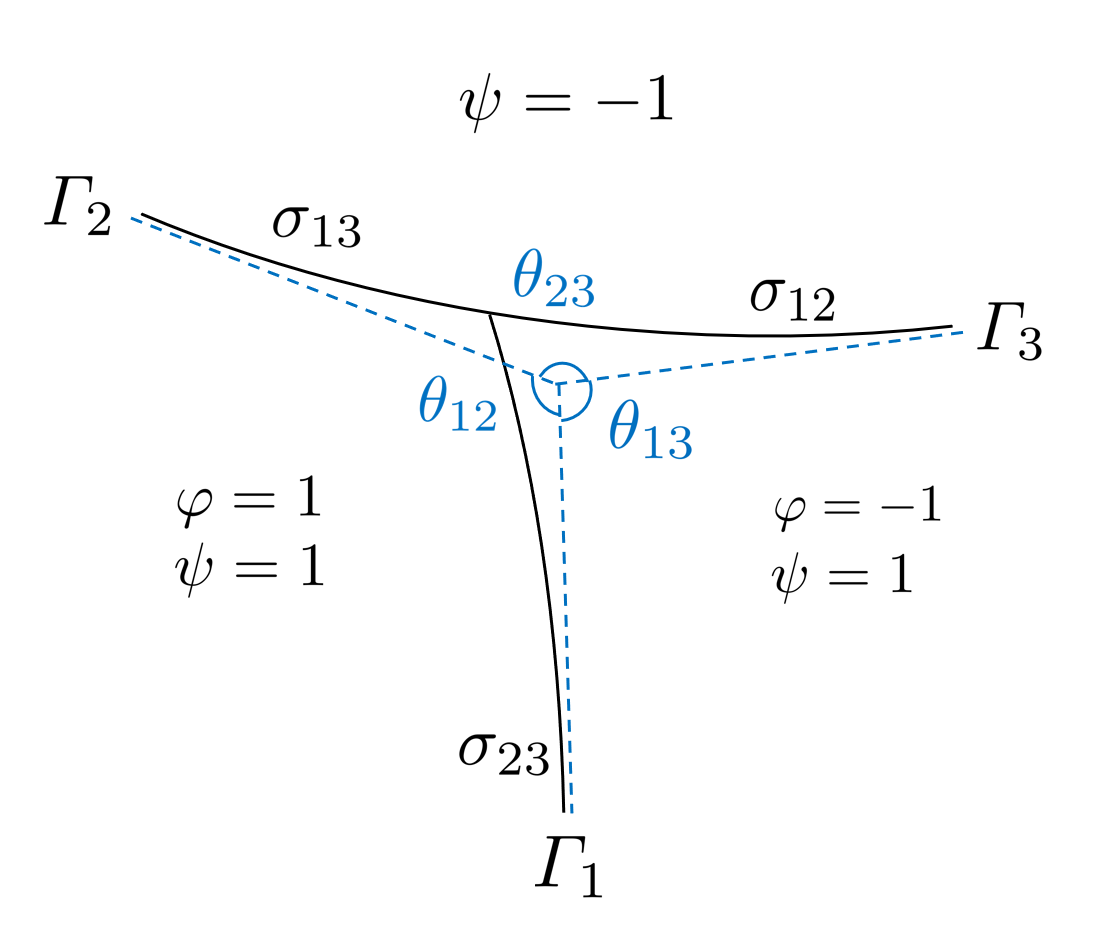}
   \end{subfigure}
   \caption{Sketch of the interfaces and apparent contact angles in the outer region (left panel: macroscopic view) and triple junction region (right panel: microscopic view). In both plots, $\theta_{23}$, $\theta_{12}$, and $\theta_{13}$ represent the apparent contact angles at the triple junction.}\label{Sketch of the interfaces}
   \end{figure}

   \vspace{4mm}
\noindent\textbf{Inner equations in the triple junction region.}
We focus on a two-dimensional plane normal to the triple junction line in the triple junction region. On such a plane, let $\mathbf{r}(t)$ be the coordinates of the triple junction point and introduce a stretched coordinate
   \begin{align*} 
   \bm{\eta}=\frac{\mathbf{x}-\mathbf{r}(t)}{\varepsilon}. 
   \end{align*}
For any function
$w$ (e.g. $\psi$, $\mu_{\psi}$, $\varphi$, $\mu_{\varphi}$) in the triple junction region, we can rewrite it as:
   \begin{align*}
   w(\mathbf{x},t)={\hat w}\Big(\frac{\mathbf{x}-\mathbf{r}(t)}{\varepsilon},t\Big).
   \end{align*} 
Then, we have
\begin{align}\label{triple_junction_derivatives}
    \begin{cases}
   \nabla w=\varepsilon^{-1}\nabla_{\bm\eta}\hat{w},\\
   \Delta w=\varepsilon^{-2}\Delta_{\bm\eta}\hat{w},\\
   \partial_{t}w=\partial_{t}\hat{w}-\varepsilon^{-1}r'(t)\nabla_{\bm\eta}\hat{w}.
   \end{cases}
   \end{align}
where $\nabla_{\bm{\eta}}$ and $\Delta_{\bm{\eta}}$ represent the gradient and Laplace operators with respect to $\bm{\eta}$. Substituting \eqref{triple_junction_derivatives} into \eqref{CH11}--\eqref{CH44} and using the asymptotic expansion for $\hat{w}(\bm\eta,t)$,
\begin{align*}
	\hat{w}(\bm\eta,t)&={\hat{w}}_{0}(\bm\eta,t)+\varepsilon\hat{w}_{1}(\bm\eta,t)+{\varepsilon^2}\hat{w}_{2}(\bm\eta,t)+\cdots,
\end{align*} 
we can obtain the systems at each order in the triple junction region.

\vspace{4mm}
\noindent\textbf{Matching conditions between the triple junction region and inner region.} 
It is more convenient to study the asymptotic behaviors of the underlying variables in the triple junction region by constructing an auxiliary triangle $T$, whose three sides $P_3P_2$, $P_1P_3$ and $P_2P_1$ are perpendicular to the three asymptotes of the interfaces $\Gamma_i~(i=1,2,3)$ in the outer region as $\varepsilon\to0$, as shown in Fig. \ref{Constructed triangle}. 
The circumcenter of $T$ exactly locates the triple junction point $\mathbf{r}(t)$. For our convenience of matching around $\Gamma_i$ ($i=1,2,3$), we employ three orthogonal coordinates $\bm\eta_i=\xi_i{\bm{\omega}}_i+\zeta_i{\bm{\tau}}_i$.
where ${\bm{\tau}}_i$ is in the tangent direction along the asymptote of the interface $\Gamma_i$ and ${\bm{\omega}}_i$ is perpendicular to ${\bm{\tau}}_i$. 
Then, we have the following matching condition:
        \begin{align}
        \lim_{\zeta_1\to+\infty}\hat{\varphi}_0(\xi_1,\zeta_1,t)&=\tilde{\varphi}_0(\xi_1,t),
        &&\lim_{\zeta_1\to+\infty}\hat{\psi}_0(\xi_1,\zeta_1,t)=1,\label{matching condition 1}\\
        \lim_{\zeta_2\to+\infty}\hat{\varphi}_0(\xi_2,\zeta_2,t)&=1,
        &&\lim_{\zeta_2\to+\infty}\hat{\psi}_0(\xi_2,\zeta_2,t)=\tilde{\psi}_0(\xi_2,t),\label{matching condition 2}
        \\
        \lim_{\zeta_3\to+\infty}\hat{\varphi}_0(\xi_3,\zeta_3,t)&=-1,
        &&\lim_{\zeta_3\to+\infty}\hat{\psi}_0(\xi_3,\zeta_3,t)=\tilde{\psi}_0(\xi_3,t).\label{matching condition 3}
    \end{align}
   
    \begin{figure}[htb]
    \centering
    \includegraphics[trim=0cm 0.3cm 0cm 0.3cm,clip,width=6cm]{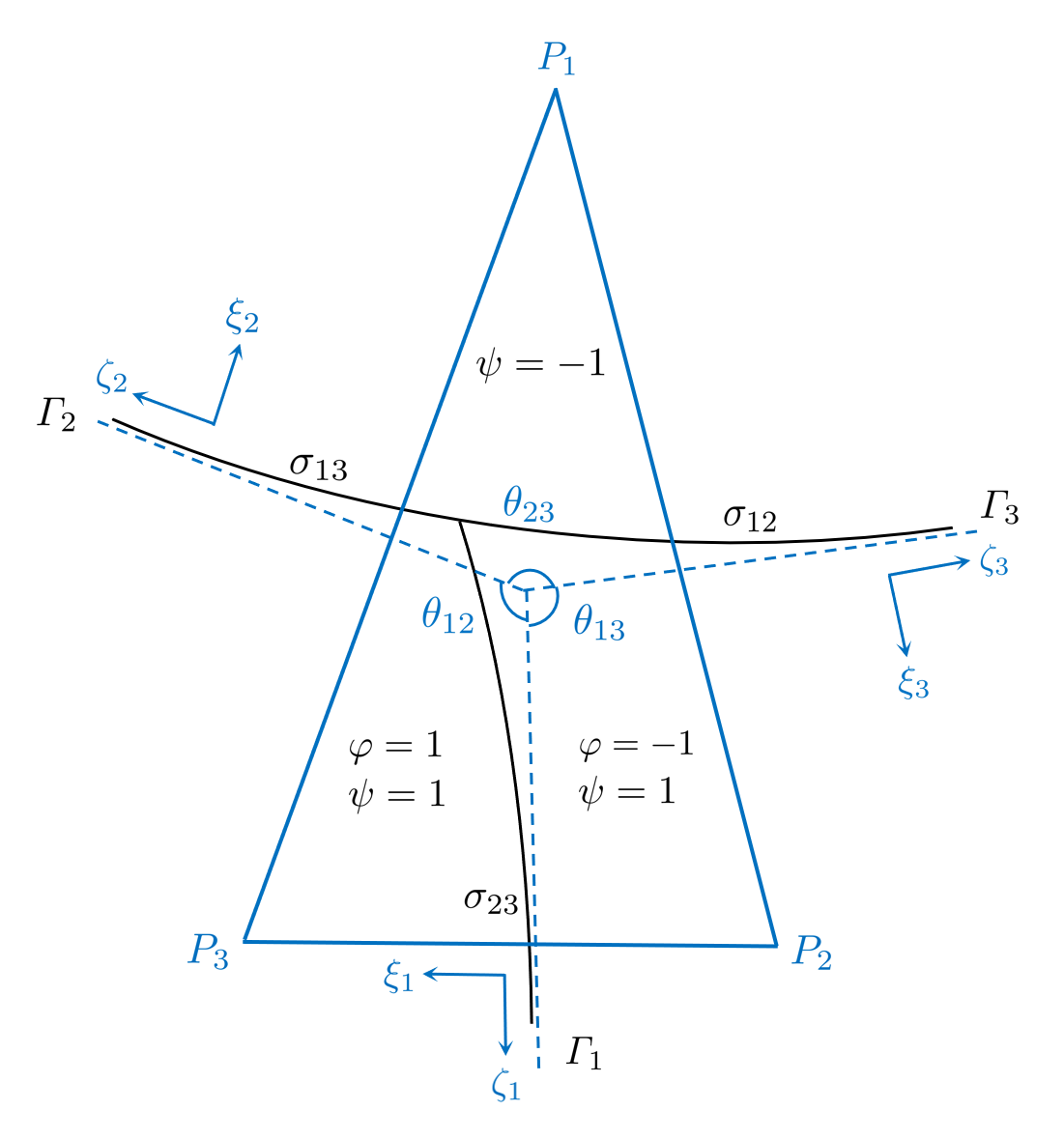}
    \caption{Auxiliary triangle for matching in the triple junction region.}\label{Constructed triangle}
    \end{figure}

\vspace{4mm}
\noindent\textbf{Leading order.} 
At the $\big(\mathcal{O}(\varepsilon^{-2}), \mathcal{O}(\varepsilon^{-1}),\mathcal{O}(\varepsilon^{-2}), \mathcal{O}(\varepsilon^{-1})\big)$ order, the systems in the triple junction region reads
   \begin{align}
   M_1\Delta_{\bm\eta}\hat{\mu}_{\psi_0}&=0,
   \label{first order 3 around triple junction}\\
   -\nabla_{\bm\eta}\cdot\big(\gamma_1(\hat{\varphi}_0)\nabla_{\bm\eta}\hat{\psi}_0)\big)+\gamma_1(\hat{\varphi}_0)f(\hat{\psi}_0)
   +\gamma_2'(\hat{\psi}_0)\Big(\frac{1}{2}(\nabla_{\bm\eta}\hat{\varphi}_0)^2+F(\hat{\varphi}_0) \Big)&=0,
   \label{first order 4 around triple junction}\\
   \nabla_{\bm\eta}\cdot\big(M_2(\hat{\psi}_0)\nabla_{\bm\eta}\hat{\mu}_{\varphi_0}\big)&=0,
   \label{first order 1 around triple junction}\\
    -\nabla_{\bm\eta}\cdot\big(\gamma_2(\hat{\psi}_0)\nabla_{\bm\eta}\hat{\varphi}_0)\big)+\gamma_2(\hat{\psi}_0)f(\hat{\varphi}_0)+\gamma_1'(\hat{\varphi}_0)\Big(\frac{1}{2}(\nabla_{\bm\eta}\hat{\psi}_0)^2+F(\hat{\psi}_0) \Big)&=0,\label{first order 2 around triple junction}
   \end{align}
We first multiply \eqref{first order 4 around triple junction} and \eqref{first order 2 around triple junction}  by $\nabla_{\bm\eta}\hat{\psi}_0$ and $\nabla_{\bm\eta}\hat{\varphi}_0$ respectively, then sum up the results and take the integration over $T$. This leads to a force balance equation for the total force over $T$:
    \begin{align}\label{inter2}
    \iint_T \nabla_{\bm\eta} \cdot \bm{F} \mathrm{d}A = \oint_{\partial T} \bm{F} \cdot \bm{\nu} \mathrm{d}s = \bm 0,
    \end{align}
where the divergence theorem is applied with $\bm\nu$ being the outer normal vector to $\partial T$, and
    \begin{align*}
        \bm{F}=&
        \gamma_1(\hat{\varphi}_0)\bigg(-(\nabla_{\bm\eta}\hat{\psi}_0\otimes\nabla_{\bm\eta}\hat{\psi}_0)+\Big(\frac{1}{2}(\nabla_{\bm\eta}\hat{\psi}_0)^2+F(\hat{\psi}_0)\Big)\mathbf{I}\bigg)\\
        &+\gamma_2(\hat{\psi}_0)\bigg(-(\nabla_{\bm\eta}\hat{\varphi}_0\otimes\nabla_{\bm\eta}\hat{\varphi}_0)+\Big(\frac{1}{2}(\nabla_{\bm\eta}\hat{\varphi}_0)^2+F(\hat{\varphi}_0)\Big)\mathbf{I}\bigg)  
    \end{align*}
is the stress tensor for the phase-field approximation of surface tension. Here $\mathbf{I}$ denotes the second-order isotropic tensor.

Denote by 
$R$ the circumradius of the triangle. For convenience, the line integrals along the three edges $P_3P_2$, $P_1P_3$ and $P_2P_1$ in \eqref{inter2} can be calculated one by one using the corresponding coordinates $\bm\eta_i~(i=1,2,3)$ when $R\to+\infty$. 
Let $\bm\eta_1=\xi_1\bm{\omega}_1+\zeta_1\bm{\tau}_1$ be the reference frame, then the outward unit normal vector can be written as
$\bm\nu=\cos(\bm\nu,\bm{\omega}_1)\bm{\omega}_1+\cos(\bm\nu,\bm{\tau}_1)\bm{\tau}_1$,
where $(\bm\nu,\bm{\omega}_1)$ and $(\bm\nu,\bm{\tau}_1)$ denote the counterclockwise rotational angles between $\bm\nu$ and the basis vectors $\bm{\omega}_1$ and $\bm{\tau}_1$, respectively.  

Taking the line integral along $P_2P_1$ as an example, it is convenient to expand the line integral in terms of the coordinate $\bm\eta_3=\xi_3\bm{\omega}_3+\zeta_3\bm{\tau}_3$, with $\bm{\tau}_3=\bm{\nu}=(-\sin\theta_{13})\bm{\omega}_1+(\cos\theta_{13})\bm{\tau}_1$ being the outward unit normal vector of $P_2P_1$. 
Using \eqref{sigma_{12} and sigma_{13}} as well as the matching conditions \eqref{matching condition 3}, $\lim_{\zeta_3\to+\infty}\partial_{\zeta_3}\hat{\varphi}_0=0$, and $\lim_{\zeta_3\to+\infty}\partial_{\zeta_3}\hat{\psi}_0=0$, we can calculate the $\bm{\omega}_1$ component of \eqref{inter2} as $R\to+\infty$ through the inner solution $(\tilde{\psi}_0,\tilde{\varphi}_0)$:
   \begin{align*}
   &\lim_{R\to+\infty}\ \int\limits_{P_2}^{P_1}\bm{\omega}_1\cdot(\bm{F}\cdot\bm{\nu})\mathrm{d}s\notag\\
   =&-\int_{-\infty}^{+\infty}\Big(\gamma_1(-1)\cos\theta_{13}\sin\theta_{13}(\partial_{\xi_3}\tilde{\psi}_0)^2 \Big)\cos\theta_{13}\mathrm{d}\xi_3
   \notag\\
   &-\int_{-\infty}^{+\infty}\Big(\gamma_2(\tilde{\psi}_0)F(-1)+\gamma_1(-1)F(\tilde{\psi}_0)-\frac{1}{2}\gamma_1(-1)\cos(2\theta_{13})(\partial_{\xi_3}\tilde{\psi}_0)^2 \Big)\sin\theta_{13}\mathrm{d}\xi_3
   \notag\\
   =&-\sigma_{12}\cos^2\theta_{13}\sin\theta_{13}-\frac{1}{2}\sigma_{12}\sin\theta_{13}+
   \frac{1}{2}\sigma_{12}\cos(2\theta_{13})\sin\theta_{13}.
   \end{align*}   
The other two line integrals in \eqref{inter2}
can be treated in a similar manner. Combining the three line integrals together, we have
    \begin{align*}
        &\lim_{R\to+\infty}\ \Big(\int_{P_3}^{P_2}+
        \int_{P_2}^{P_1}+
        \int_{P_1}^{P_3}\Big)
        \bm{\omega}_1\cdot(\bm{F}\cdot\bm{\nu}) \mathrm{d}s
        \notag\\
        &~~~~~~~~~~=-\sigma_{12}\cos^2\theta_{13}\sin\theta_{13}-\frac{1}{2}\sigma_{12}\sin\theta_{13}+
        \frac{1}{2}\sigma_{12}\cos(2\theta_{13})\sin\theta_{13}
        \notag\\
        &~~~~~~~~~~~~~+\sigma_{13}\cos^2\theta_{12}\sin\theta_{12}+\frac{1}{2}\sigma_{13}\sin\theta_{12}-
        \frac{1}{2}\sigma_{13}\cos(2\theta_{12})\sin\theta_{12}
        \notag\\
        &~~~~~~~~~~=0.        
    \end{align*}
It follows that
   \begin{align*}
   \sin\theta_{13}\sigma_{12}=\sin\theta_{12}\sigma_{13}.
   \end{align*}
If we calculate the $\bm{\tau}_1$ component of \eqref{inter2} as $R\to+\infty$, we can derive another force balance equation at the triple junction in the same way:
   \begin{align*}
   \sin\theta_{12}\sigma_{23}=\sin\theta_{23}\sigma_{12}.
   \end{align*}
Combining these two equations, we obtain the Neumann triangle condition at the triple junction:
   \begin{align}\label{Neumann_triangle_condition}
   \frac{\sin\theta_{23}}{\sigma_{23}}=\frac{\sin\theta_{12}}{\sigma_{12}}=\frac{\sin\theta_{13}}{\sigma_{13}}.
   \end{align}
\par We now summarize the leading order behavior:
   \begin{align*}
   \begin{cases}
   \psi_0=-1 \quad &\text{in}\ \Omega_1,\\
   \psi_0=1,~~ \varphi_0=-1 \quad &\text{in}\ \Omega_2,\\
   \psi_0=1,~~ \varphi_0=1~~  \quad &\text{in}\ \Omega_3,\\
   \end{cases}
   \end{align*}
and
   \begin{align*}
   \begin{cases}
   \Delta{{\mu_\psi}_0}=0 \quad &\text{in}\ \Omega_1\cup\Omega_2\cup\Omega_3,\\
   \Delta{{\mu_\varphi}_0}=0 \quad &\text{in}\ \Omega_2\cup\Omega_3,
   \end{cases}
   \end{align*}
with  
   \begin{align*}
   \begin{cases}
   \partial_{\bm{n}}{\mu_{\psi_0}}=0, \
   \partial_{\bm{n}}{\mu_{\varphi_0}}=0 \quad &\text{on}\ \partial\Omega,
   \\
   \mu_{\varphi_0}=-\frac{1}{2}\sigma_{23}\kappa_{23}(x,t),\
   d_t=\frac{1}{2}[{\mathbf{m}}\cdot\nabla\mu_{\varphi_0}]_{-}^{+}\quad &\text{on}\ \Gamma_1,
   \\
   \mu_{\psi_0}=-\frac{1}{2}\sigma_{13}\kappa_{13}(x,t), \ 
   d_t=\frac{1}{2}[{\mathbf{m}}\cdot\nabla\mu_{\psi_0}]_{-}^{+} 
   \quad &\text{on}\ \Gamma_2,
   \\
   \mu_{\psi_0}=-\frac{1}{2}\sigma_{12}\kappa_{12}(x,t), \ 
   d_t=\frac{1}{2}[{\mathbf{m}}\cdot\nabla\mu_{\psi_0}]_{-}^{+} \quad &\text{on}\ \Gamma_3,
   \\
   \frac{\sin\theta_{23}}{\sigma_{23}}=\frac{\sin\theta_{12}}{\sigma_{12}}=\frac{\sin\theta_{13}}{\sigma_{13}}\quad &\text{at}\  \Gamma_1\cap\Gamma_2\cap\Gamma_3.
   \end{cases}
   \end{align*}
       
    \begin{figure}[h]
       \centering
        \includegraphics[trim=0cm 0.5cm 0cm 0.3cm,clip,width=8.6cm]{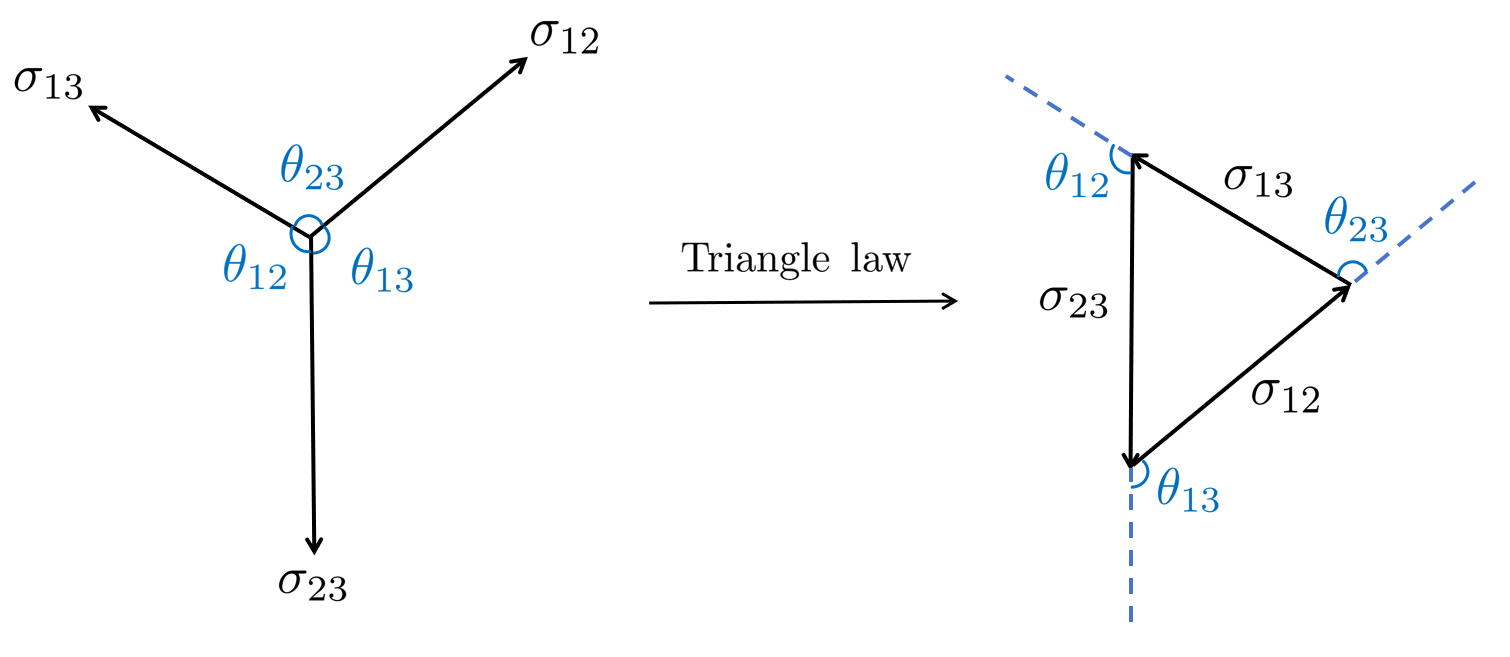}
        \caption{Sketch of the force balance at the triple junction and the closed force triangle.}\label{triangle_law}
    \end{figure} 
    \begin{remark}\label{remark_triangle_law}
   When the three surface tensions are balanced at the sharp-interface triple junction, the total force is zero. As shown in Fig. \ref{triangle_law}, this is equivalent to the triangle law of vector addition, i.e., the three vectors form a closed force triangle. It is obvious that the Neumann triangle condition \eqref{Neumann_triangle_condition} is essentially the sine rule derived from this closed force triangle.
   \end{remark}

\section{Numerical schemes}\label{Sec:Numerical_schemes}
In this section, we extend the MOS technique \cite{Lu2025Decoupled} to general dissipative systems including the proposed $N$-phase model, thus establishing a unified framework to construct a decoupled and energy stable numerical scheme with second-order temporal accuracy. 
As a specific implementation, we develop an efficient scheme for the ternary-phase case. This scheme requires only solving a sequence of linear equations at each time step.
We show that the proposed scheme is energy stable with respect to the original free energy.
   
\subsection{MOS technique for general dissipative systems}\label{Sec4.2}
We consider the general dissipative system:
\begin{equation}\label{PDE}
\left\{\begin{array}{l}
	\Phi_t = \mathcal M(\Phi)\frac{\delta W(\Phi)}{\delta \Phi},\\
    \Phi|_{t=0}=\Phi_0.
\end{array}\right.
\end{equation}
where $\Phi = (\varphi_1,\varphi_2,\cdots,\varphi_m)^{\top}$, $\mathcal M(\Phi)$ is a negative semi-definite mobility operator, and $W(\Phi)$ denotes the total free energy. Due to the coupling among the variables, it is challenging to numerically solving this system with high-order accuracy in a decoupling way, while preserving the original energy dissipation structure. The method developed in this section aims to accomplish this complex objective. 

Suppose we can decompose $\mathcal{M}$ into $n$ parts as follows:
\begin{equation*}
	\mathcal{M}=\mathcal{M}_1+\cdots+\mathcal{M}_n,
\end{equation*}
where $\mathcal{M}_i$ $(1\leqslant i\leqslant n)$ is negative semi-definite. Consequently, the original system is split into the following $n$ dissipative subsystems:
\begin{equation}\label{splited system}
	\Phi_t = \mathcal M_i\frac{\delta W}{\delta \Phi},\qquad i=1,2,\cdots, n,
\end{equation}

Different from other operator splitting schemes \cite{li2022stability, li2022stabilityand}, the MOS technique offers flexibility in constructing high-order methods by a proper decomposition of the mobility operator. Let $\mathcal{S}^i_\tau$ (for $1 \leqslant i \leqslant n$) denote the solution map with time step $\tau$ for the $i$-th subsystem in \eqref{splited system} and let $\mathcal{S}^{\mathrm{e}}_\tau$ denote the exact solution map of the entire system \eqref{PDE}.
Then the Strang-type composition leads to the following second-order approximation:
\begin{align*}
\Phi(\tau,\cdot):=\mathcal{S}^{n}_{\tau/2}
\circ\cdots\circ\mathcal{S}^{2}_{\tau/2}
\circ\mathcal{S}^{1}_{\tau}\circ
\mathcal{S}^{2}_{\tau/2}\circ\cdots\circ\mathcal{S}^{n}_{\tau/2}\Phi_0=\mathcal{S}_\tau^{\mathrm{e}}\Phi_0
+O(\tau^3).
\end{align*}
By a simple argument using the transition property of inequalities, we can easily establish the energy stability of this approximation.
\begin{theorem}\label{theorem energy stable}
Assume that $\mathcal{S}^1_{\tau}$, $\mathcal{S}^2_{\tau}$, $\cdots$, $\mathcal{S}^{n}_{\tau}$ are energy stable integrators for the dissipative systems in (\ref{splited system}) respectively, i.e., $W(\mathcal{S}^i_{\tau}\Phi)\leqslant W(\Phi)$ for $i=1,2,\cdots,n$.
Then the Strang-type composition method
$\mathcal{S}^{n}_{\tau/2}
\circ\cdots\circ\mathcal{S}^{2}_{\tau/2}
\circ\mathcal{S}^{1}_{\tau}\circ
\mathcal{S}^{2}_{\tau/2}\circ\cdots\circ\mathcal{S}^{n}_{\tau/2}$ is also energy stable for the original system (\ref{PDE}): $W(\Phi(\tau,\cdot))\leqslant W(\Phi_0)$.
\end{theorem}

The MOS method allows the use of any numerical integrators for each subsystem and composite such integrators in different manners. One only needs to ensure the desired numerical accuracy and the energy stability issue when selecting numerical integrators. For example, one can apply energy stable numerical methods such as convex splitting \cite{ Unconditionallygradient}, the stabilization method \cite{shen2010numerical}, IEQ \cite{yang2016linear}, SAV \cite{shen2018scalar}, etc., to the subsystems \eqref{splited system}.
Let $\widehat{\mathcal{S}}^{i}_{\tau}$ denote a second-order accurate numerical solver for the $i$-th subsystem (that is, $\widehat{\mathcal{S}}^{i}_{\tau}=\mathcal{S}^{i}_{\tau}+O(\tau^3))$. Using the Strang-type composition, we can construct a second-order energy stable scheme for \eqref{PDE} as follows
\begin{equation}\label{eq second order ss}
        \widehat{\Phi}(\tau,\cdot)=\widehat{\mathcal{S}}^{n}_
        {\tau/2}
        \circ\cdots\circ\widehat{\mathcal{S}}^{2}
        _{\tau/2}
        \circ\widehat{\mathcal{S}}^{1}_{\tau}
        \circ
       \widehat{\mathcal{S}}^{2}_{\tau/2}
       \circ\cdots\circ\widehat{\mathcal{S}}^{n}
       _{\tau/2}\Phi_0.
\end{equation}
It is straightforward to see that this composition has second-order accuracy,
\begin{align*} 
\mathcal{S}_{\tau}^{\mathrm{e}}\Phi_0=&~\mathcal{S}^{n}_{\tau/2}
\circ\cdots\circ\mathcal{S}^{2}_{\tau/2}
\circ\mathcal{S}^{1}_{\tau}\circ
\mathcal{S}^{2}_{\tau/2}\circ\cdots\circ\mathcal{S}^{n}_{\tau/2}\Phi_0+O(\tau^3)\\
=&~\big(\widehat{\mathcal{S}}^{n}_{\tau/2}+O(\tau^3)\big)
\circ\cdots\circ
\big(\widehat{\mathcal{S}}^{2}_{\tau/2}+O(\tau^3)\big)
\circ\big(\widehat{\mathcal{S}}^{1}_\tau+O(\tau^3)\big)
\notag\\
&~\circ\big(\widehat{\mathcal{S}}^{2}_{\tau/2}+O(\tau^3)\big)
\circ\cdots\circ
\big(\widehat{\mathcal{S}}^{n}_{\tau/2}+O(\tau^3)\big)\Phi_0+O(\tau^3)\\
=&~\widehat{\mathcal{S}}^{n}_{\tau/2}
\circ\cdots\circ\widehat{\mathcal{S}}^{2}_{\tau/2}
\circ\widehat{\mathcal{S}}^{1}_{\tau}\circ
\widehat{\mathcal{S}}^{2}_{\tau/2}\circ\cdots\circ\widehat{\mathcal{S}}^{n}_{\tau/2}\Phi_0
+O(\tau^3).
\end{align*}
Analogous to the Theorem \ref{theorem energy stable}, we have the following corollary.
\begin{corollary}\label{theorem energy stable1}
Assume that the discrete operators $\widehat{\mathcal{S}}^{1}_{\tau}$, $\widehat{\mathcal{S}}^{2}_{\tau}$, $\cdots$, $\widehat{\mathcal{S}}^{n}_{\tau}$ are energy stable integrators for the $n$ dissipative systems in \eqref{splited system}, i.e., $W(\widehat{\mathcal{S}}^i_{\tau}\Phi)\leqslant W(\Phi)$ for $i=1,2,\cdots,n$. Then the second-order Strang-type composition 
$\widehat{\mathcal{S}}^{n}_{\tau/2}
\circ\cdots\circ\widehat{\mathcal{S}}^{2}_{\tau/2}
\circ\widehat{\mathcal{S}}^{1}_{\tau}\circ
\widehat{\mathcal{S}}^{2}_{\tau/2}\circ\cdots\circ\widehat{\mathcal{S}}^{n}_{\tau/2}$ has the energy dissipation property: $W(\widehat\Phi(\tau,\cdot))\leqslant W(\Phi_0)$.
\end{corollary}

\subsection{Energy stable scheme for ternary-phase model}
By taking $n=m=N-1$, the framework for constructing decoupled and energy stable numerical schemes for general dissipative systems established in the previous subsection can be made concrete in its application to the proposed $N$-phase model. In this subsection, we apply the MOS technique to the ternary-phase model by providing a specific numerical scheme. We first introduce Lipschitz continuity assumptions on the derivatives of $F(z)$, $\gamma_1(\varphi)$ and $\gamma_2(\psi)$:
there exist constants $L_F$, $L_{\gamma_1}$ and $L_{\gamma_2}$ such that
   \begin{align*}
   \max_{z\in \mathbb{R}}|F''(z)|\leqslant L_F,\qquad
   \max_{\varphi\in \mathbb{R}}|\gamma_1''(\varphi)|\leqslant L_{\gamma_1},\qquad
   \max_{\psi\in \mathbb{R}}|\gamma_2''(\psi)|\leqslant L_{\gamma_2}.
   \end{align*}
Note that for commonly used potentials such as the double-well potential $F(z)=\frac{1}{4}(z^2-1)^2$, a regularized version can be constructed by modifying its growth rate outside the interval $[-1,1]$ to ensure that the Lipschitz continuity condition is satisfied. Similar regularization strategies can be applied to $\gamma_1(\varphi)$ and $\gamma_2(\psi)$, since the physically relevant phase-field behavior is confined to the bounded domain where $(\psi,\varphi)$ lies within a 2-cube.

Based on the composition \eqref{eq second order ss}, that is, $\widehat{\mathcal{S}}_{\tau/2}^{2}\circ\widehat{\mathcal{S}}_{\tau}^{1}\circ
\widehat{\mathcal{S}}_{\tau/2}^{2}(\psi^n,\varphi^n)$,
the second-order modified Crank-Nicolson scheme for the ternary-phase model \eqref{CH11}--\eqref{ternary_BC} is given by:

   \noindent\textbf{Step 1 ($\widehat{\mathcal{S}}_{\tau/2}^{2}$)}:
   \begin{align}  
   &\frac{\varphi^{n+\frac{1}{2}}-\varphi^n}{\tau/2}=\nabla\cdot\big(M_2(\psi^n)\nabla\mu_{\varphi}^{n+\frac{1}{4}}\big),\label{second order scheme3 1}\\
   &\mu_{\varphi}^{n+\frac{1}{4}}=-\varepsilon\nabla\cdot\Big(
   \gamma_2(\psi^n)\nabla\frac{\varphi^{n+\frac{1}{2}}+\varphi^n}{2} \Big)+\frac{1}{\varepsilon}\gamma_2(\psi^n)\mathcal{T}[f](\varphi^n,\varphi^{n+\frac{1}{2}};A_1,\tau)
   \notag\\
   &~\quad\qquad+g(\psi^n)\mathcal{T}[\gamma_1'](\varphi^n,\varphi^{n+\frac{1}{2}};B_1,\tau),\label{second order scheme3 2}
   \end{align} 
   \noindent\textbf{Step 2 ($\widehat{\mathcal{S}}_{\tau}^{1})$}:
   \begin{align}     
   &\frac{\psi^{n+1}-\psi^n}{\tau}=\nabla\cdot\big(M_1\nabla\mu_{\psi}^{n+\frac{1}{2}}\big),\label{second order scheme3 3}
   \\
   &\mu_{\psi}^{n+\frac{1}{2}}=-\varepsilon\nabla\cdot\Big(
   \gamma_1(\varphi^{n+\frac{1}{2}})\nabla\frac{\psi^{n+1}+\psi^n}{2} \Big)
   +\frac{1}{\varepsilon}\gamma_1(\varphi^{n+\frac{1}{2}})\mathcal{T}[f](\psi^n,\psi^{n+1};A_2,\tau)
   \notag\\
   &~\quad\qquad+g(\varphi^{n+\frac{1}{2}})\mathcal{T}[\gamma_2'](\psi^n,\psi^{n+1};B_2,\tau),
   \label{second order scheme3 4}
   \end{align}
   \noindent\textbf{Step 3 ($\widehat{\mathcal{S}}_{\tau/2}^{2})$}:
   \begin{align}   
   &\frac{\varphi^{n+1}-\varphi^{n+\frac{1}{2}}}{\tau/2}=\nabla\cdot\big(M_2(\psi^{n+1})\nabla\mu_{\varphi}^{n+\frac{3}{4}}\big),\label{second order scheme3 5}
   \\
   &\mu_{\varphi}^{n+\frac{3}{4}}=-\varepsilon\nabla\cdot\Big(
   \gamma_2(\psi^{n+1})\nabla\frac{\varphi^{n+1}+\varphi^{n+\frac{1}{2}}}{2} \Big)
   +\frac{1}{\varepsilon}\gamma_2(\psi^{n+1})\mathcal{T}[f](\varphi^{n+\frac{1}{2}},\varphi^{n+1};A_1,\tau)
   \notag\\
   &~\quad\qquad+g(\psi^{n+1})\mathcal{T}[\gamma_1'](\varphi^{n+\frac{1}{2}},\varphi^{n+1};B_1,\tau),\label{second order scheme3 6}
   \end{align}
with boundary conditions
\begin{equation}\label{second order scheme3 BC3}
   \begin{aligned}
   \partial_{n}\varphi^{n+\frac{1}{2}}=0,\qquad\partial_{n}\psi^{n+1}=0,\qquad \partial_{n}\varphi^{n+1}=0,\\
    \partial_{n}\mu_{\varphi}^{n+\frac{1}{4}}=0,\qquad \partial_{n}\mu_{\psi}^{n+\frac{1}{2}}=0,   \qquad \partial_{n}\mu_{\varphi}^{n+\frac{3}{4}}=0
   \end{aligned}\qquad \text{on}\ \partial\Omega,
   \end{equation}
where 
\begin{equation*}
\mathcal{T}[f](\varphi,\psi;A,\tau)=f(\varphi)+\frac{1}{2}f'(\varphi)(\psi-\varphi)
   +A\tau(\psi-\varphi),
\end{equation*}
$g(\bullet)=\frac{\varepsilon}{2}|\nabla \bullet|^2+\frac{1}{\varepsilon}F(\bullet)$, and $A_1$, $A_2$, $B_1$, $B_2$ are four nonnegative constants.

   
\begin{theorem}\label{Thm_energy_stable}
Under the condition
   \begin{align*}
   \tau\geqslant \max\biggl\{\frac{L_{F}}{A_1},\frac{L_{\gamma_1}}{B_1},\frac{L_{F}}{A_2},\frac{L_{\gamma_2}}{B_2}\biggr\},
   \end{align*}
the scheme \eqref{second order scheme3 1}--\eqref{second order scheme3 BC3} is energy stable and the following discrete energy law holds:
   \begin{align*}
   {W}(\psi^{n+1},\varphi^{n+1})-{W}(\psi^{n},\varphi^{n})\leqslant&
   -\frac{\tau}{2}\int_{\Omega}M_2(\psi^n)|\nabla\mu_\varphi^{n+\frac{1}{4}}|^2\mathrm{d}\mathbf{x}
   -\tau\int_{\Omega}M_1|\nabla\mu_\psi^{n+\frac{1}{2}}|^2\mathrm{d}\mathbf{x}
   \notag\\
   &-\frac{\tau}{2}\int_{\Omega}M_2(\psi^{n+1})|\nabla\mu_\varphi^{n+\frac{3}{4}}|^2\mathrm{d}\mathbf{x}.
   \end{align*}
\end{theorem}

\begin{proof}
    The detailed proof is presented in \ref{proof_of_Thm}.
\end{proof}

\section{Numerical simulations}\label{Sec:Numerical_simulations}
In this section, we numerically validate the theoretical results in the previous sections and investigate the applications of the proposed model. In Section \ref{Accuracy test}, we validate the temporal and spatial orders of accuracy for the proposed numerical scheme. Then the energy stability, the algebraic consistency, the volume conservation, and the asymptotic properties such as the Neumann angle condition of the DBPF model are validated in Section \ref{stability_algebraic consistency_volume-conservation}. Lastly in Sections \ref{Liquid lens between two stratified fluids} and \ref{Evolution of two close-by droplets}, the proposed DBPF model is successfully employed to numerically simulate several benchmark problems, including the formation of liquid lenses between two stratified fluids, and the compound droplet emulsions.
  
In all numerical experiments, we numerically solve the DBPF model using the MOS-based numerical scheme on a uniform mesh in the computational domain $\Omega=[0,1]\times[0,1]$. Unless otherwise specified, we always take $M_1=m_1$ and $M_2(\psi)=m_2(\frac{1+\psi}{2})^8$.

\subsection{Accuracy test of numerical scheme}\label{Accuracy test}
In this subsection, we validate the orders of accuracy for the proposed second-order scheme \eqref{second order scheme3 1}--\eqref{second order scheme3 BC3}. As exact solutions are unknown, 
the convergence rates are estimated by examining the differences between solutions on adjacent mesh levels \cite{LeVeque2007Finite}:
\begin{equation*}
p\approx\log_2\big(\|U^h-U^{h/2}\|/\|U^{h/2}-U^{h/4}\|\big),
\end{equation*}
where $U^h$ is the numerical solution in the mesh with size $h$. 
Then the numerical error can be approximated as 
    \begin{align*}
        &e(h):=Ch^p+O(h^{p+1})
         \approx\frac{1}{1-1/2^p}\|U^h-U^{h/2}\|.
    \end{align*}
We will compute the $l_\infty$- and $l_2$-convergence rates in space and time by employing this approach. We set the interfacial width as $\varepsilon=0.03$, the mobility coefficient $m_1=m_2=1e-4$, the stabilization parameters $A_1=B_1=A_2=B_2=100$. The initial conditions are given by
   \begin{align}\label{initial_conditions_square}
   \psi(x,y)=\tanh\Big(\frac{y-0.5}{\varepsilon}\Big),\qquad
   \varphi(x,y)=\tanh\Big(\frac{x-0.5}{\varepsilon}\Big).
   \end{align}

In the accuracy test for spatial discretization, the time step is fixed at $\tau=10^{-3}$.  The numerical solutions are computed using the second-order scheme with different mesh sizes $h=1/64$, $1/128$, $1/256$, $1/512$, $1/1024$ at $t=0.5$ and $t=1$. As shown in Table \ref{l_2 spatial errors}, we observe that the scheme \eqref{second order scheme3 1}--\eqref{second order scheme3 BC3} achieves second-order accuracy in space in the sense of the norms $l_\infty$ and $l_2$. 
\begin{table}[ht]
\centering
\caption{The $l_\infty$-errors, $l_2$-errors and order of convergence for $\varphi$ and $\psi$ in space at $t=0.5$ and $t=1$ with fixed $\tau=10^{-3}$ and different mesh size $h$.}
\begin{tabular}{@{}>{\centering\arraybackslash}p{1.4cm}cccccccc@{}}
\toprule[1.5pt]
\cmidrule(lr){2-5}\cmidrule(lr){6-9}
 & \multicolumn{2}{c}{$t=0.5$} & \multicolumn{2}{c}{$t=1$} & \multicolumn{2}{c}{$t=0.5$} & \multicolumn{2}{c}{$t=1$}\\
\cmidrule(lr){2-3}\cmidrule(lr){4-5}\cmidrule(lr){6-7}\cmidrule(lr){8-9}
$h$ & $\|e_{\varphi}\|_{\infty}$ & $\text{Order}$  & $\|e_{\varphi}\|_{\infty}$ & $\text{Order}$ & $\|e_{\varphi}\|_{2}$ & $\text{Order}$ & $\|e_{\varphi}\|_{2}$ & $\text{Order}$\\
\midrule[0.8pt]
$1/64$ & 7.58e-3 & - & 7.46e-3  & - & 1.35e-3 & - & 1.35e-3  & -\\
$1/128$ & 2.00e-3 & 1.92 & 1.97e-3 & 1.92 & 3.27e-4 & 2.05 & 3.28e-4 & 2.04\\
$1/256$ & 4.95e-4 & 2.02 & 4.87e-4 & 2.02 & 8.07e-5 & 2.02 & 8.10e-5 &2.02 \\
$1/512$ & 1.24e-4 & 2.01 & 1.22e-4 & 2.00 & 2.01e-5 & 2.00 & 2.02e-5  &2.00\\
\bottomrule[1.2pt]
$h$ & $\|e_{\psi}\|_{\infty}$ & $\text{Order}$ & $\|e_{\psi}\|_{\infty}$ & $\text{Order}$ & $\|e_{\psi}\|_{2}$ & $\text{Order}$ & $\|e_{\psi}\|_{2}$ & $\text{Order}$\\
\midrule[0.8pt]
$1/64$ & 1.69e-2 & - & 2.04e-2  &-   & 2.11e-3 & - & 2.20e-3 &-\\
$1/128$ & 4.75e-3 & 1.83 & 5.56e-3 & 1.88 & 5.14e-4 & 2.04 & 5.43e-4 & 2.02 \\
$1/256$ & 1.20e-3 & 1.99 & 1.40e-3 & 1.98 & 1.28e-4 & 2.01 & 1.35e-4 & 2.00\\
$1/512$ & 3.00e-4 & 2.00 & 3.54e-4 & 1.99 & 3.19e-5 & 2.00 & 3.38e-5 & 2.00\\
\bottomrule[1.5pt]
\end{tabular}
\label{l_2 spatial errors}
\end{table}

\begin{table}[ht]
\centering
\caption{The $l_\infty$-errors, $l_2$-errors and order of convergence for $\varphi$ and $\psi$ in time at $t=0.5$ and $t=1$ with fixed $h=1/256$ and different time step $\tau$.}
\begin{tabular}{@{}>{\centering\arraybackslash}p{1.4cm}cccccccc@{}}
\toprule[1.5pt]
\cmidrule(lr){2-5}\cmidrule(lr){6-9}
 & \multicolumn{2}{c}{$t=0.5$} & \multicolumn{2}{c}{$t=1$} & \multicolumn{2}{c}{$t=0.5$} & \multicolumn{2}{c}{$t=1$}\\
\cmidrule(lr){2-3}\cmidrule(lr){4-5}\cmidrule(lr){6-7}\cmidrule(lr){8-9}
$\tau$ & $\|e_{\varphi}\|_{\infty}$ & $\text{Order}$  & $\|e_{\varphi}\|_{\infty}$ & $\text{Order}$ & $\|e_{\varphi}\|_{2}$ & $\text{Order}$ & $\|e_{\varphi}\|_{2}$ & $\text{Order}$\\
\midrule[0.8pt]
$1/64$ & 1.30e-3 & - & 1.47e-3  &  -& 1.35e-4 & - & 7.50e-5 & -\\
$1/128$ & 3.44e-4 & 1.91 & 3.70e-4 & 1.93 & 3.38e-5 & 2.00 & 1.91e-5 & 1.98\\
$1/256$ & 8.74e-5 & 1.98 & 9.37e-5 & 1.98 & 8.44e-6 & 2.00 & 4.79e-6 & 1.99\\
$1/512$ & 2.19e-5 & 1.99 & 2.35e-5 & 2.00  & 2.11e-6 & 2.00 & 1.20e-6  & 2.00\\
\bottomrule[1.2pt]
$\tau$ & $\|e_{\psi}\|_{\infty}$ & $\text{Order}$ & $\|e_{\psi}\|_{\infty}$ & $\text{Order}$ & $\|e_{\psi}\|_{2}$ & $\text{Order}$ & $\|e_{\psi}\|_{2}$ & $\text{Order}$\\
\midrule[0.8pt]
$1/64$ & 1.17e-2 & - & 1.05e-2  &-   & 7.90e-4 & - & 7.12e-4 &-\\
$1/128$ & 3.07e-3 & 1.97 & 2.67e-3 & 1.98 & 2.02e-4 & 1.97 & 1.82e-4 & 1.97\\
$1/256$ & 7.57e-4 & 1.99 & 6.75e-4 & 1.99 & 5.08e-5 & 1.99 & 4.57e-5 & 1.99\\
$1/512$ & 1.89e-4 & 2.00 & 1.60e-4 & 2.00 & 1.27e-5 & 2.00 & 1.14e-5 & 2.00\\
\bottomrule[1.5pt]
\end{tabular}
\label{l_2 temporal errors}
\end{table}

In the accuracy test for temporal discretization, the mesh sizes are fixed as $h=1/256$.  The numerical solutions are computed using the second-order scheme with different time steps $\tau=1/64$, $1/128$, $1/256$, $1/512$, $1/1024$ at $t=0.5$ and $t=1$. As shown in Table \ref{l_2 temporal errors}, we can see that the scheme  \eqref{second order scheme3 1}--\eqref{second order scheme3 BC3} achieve  second-order accuracy in time in the sense of the norms $l_\infty$ and $l_2$.

\subsection{Numerical validation of stability, algebraic consistency, volume conservation, and Neumann triangle condition}\label{stability_algebraic consistency_volume-conservation}
In this subsection, we numerically validate the energy stability, algebraic consistency, volume conservation and asymptotic properties such as the Neumann angle condition of the ternary DBPF model. Unless otherwise specified, the parameters for the second-order scheme are set as $\varepsilon=0.01$, $h=1/400$, $\tau=0.01$, $m_1=m_2=1e-4$, $A_1=B_1=A_2=B_2=100$.

\subsubsection{Stability}
Using the initial conditions \eqref{initial_conditions_square},  Fig. \ref{energy evolution} shows that the energies of the proposed second-order scheme \eqref{second order scheme3 1}--\ref{second order scheme3 BC3} always decay in time with different surface tension parameters $(\sigma_{23},\sigma_{12},\sigma_{13})=(1,1,1),(1,2,2),(0.6,1,0.6),(1,0.8,1.4)$. It is also remarkable that the decay rates of the total free energies show clear differences, indicating the influence in dissipation by the surface tension parameters.  
\begin{figure}[htb]
   \centering
   \includegraphics[width=0.5\linewidth]{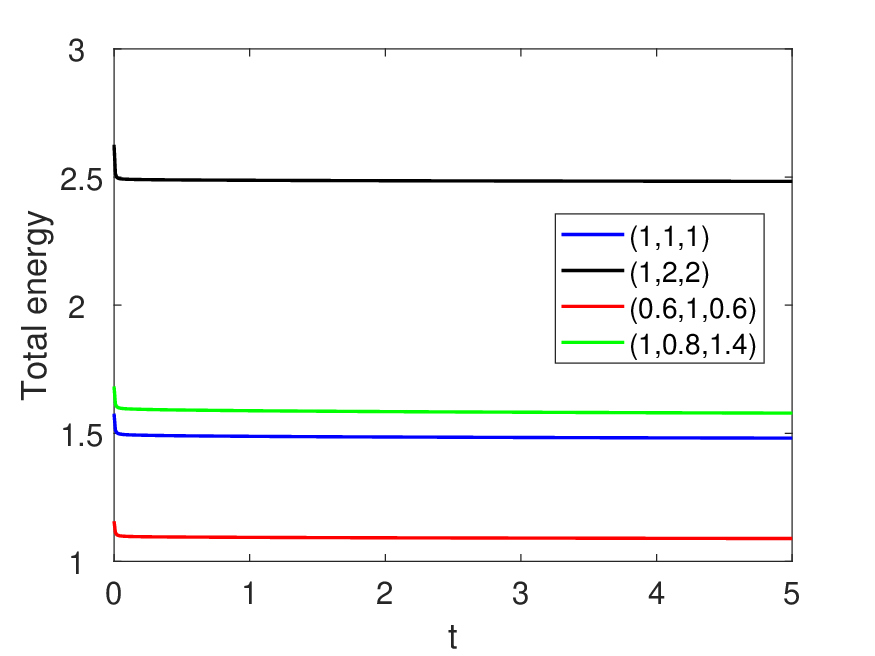}
   \caption{
   The evolution of the total energies obtained from second-order scheme with  $(\sigma_{23},\sigma_{12},\sigma_{13})=(1,1,1),(1,2,2),(0.6,1,0.6),(1,0.8,1.4)$.}\label{energy evolution}
   \end{figure}

\subsubsection{Algebraic consistency}
To validate the algebraic consistency of the ternary DBPF model, we take the initial conditions with phase 1 absent. That is,
   \begin{align*}
   \psi(x,y)=1,\qquad
   \varphi(x,y)=\tanh\bigg(\frac{\sqrt{(x-0.5)^2/1.7+(y-0.5)^2}-0.2 }{\varepsilon} \bigg).
   \end{align*}
As shown in Figs. \ref{algebraic consistency 1} (a) and (b), $\psi$ remains 1 everywhere during relaxation of $\varphi$, indicating the non-nucleation of phase 1 at the interface between phase 2 and phase 3. Finally, the interface between phase 2 and phase 3 becomes a circle. Therefore, the results obtained from the newly proposed model are physically meaningful. 
  \begin{figure}[!]
   \centering
   \begin{subfigure}{0.45\linewidth}
   \centering
   \includegraphics[trim=0cm 0cm 0cm 0cm,clip,width=\linewidth]{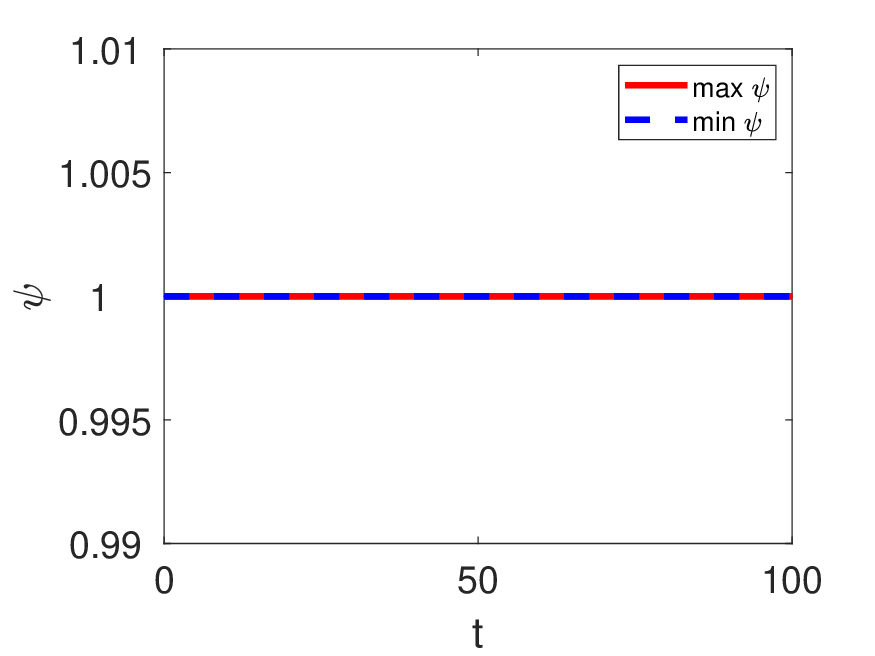}
   \subcaption{}
   \end{subfigure}
   \begin{subfigure}{0.4\linewidth}
   \centering
   \includegraphics[trim=1cm 0cm 1cm 0cm,clip,width=\linewidth]{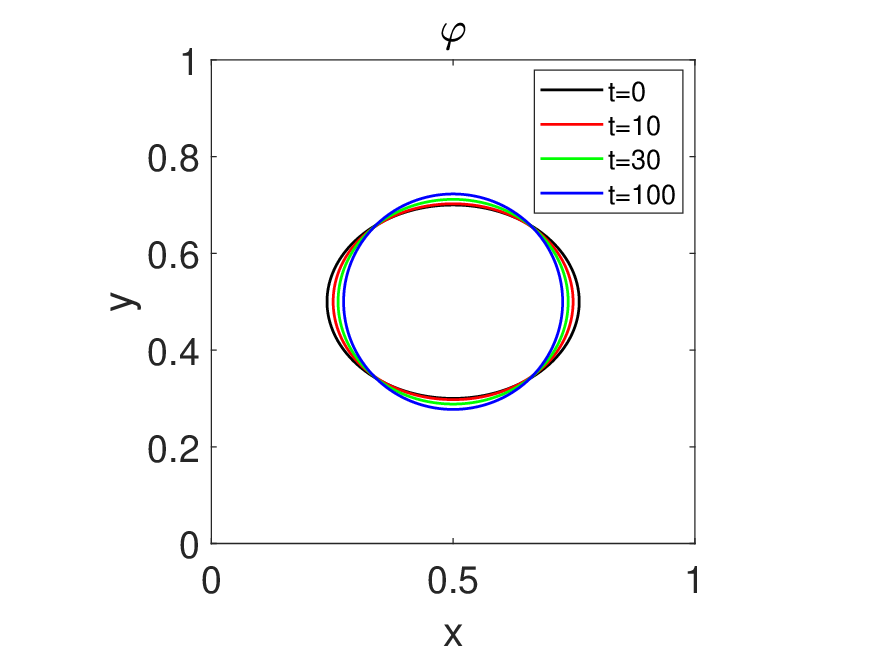}
   \subcaption{}
   \end{subfigure}
   \caption{(a) The maximum and minimum values of $\psi$ during the evolution process. (b) The interface plots of $\varphi$ at $t=$ 0, 10, 30, 100.}\label{algebraic consistency 1}
   \end{figure}
   
\subsubsection{Volume conservation}
For the DBPF model with degenerate mobility $M_2=m_2(\frac{1+\psi}{2})^8$, we numerically validate the volume conservation of each phase. The volume $V_i$ of each phase ($i=1,2,3$) is computed using the following approximation formulae:
\begin{align*}
    V_1=\int_{\Omega}\frac{1-\psi}{2}
    \mathrm{d}\mathbf{x},\qquad
    V_2=\int_{\Omega}\frac{1+\psi}{2}
    \frac{1-\varphi}{2}\mathrm{d}\mathbf{x},\qquad
    V_3=\int_{\Omega}\frac{1+\psi}{2}
    \frac{1+\varphi}{2}\mathrm{d}\mathbf{x}.
\end{align*}
The initial conditions are set as in \eqref{initial_conditions_square}. The initial volumes of each phase are 0.5, 0.25 and 0.25 respectively. As depicted in Fig. \ref{volume conservation}, the proposed second-order scheme \eqref{second order scheme3 1}--\ref{second order scheme3 BC3} of the DBPF model are able to conserve the volume of each phase for different surface tension parameters $(\sigma_{23},\sigma_{12},\sigma_{13})$.
   \begin{figure}[htb]
   \centering
   \begin{subfigure}{0.323\linewidth}
   \centering
   \includegraphics[trim=0.2cm 0.1cm 1.1cm 0cm,clip,width=\linewidth]{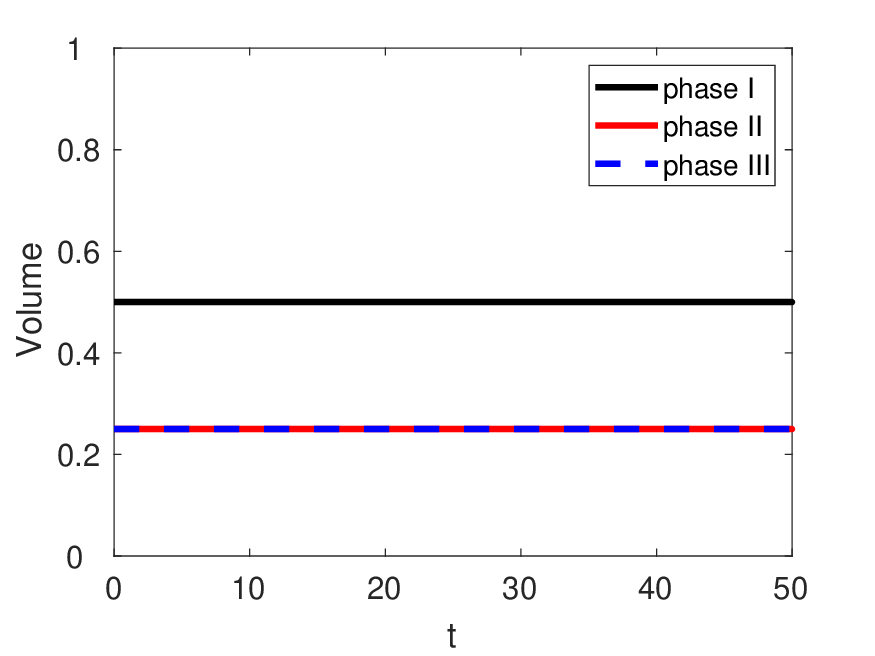}
   \subcaption{$(1,1,1)$}
   \end{subfigure}
   \begin{subfigure}{0.323\linewidth}
   \centering
   \includegraphics[trim=0.2cm 0.1cm 1.1cm 0cm,clip,width=\linewidth]{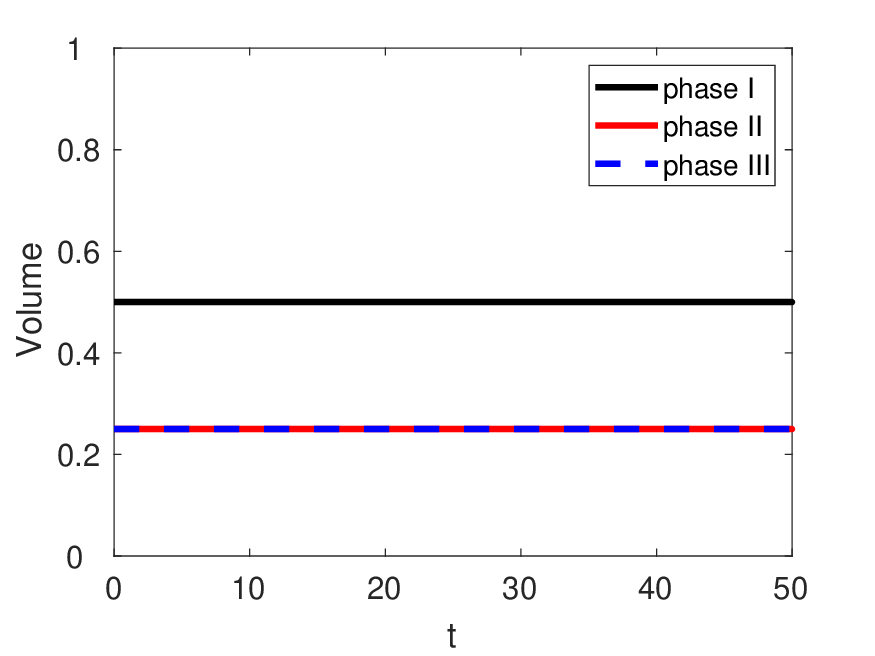}
   \subcaption{$(1,2,2)$}
   \end{subfigure}
   \begin{subfigure}{0.323\linewidth}
   \centering
   \includegraphics[trim=0.2cm 0.1cm 1.1cm 0cm,clip,width=\linewidth]{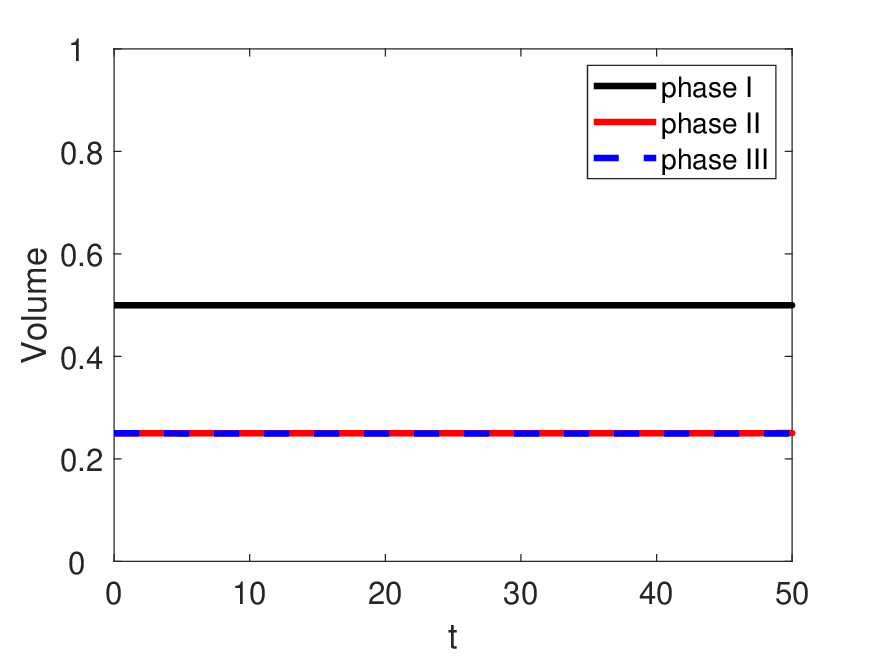}
   \subcaption{$(1,0.9,1.1)$}
   \end{subfigure}
   \caption{The volume conservation of the  second-order scheme with different surface tension parameters $(\sigma_{23},\sigma_{12},\sigma_{13})$ at $t=50$.}\label{volume conservation}
   \end{figure}
   
\subsubsection{Neumann triangle condition}\label{section: The Neumann angle condition}  
In section \ref{Sec:Sharp_interface_limit}, using the matched asymptotic analysis, we derive the Neumann triangle condition at the triple junction in equilibrium state, namely
    \begin{align*}
    \frac{\sin(\theta_{23})}{\sigma_{23}}=\frac{\sin(\theta_{12})}{\sigma_{12}}=\frac{\sin(\theta_{13})}{\sigma_{13}}\qquad \text{at}\  \Gamma_1\cap\Gamma_2\cap\Gamma_3.
    \end{align*}
To numerically validate the Neumann triangle condition at the triple junction, we apply the proposed second-order scheme \eqref{second order scheme3 1}--\eqref{second order scheme3 BC3} with different surface tension parameters $(\sigma_{23},\sigma_{12},\sigma_{13})=(0.6,1,1),(1,1,1),(1,2,2)$ to perform the numerical experiment until $t=50$ to ensure that the system reaches steady state.
\begin{figure}[!]
   \centering
   \begin{subfigure}{0.283\linewidth}
   \centering
   \includegraphics[trim=2.3cm 1cm 2.3cm 0cm,clip,width=\linewidth]{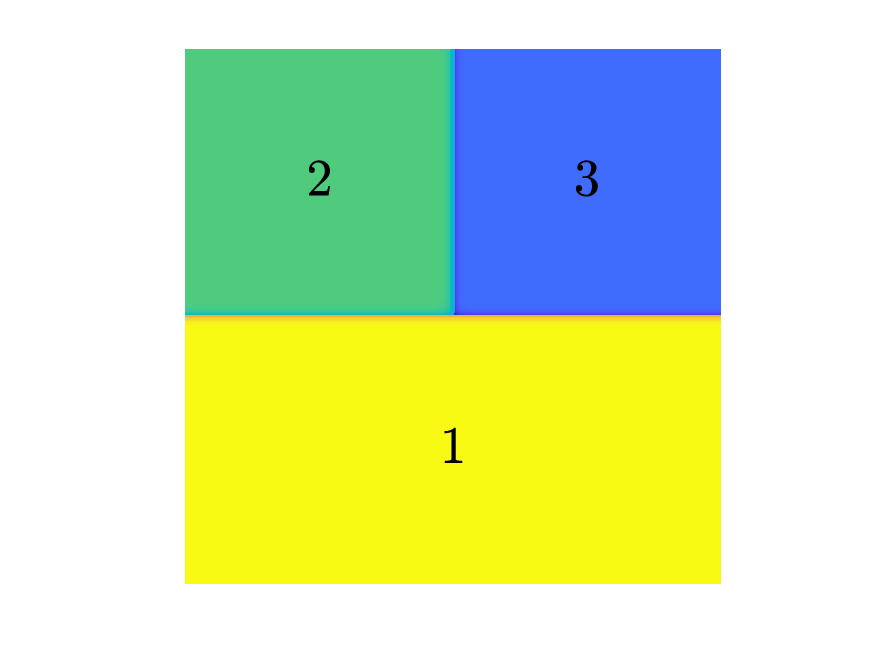}
   \subcaption{}
   \end{subfigure}
   \begin{subfigure}{0.265\linewidth}
   \centering
   \includegraphics[trim=0cm 0cm 0cm 0cm,clip,width=\linewidth]{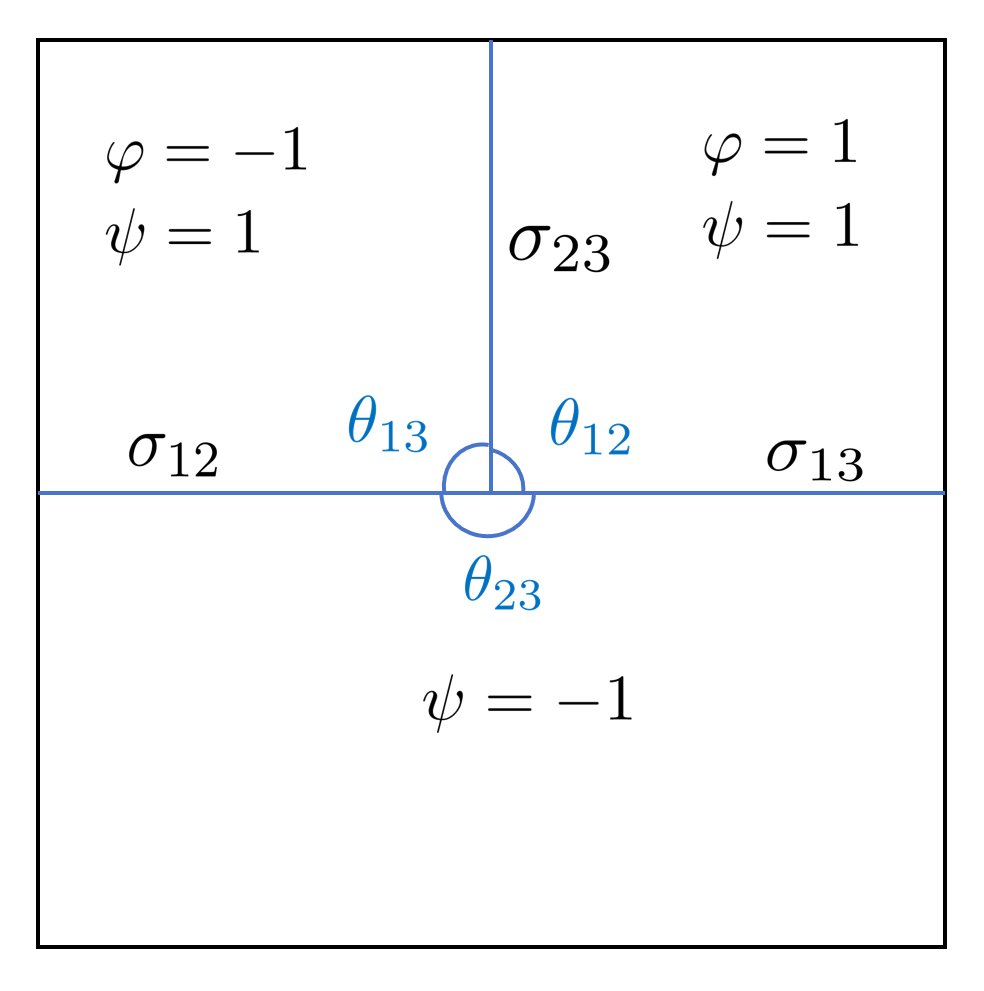}
   \subcaption{}
   \end{subfigure}
   \begin{subfigure}{0.265\linewidth}
   \centering
   \includegraphics[trim=0cm 0cm 0cm 0cm,clip,width=\linewidth]{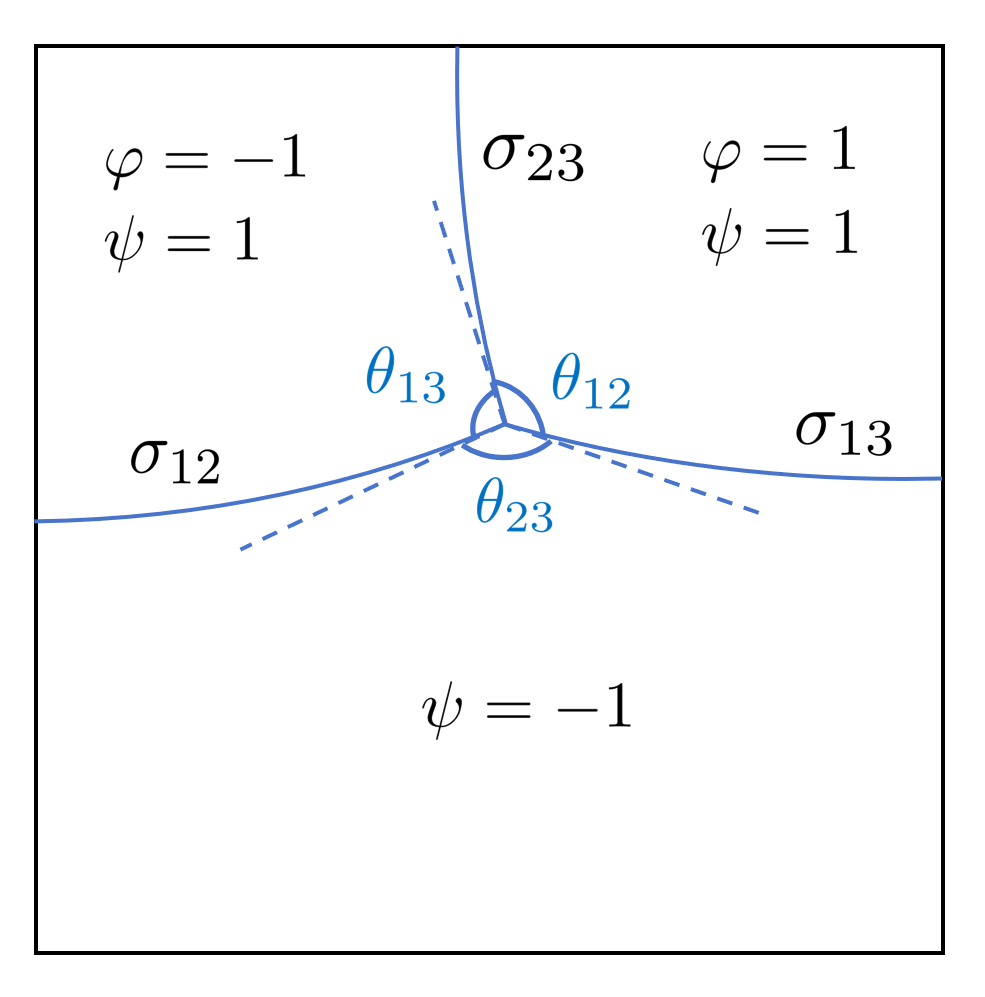}
   \subcaption{}
   \end{subfigure}
   \caption{(a): the initial configuration of phase 1 (yellow part), phase 2 (green part) and phase 3 (blue part). (b) and (c): the corresponding surface tension parameters on the interface and the contact angles between any two interfaces of the initial  and equilibrium configurations, respectively.}\label{figure of the Neumann angle}
   \end{figure}
As shown in Fig. \ref{figure of the Neumann angle} (c), we compute the apparent contact angles ($\theta_{23}$, $\theta_{12}$, $\theta_{13}$) by finding the asymptotes of three interfaces near the triple junction. It is worth noting that the three asymptotes should intersect at the same point, known as the triple junction, as $\varepsilon\to 0$. It can be seen in Table \ref{contact angles} that the computed apparent angles $(\theta_{23},\theta_{12},\theta_{13})$ and the theoretically predicted ones from the Neumann triangle condition coincide within an acceptable error range.

   \begin{table}[!]
   \centering
   \caption{The apparent contact angles at the triple junction with different surface tensions parameters.}
   \begin{tabular}{@{}>{\centering\arraybackslash}ccc}
   \toprule[1.5pt]
   {\footnotesize$(\sigma_{23},\sigma_{12},\sigma_{13})$} & $\text{numerical}\ ( \theta_{23},\theta_{12},\theta_{13})$ & $\text{theoretical}\ ( \theta_{23},\theta_{12},\theta_{13})$\\
   \midrule[0.8pt]
   $(0.6,1,1)$ & $(146.72^{\circ},106.64^{\circ},106.64^{\circ})$ & $(145.08^{\circ},107.46^{\circ},107.46^{\circ})$\\
   $(1,1,1)$ & $(122.92^{\circ},118.54^{\circ},118.54^{\circ})$ & $(120^{\circ},120^{\circ},120^{\circ})$\\
   $(1,2,2)$ & $(150.74^{\circ},104.63^{\circ},104.63^{\circ})$ & $(151.04^{\circ},104.48^{\circ},104.48^{\circ})$\\
   \bottomrule[1.5pt]
   \end{tabular}
   \label{contact angles}
   \end{table}

\subsection{Liquid lens between two stratified fluids}\label{Liquid lens between two stratified fluids}
In this example, we investigate numerically the partial spreading and total spreading of the liquid lens between two stratified fluids. The spreading coefficient of phase $i$ is defined by $S_i=\sigma_{ij}+\sigma_{ik}-\sigma_{jk},(i=1,2,3)$ \cite{boyer2006study, boyer2011numerical}, which indicates the partial spreading if $S_i>0$ for all $i$ and the total spreading if $S_i<0$ for some $i$. We label the regions inside and outside the droplets by $\{\psi=-1\}$ and $\{\psi=1\}$, respectively. Then, we introduce $\{\varphi=-1\}$ and $\{\varphi=1\}$ to characterize the lower and upper fluids. In summary, the initial condition is set as follows:
   \begin{align*}
   \psi(x,y)=\tanh\Big(\frac{\sqrt{(x-0.5)^2+(y-0.5)^2}-0.15 }{\varepsilon} \Big),\quad
   \varphi(x,y)=\tanh\Big(\frac{y-0.5}{\varepsilon}\Big),
   \end{align*}
which indicates that the liquid lens is initially spherical and is located between the other two phases.  The parameters of this example are set as $\varepsilon=0.02$, $h=1/400$, $\tau=0.01$, $m_1=m_2=1e-3$, $A_1=B_1=A_2=B_2=100$.

   \begin{figure}[htp]
   \begin{subfigure}{0.23\linewidth}
   \centering
   \includegraphics[trim=3.3cm 0cm 2.7cm 0cm,clip,width=1\linewidth]{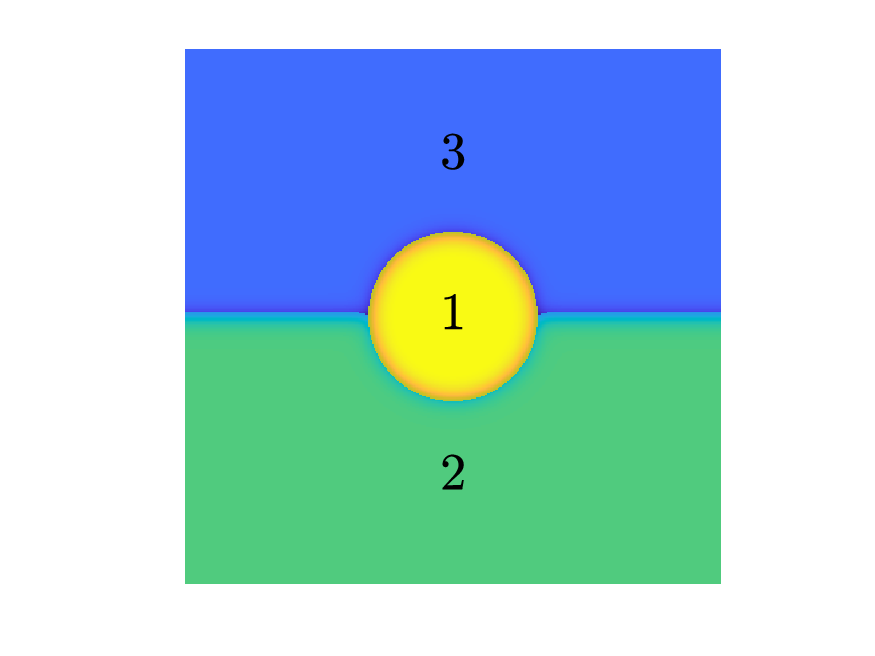}
   \end{subfigure}
   \begin{subfigure}{0.23\linewidth}
   \centering
   \includegraphics[trim=3.3cm 0cm 2.7cm 0cm,clip,width=1\linewidth]{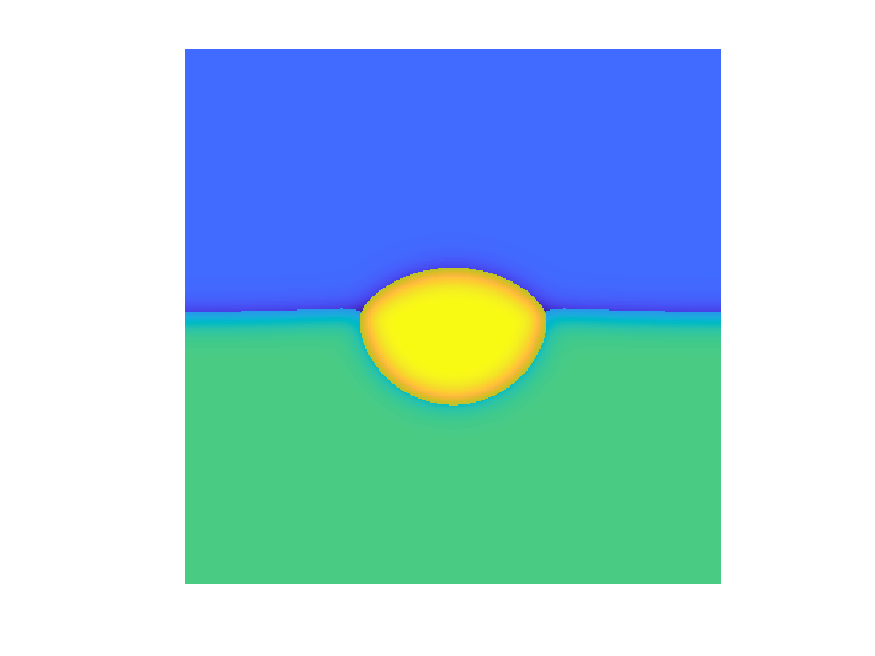}
   \end{subfigure}
   \centering
   \begin{subfigure}{0.23\linewidth}
   \centering
   \includegraphics[trim=3.3cm 0cm 2.7cm 0cm,clip,width=1\linewidth]{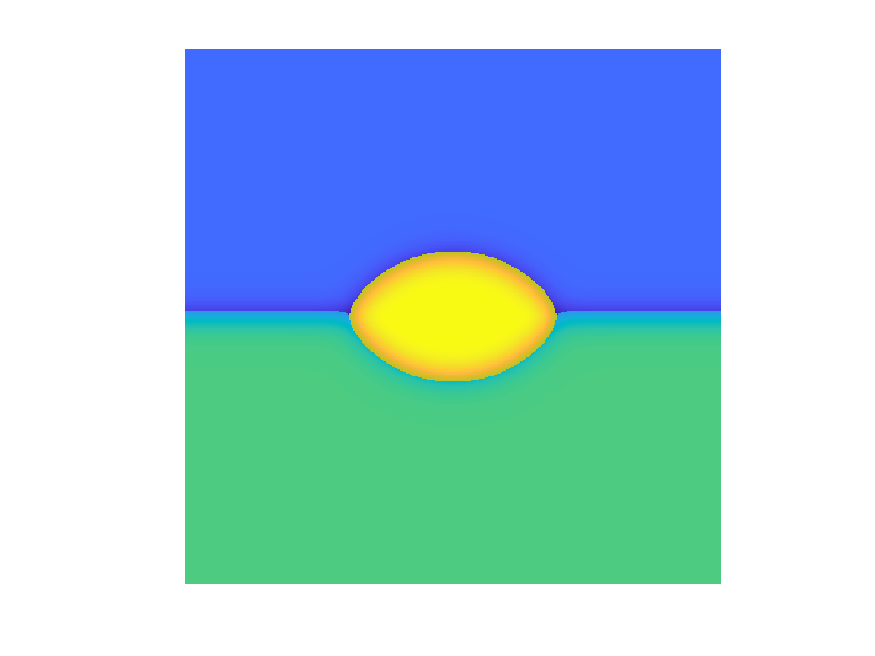}
   \end{subfigure}
   \begin{subfigure}{0.23\linewidth}
   \centering
   \includegraphics[trim=3.3cm 0cm 2.7cm 0cm,clip,width=1\linewidth]{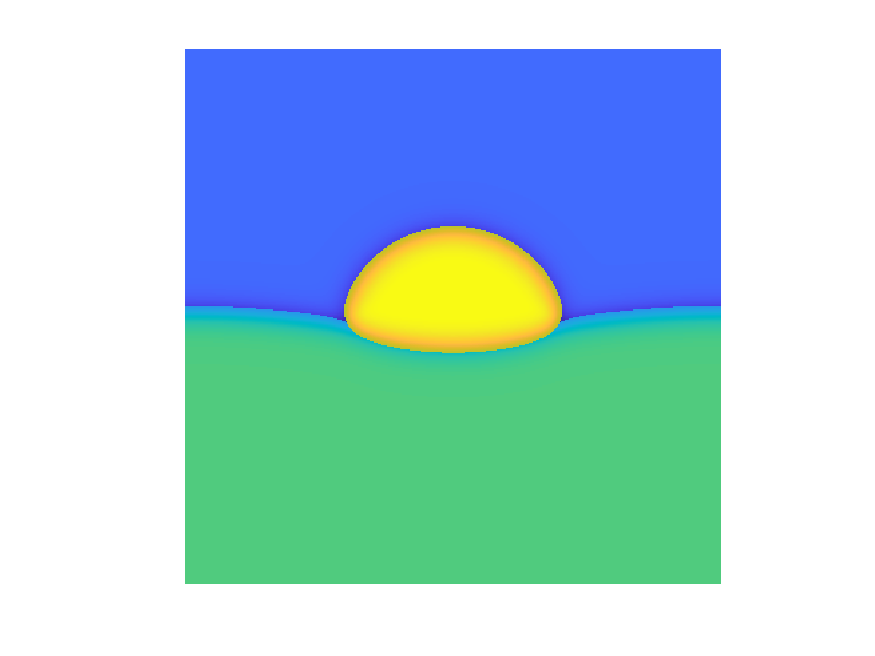}
   \end{subfigure}
   \caption{The initial configuration and  the equilibrium configuration of the liquid lens with different surface tension parameters $(\sigma_{23},\sigma_{12},\sigma_{13})=(1,1,1.4)$, $(1,1,1), (1,1,0.2)$ at $t=50$.}\label{The dynamical behaviors of initial condition 2}
   \end{figure}

As illustrated in Remark \ref{remark_triangle_law}, in the partial spreading case with $S_i>0$ for all $i$, a closed force triangle can be formed and thus the force balances hold at triple junctions when the system reaches steady state. As shown in Fig. \ref{The dynamical behaviors of initial condition 2}, in the three partial spreading cases with $(\sigma_{23},\sigma_{12},\sigma_{13})=(1,1,1.4),(1,1,1),(1,1,0.2)$, the yellow drop may shift downwards, stay in the middle, or shift upwards, due to the unbalanced forces at triple junction points. It can be observed that equilibrium states always exist with Neumann triangle condition held at the two triple junctions, whereas the interfaces become circular eventually. The total interfacial area weighted by surface tension coefficients is minimized in equilibrium.
\begin{figure}[htp]
   \hspace{2mm}
   \begin{subfigure}{0.23\linewidth}
   \centering
   \includegraphics[trim=3.5cm 1cm 2.5cm 1cm,clip,width=1\linewidth]{liquid_lens/311_initial0.eps}
   \end{subfigure}
   \begin{subfigure}{0.23\linewidth}
   \centering
   \includegraphics[trim=3.5cm 1cm 2.5cm 1cm,clip,width=1\linewidth]{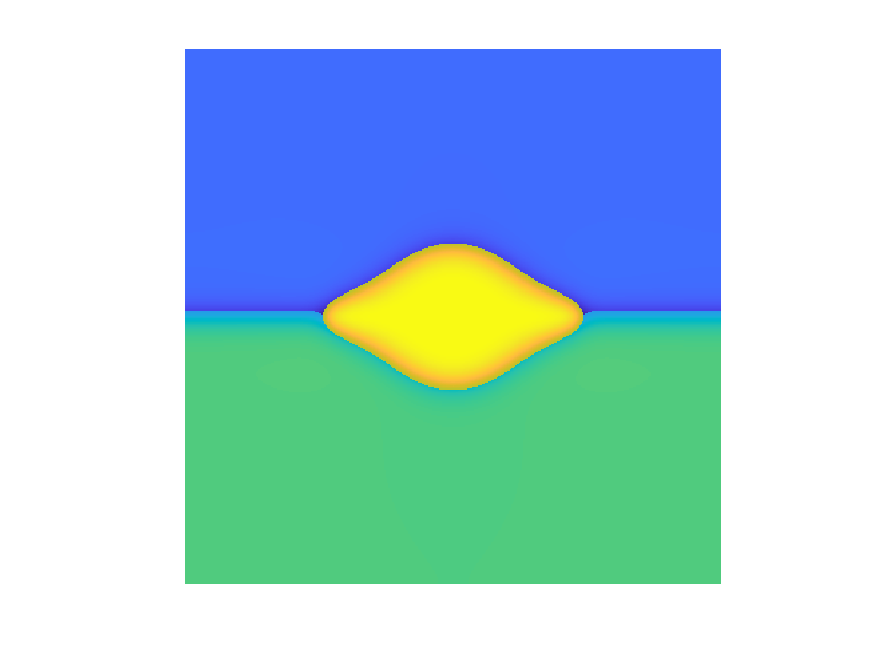}
   \end{subfigure}
   \centering
   \begin{subfigure}{0.23\linewidth}
   \centering
   \includegraphics[trim=3.5cm 1cm 2.5cm 1cm,clip,width=1\linewidth]{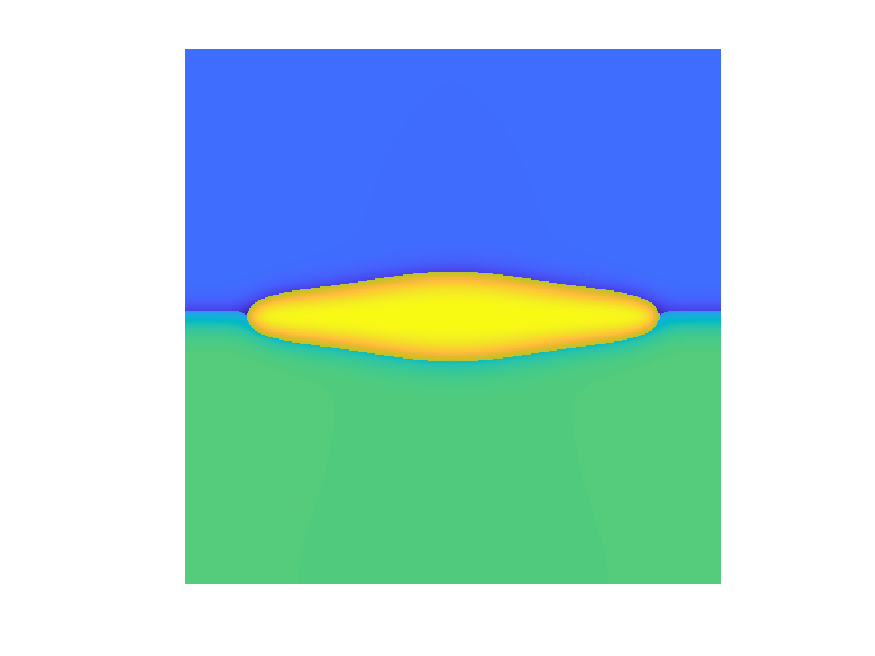}
   \end{subfigure}
   \begin{subfigure}{0.23\linewidth}
   \centering
   \includegraphics[trim=3.5cm 1cm 2.5cm 1cm,clip,width=1\linewidth]{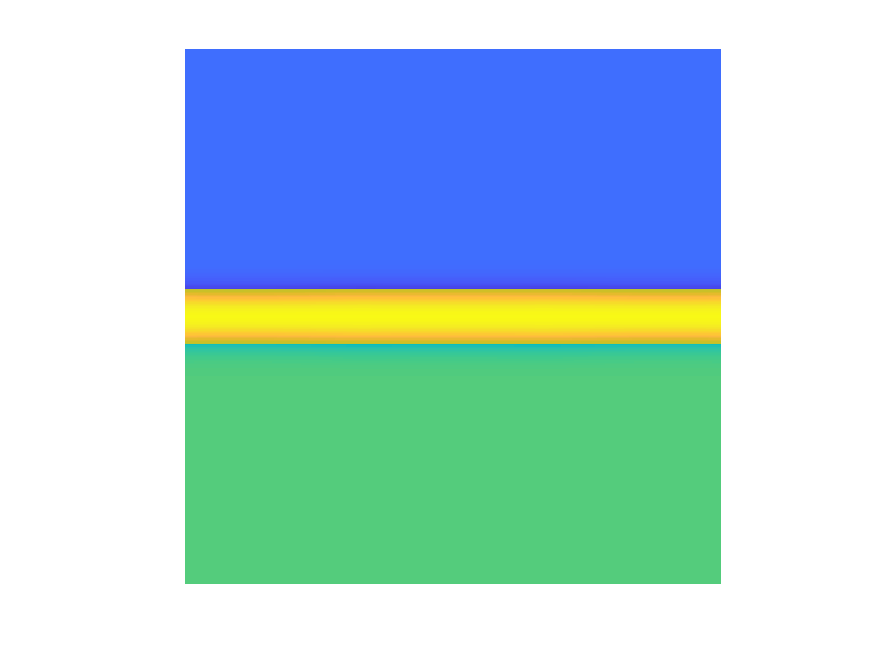}
   \end{subfigure}
   \vspace{-9pt}
   \caption*{$(a)\ (\sigma_{23},\sigma_{12},\sigma_{13})=(3,1,1)$.}

   \hspace{2mm}
   \begin{subfigure}{0.23\linewidth}
   \centering
   \includegraphics[trim=3.5cm 1cm 2.5cm 1cm,clip,width=1\linewidth]{liquid_lens/311_initial0.eps}
   \end{subfigure}
   \begin{subfigure}{0.23\linewidth}
   \centering
   \includegraphics[trim=3.5cm 1cm 2.5cm 1cm,clip,width=1\linewidth]{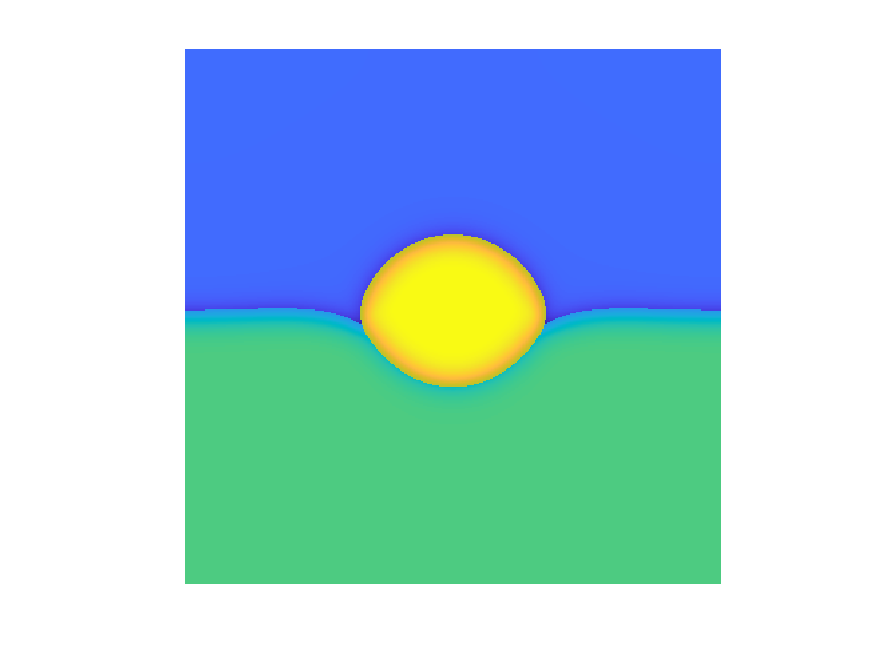}
   \end{subfigure}
   \centering
   \begin{subfigure}{0.23\linewidth}
   \centering
   \includegraphics[trim=3.5cm 1cm 2.5cm 1cm,clip,width=1\linewidth]{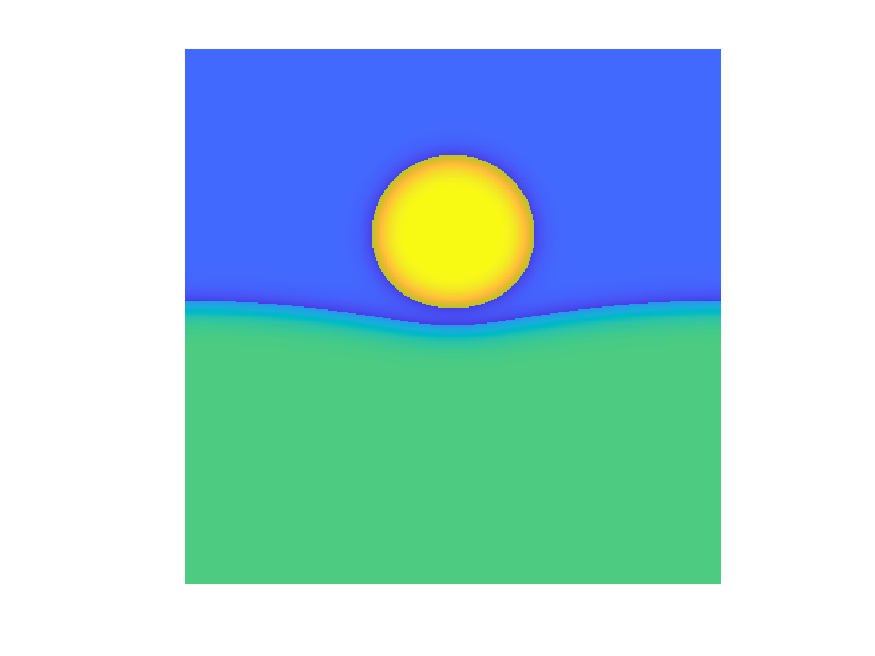}
   \end{subfigure}
   \begin{subfigure}{0.23\linewidth}
   \centering
   \includegraphics[trim=3.5cm 1cm 2.5cm 1cm,clip,width=1\linewidth]{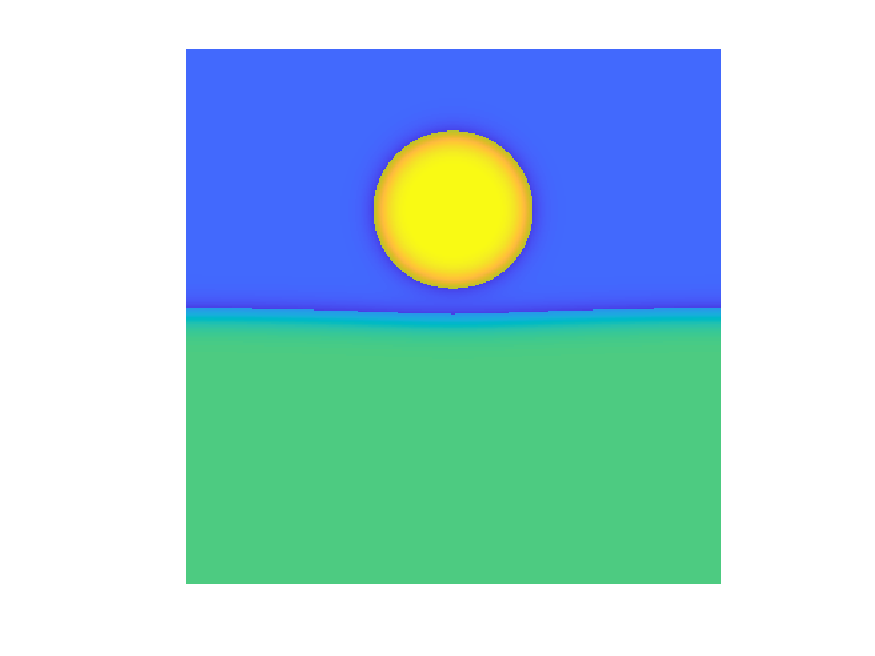}
   \end{subfigure}
   \vspace{-9pt}
   \caption*{$(b)\ (\sigma_{23},\sigma_{12},\sigma_{13})=(1,3,1)$.}
   \caption{The initial configuration (1st colomn) and the dynamic behaviors of the liquid lens for the total spreading case $(\sigma_{23},\sigma_{12},\sigma_{13})=(3,1,1)$ at $t=0.5, 2.5, 50$ and $(\sigma_{23},\sigma_{12},\sigma_{13})=(1,3,1)$ at $t=0.5, 15, 50$.}\label{The dynamical behaviors of initial condition 3}
   \end{figure}
   
In contrast, in the total spreading case $(\sigma_{23},\sigma_{12},\sigma_{13})=(3,1,1)$ or $(1,3,1)$,
Remark \ref{remark_triangle_law} implies that no closed force triangle can be formed so that triple junctions can not be stable in equilibrium. It can be observed that in both cases of Fig. \ref{The dynamical behaviors of initial condition 3}, there are no three-phase intersections in equilibrium. When $(\sigma_{23},\sigma_{12},\sigma_{13})=(3,1,1)$, the total forces at the two triple junctions point outwards, leading to the stretching of the yellow drop until its complete spreading. In equilibrium, a configuration with sandwich layers occurs. 
When $(\sigma_{23},\sigma_{12},\sigma_{13})=(1,3,1)$, the total forces at the two triple junctions point upwards, yielding a lift of the yellow drop. We eventually see a separation of the yellow drop from the blue-green liquid interface.

\subsection{Interaction of two droplets in contact}\label{Evolution of two close-by droplets}
In this example, we numerically investigate the interaction of two droplets in contact. The two equal-sized droplets with radius $0.15$ are centered at $(0.35,0.5)$ and $(0.65,0.5)$. We label the two droplets by $\{\psi=1\}$ and the region outside it by $\{\psi=-1\}$, namely,
\begin{align*}
   \psi(x,y)=&~1-\tanh\Big(\frac{\sqrt{(x-0.65)^2+(y-0.5)^2}-0.15 }{\varepsilon}\Big)\notag\\
   &-\tanh\Big(\frac{\sqrt{(x-0.35)^2+(y-0.5)^2}-0.15 }{\varepsilon}\Big).
\end{align*}
The left and right droplets are distinguished by $\{\varphi=-1\}$ and $\{\varphi=1\}$ respectively. It is remarkable that the representation of such a configuration using $\varphi$ is not unique. For example, one can have the following alternatives:
\begin{align*}
   \varphi(x,y)=\tanh\Big(\frac{x-0.5}{\varepsilon} \Big),\quad\mbox{or}\quad
   \tanh\Big(\frac{0.15-\sqrt{(x-0.65)^2+(y-0.5)^2} }{\varepsilon}\Big).
\end{align*}
Due to the nonconvexity of the energy functional \eqref{eq:ternary_energy}, the final state as a result of the relaxation dynamics \eqref{CH11}--\eqref{CH44} depends on the choice of initial condition. In this numerical example, we are interested in stable final states which are physically meaningful. For this purpose, we will tune our initial conditions for different cases of surface tension parameters to seek such final states. 
The model parameters in the simulation are set as $\varepsilon=0.01$, $h=1/400$, $\tau=0.01$, $m_1=m_2=1e-3$, $A_1=B_1=A_2=B_2=100$.
   \begin{figure}[htp]
   \begin{subfigure}{0.23\linewidth}
   \centering
   \includegraphics[trim=3.5cm 1cm 2.5cm 1cm,clip,width=1\linewidth]{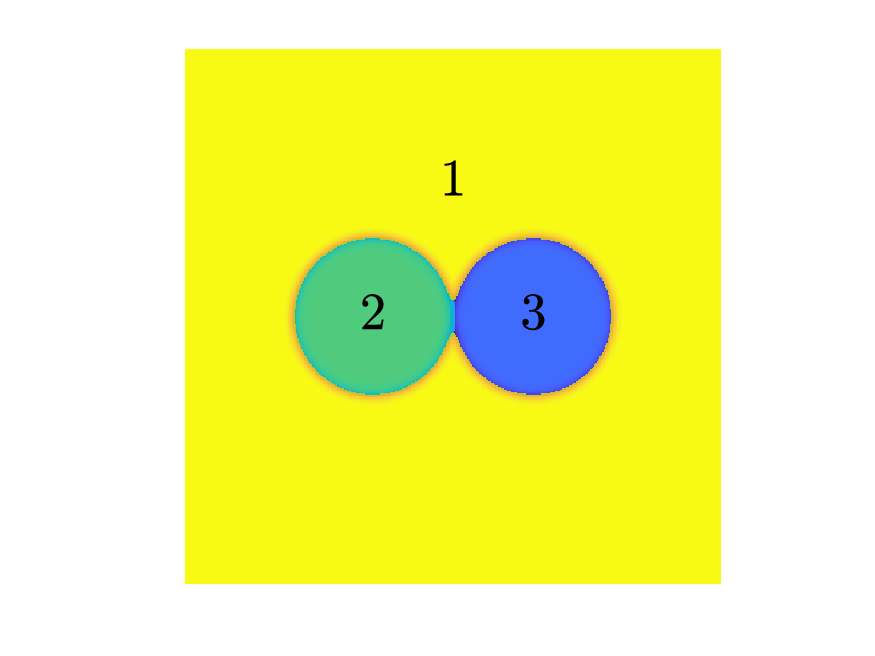}
   \end{subfigure}
   \begin{subfigure}{0.23\linewidth}
   \centering
   \includegraphics[trim=3.5cm 1cm 2.5cm 1cm,clip,width=1\linewidth]{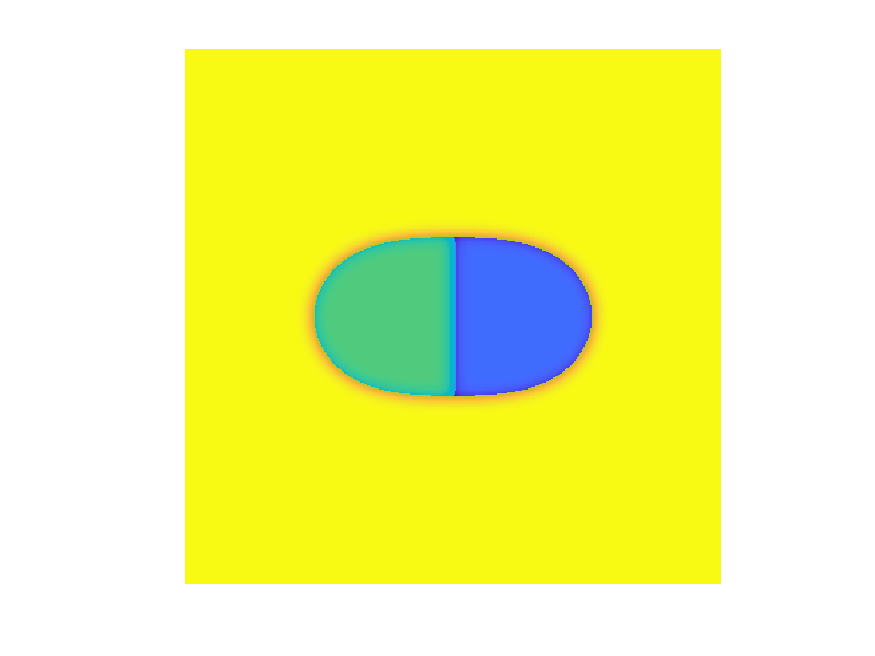}
   \end{subfigure}
   \centering
   \begin{subfigure}{0.23\linewidth}
   \centering
   \includegraphics[trim=3.5cm 1cm 2.5cm 1cm,clip,width=1\linewidth]{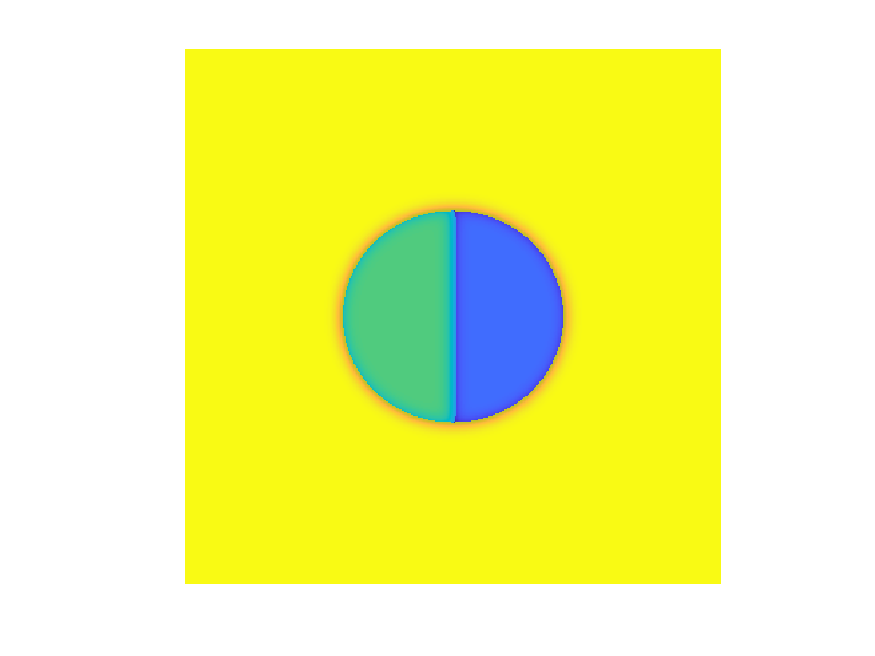}
   \end{subfigure}
   \begin{subfigure}{0.23\linewidth}
   \centering
   \includegraphics[trim=3.5cm 1cm 2.5cm 1cm,clip,width=1\linewidth]{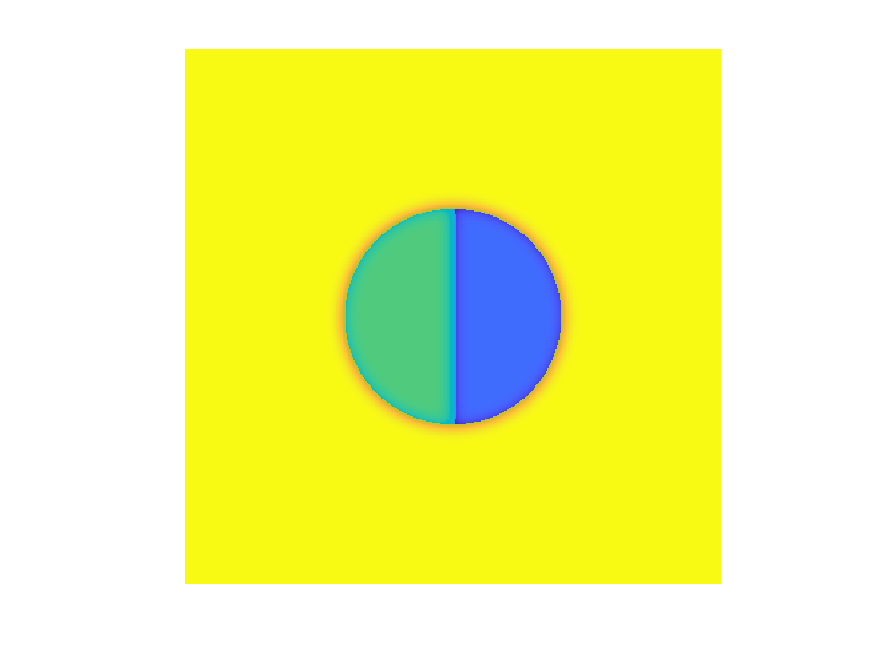}
   \end{subfigure}
   \vspace{-9pt}
   \caption*{$(a)\ (\sigma_{23},\sigma_{12},\sigma_{13})=(1,100,100)$.}
   
   \begin{subfigure}{0.23\linewidth}
   \centering
   \includegraphics[trim=3.5cm 1cm 2.5cm 1cm,clip,width=1\linewidth]{two_close_circles/two_close_circles_t0.eps}
   \end{subfigure}
   \begin{subfigure}{0.23\linewidth}
   \centering
   \includegraphics[trim=3.5cm 1cm 2.5cm 1cm,clip,width=1\linewidth]{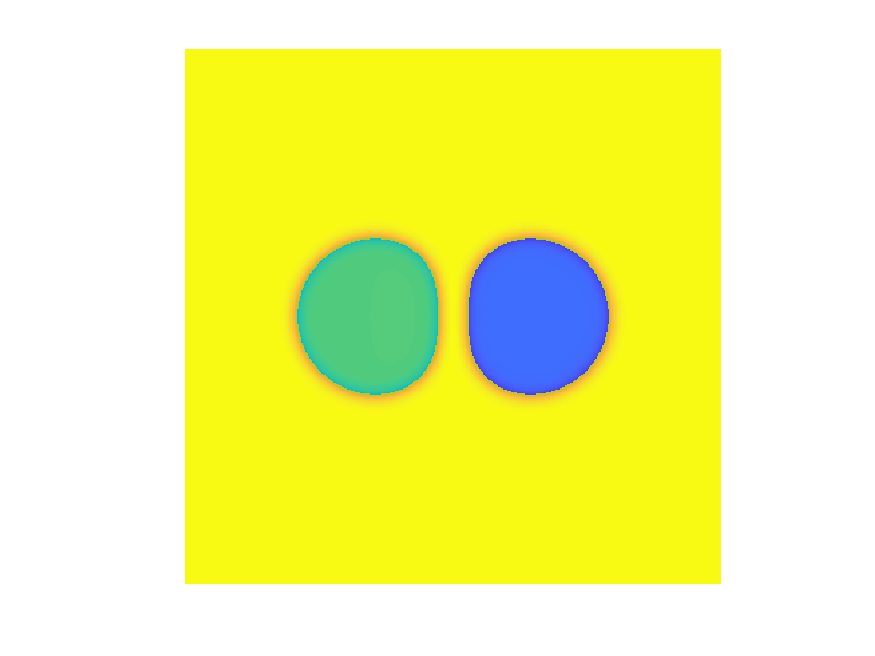}
   \end{subfigure}
   \centering
   \begin{subfigure}{0.23\linewidth}
   \centering
   \includegraphics[trim=3.5cm 1cm 2.5cm 1cm,clip,width=1\linewidth]{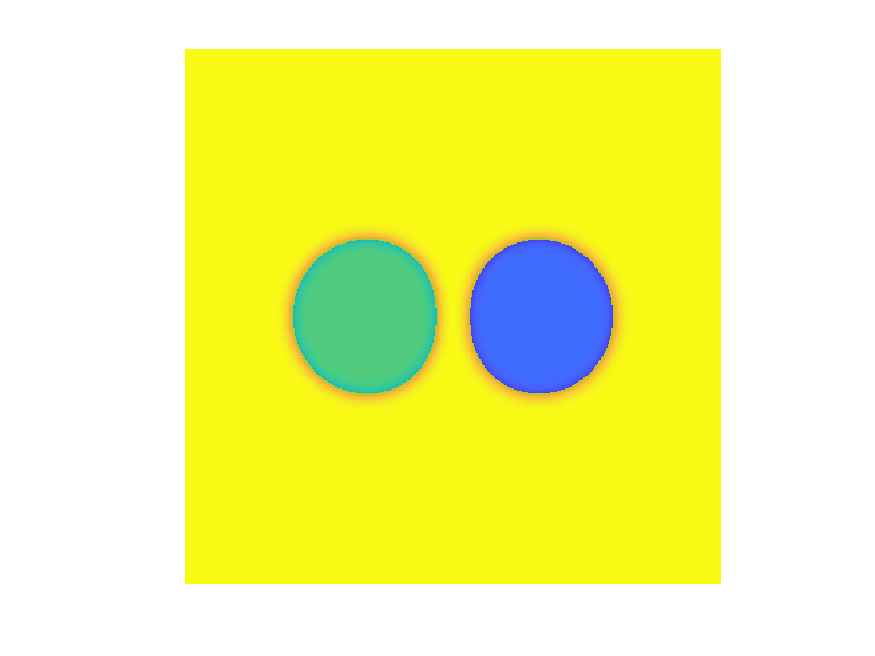}
   \end{subfigure}
   \begin{subfigure}{0.23\linewidth}
   \centering
   \includegraphics[trim=3.5cm 1cm 2.5cm 1cm,clip,width=1\linewidth]{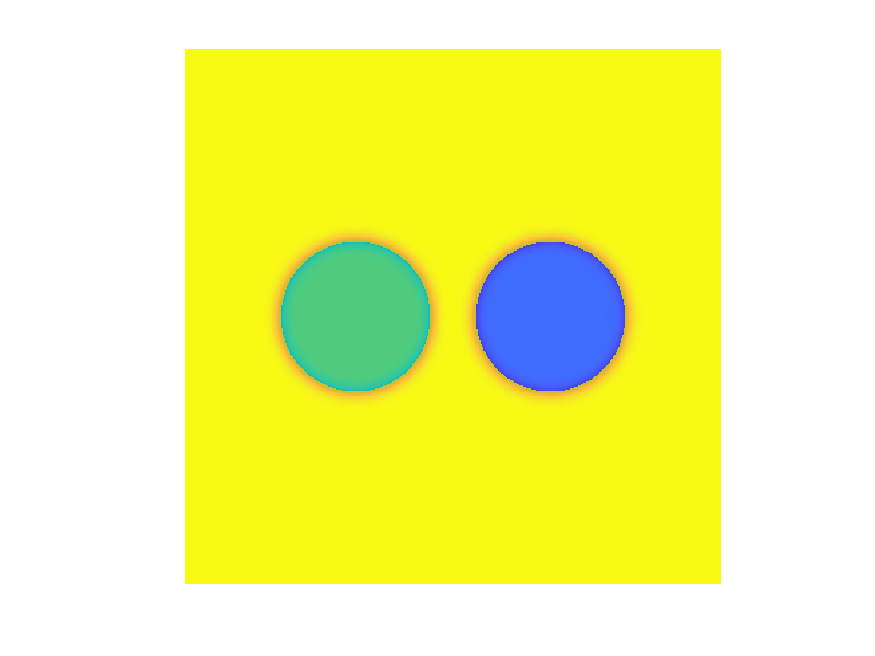}
   \end{subfigure}
   \vspace{-9pt}
   \caption*{$(b)\ (\sigma_{23},\sigma_{12},\sigma_{13})=(3,1,1)$.}
   
   \begin{subfigure}{0.23\linewidth}
   \centering
   \includegraphics[trim=3.5cm 1cm 2.5cm 1cm,clip,width=1\linewidth]{two_close_circles/two_close_circles_t0.eps}
   \end{subfigure}
   \begin{subfigure}{0.23\linewidth}
   \centering
   \includegraphics[trim=3.5cm 1cm 2.5cm 1cm,clip,width=1\linewidth]{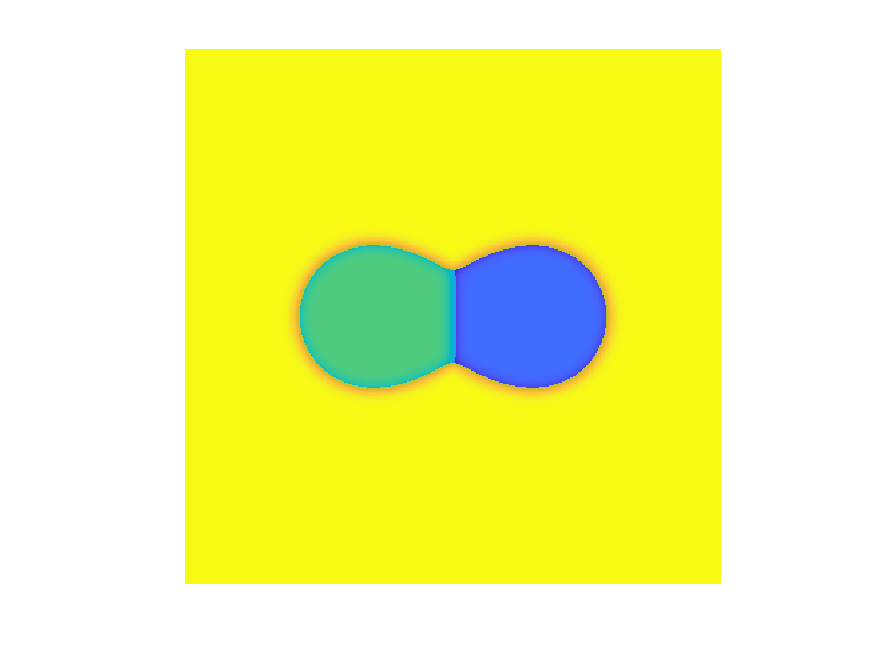}
   \end{subfigure}
   \centering
   \begin{subfigure}{0.23\linewidth}
   \centering
   \includegraphics[trim=3.5cm 1cm 2.5cm 1cm,clip,width=1\linewidth]{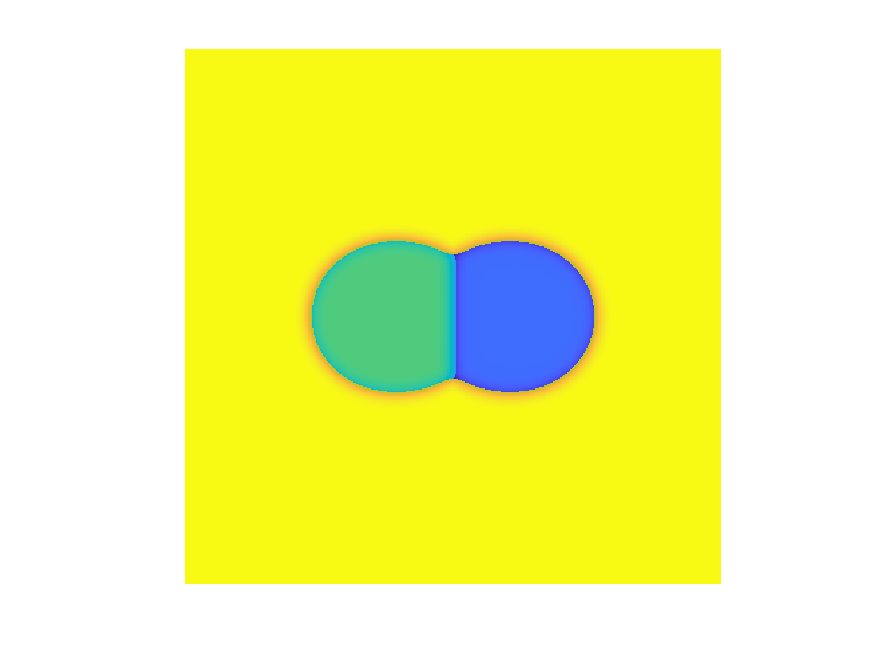}
   \end{subfigure}
   \begin{subfigure}{0.23\linewidth}
   \centering
   \includegraphics[trim=3.5cm 1cm 2.5cm 1cm,clip,width=1\linewidth]{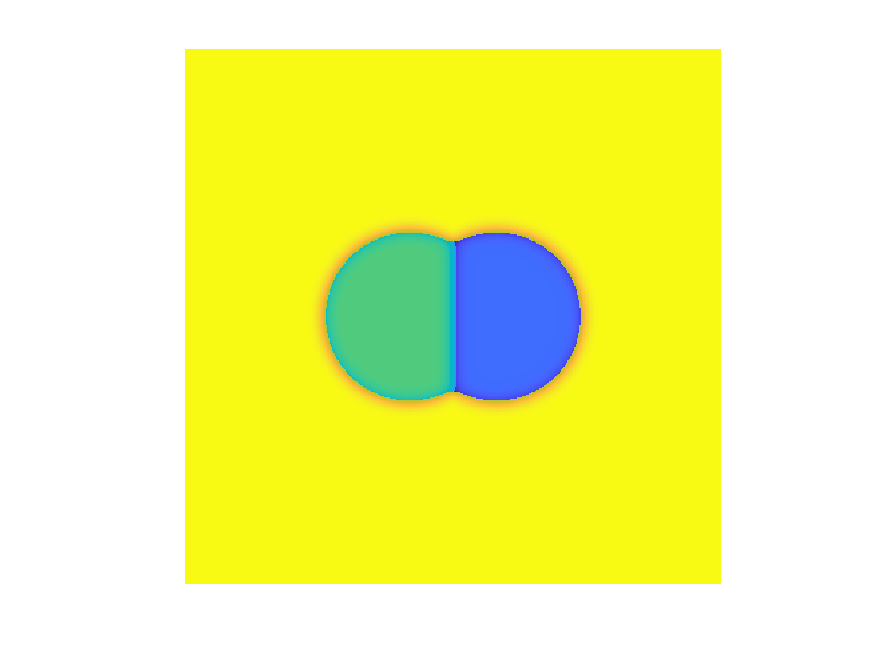}
   \end{subfigure}
   \vspace{-9pt}
   \caption*{$(c)\ (\sigma_{23},\sigma_{12},\sigma_{13})=(1,1,1)$.}
   \caption{The initial configuration (1st column) and long time behavior of two close-by droplets with three different surface tension parameters: $(\sigma_{23},\sigma_{12},\sigma_{13})=(1,100,100), (3,1,1), (1,1,1)$ at $t=1$ (2nd column), $t=5$ (3rd column), $t=50$ (4th column).}\label{two circumscribed droplets 1}
   \end{figure}

We first consider formation of symmetric compound droplets by taking $(\sigma_{23},\sigma_{12},\sigma_{13})=(1,100,100), (3,1,1)$ or $(1,1,1)$. To this end, we naturally choose an initial condition of $\varphi$ with symmetric separation:
\begin{align*}   \varphi(x,y)=&~\tanh\Big(\frac{x-0.5}{\varepsilon} \Big).
\end{align*}
Due to the finite interface thickness in phase-field model, the two droplets initially get in touch with overlapping interfacial region in the middle.   
For the partial spreading cases $(\sigma_{23},\sigma_{12},\sigma_{13})=(1,100,100)$ or $(1,1,1)$, triple junctions can exist as stable equilibrium. The total forces acting at the overlapping region of the two droplets stretch the compound droplet upwards and downwards in the vertical direction. Eventually a symmetric compound droplet is formed, with nearly-circular shape (Fig. \ref{two circumscribed droplets 1} (a)) or trisected $120^\circ$ contact angles at the triple junctions (Fig. \ref{two circumscribed droplets 1} (c)). 
Such final configurations exhibit similar features as Janus emulsions, which are bi-phasic droplets containing two immiscible liquid phases in distinct hemispheres within a continuous medium \cite{Xu2021}. 
For the total spreading case $(\sigma_{23},\sigma_{12},\sigma_{13})=(3,1,1)$, triple junctions are not stable. The unbalanced forces at the initial overlapping region attempt to separate the compound droplet, making the two droplets fall apart. Then they get away from each other like a repulsive dipole until their pairwise interaction becomes sufficiently weak (Fig. \ref{two circumscribed droplets 1} (b)). 
The final two-eye configuration is an example of multi-core emulsions \cite{Nisisako2005}, which are colloidal systems where multiple discrete inner droplets are encapsulated within a single outer droplet.
   
We then turn to an asymmetric case with $(\sigma_{23},\sigma_{12},\sigma_{13})=(1,1,3)$, and choose the initial condition of $\varphi$ to separate the blue droplet from the rest:
\begin{align*}   \varphi(x,y)=&~\tanh\Big(\frac{0.15-\sqrt{(x-0.65)^2+(y-0.5)^2} }{\varepsilon}\Big).
   \end{align*} 
As shown in Fig. \ref{two circumscribed droplets 2}, the unbalanced forces drag the green drop to the right, coating on the surface of the blue drop. 
The final configuration exhibits an annulus shape, forming a so-called double emulsion \cite{Chao2016}. 
The numerical final configurations of compound drops qualitatively reproduce the three types of emulsions from experiments, as demonstrated in Fig. \ref{experimental works}.
    \begin{figure}[htp]
   
   \begin{subfigure}{0.23\linewidth}
   \centering
   \includegraphics[trim=3.5cm 1cm 2.5cm 1cm,clip,width=1\linewidth]{two_close_circles/two_close_circles_t0.eps}
   \end{subfigure}
   \begin{subfigure}{0.23\linewidth}
   \centering
   \includegraphics[trim=3.5cm 1cm 2.5cm 1cm,clip,width=1\linewidth]{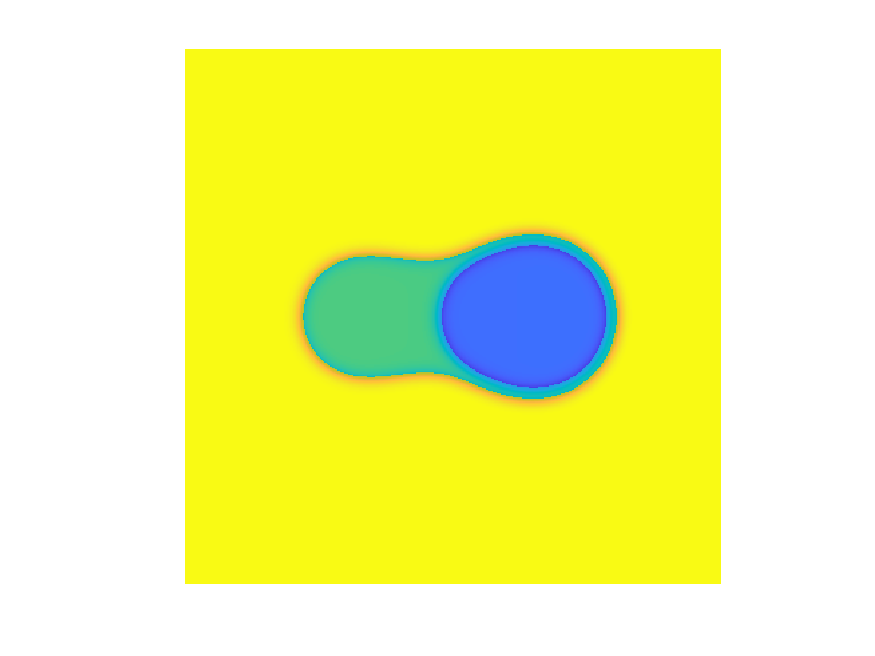}
   \end{subfigure}
   \centering
   \begin{subfigure}{0.23\linewidth}
   \centering
   \includegraphics[trim=3.5cm 1cm 2.5cm 1cm,clip,width=1\linewidth]{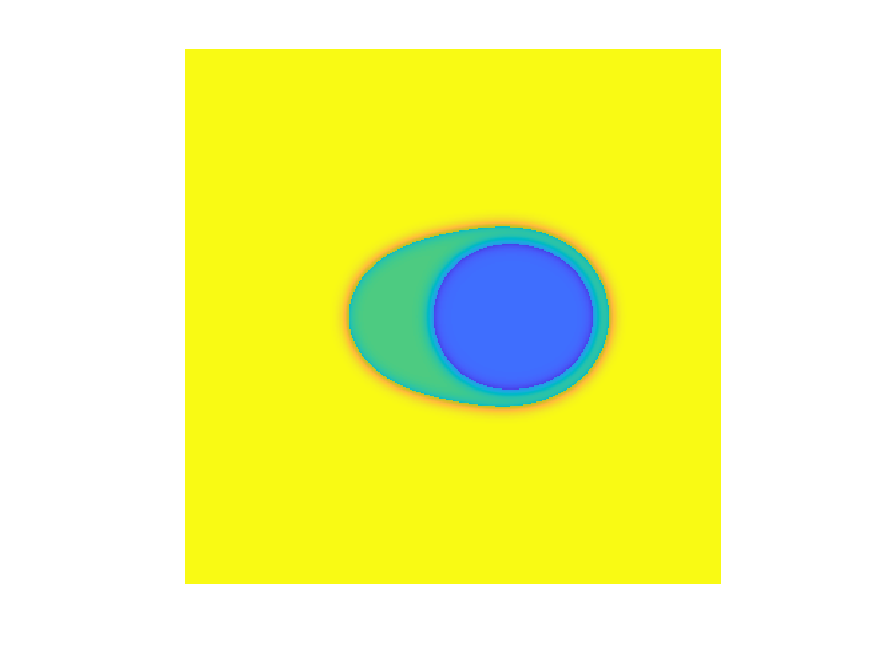}
   \end{subfigure}
   \begin{subfigure}{0.23\linewidth}
   \centering
   \includegraphics[trim=3.5cm 1cm 2.5cm 1cm,clip,width=1\linewidth]{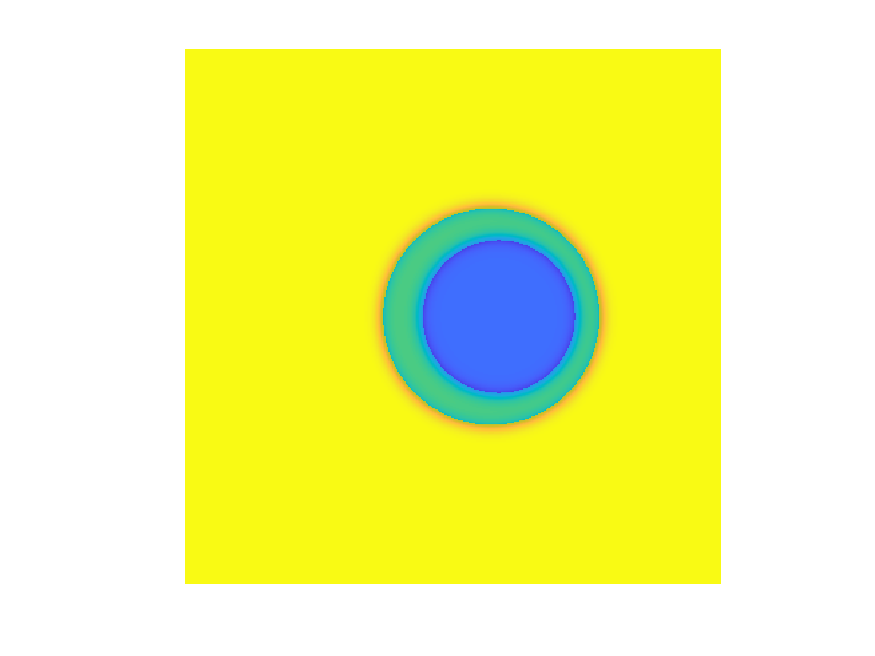}
   \end{subfigure}
   \vspace{-9pt}
   \caption{The initial configuration (1st column) and long time behavior of two close-by droplets with surface tension parameters: $(\sigma_{23},\sigma_{12},\sigma_{13})=(1,1,3)$ at $t=1$ (2nd column), $t=5$ (3rd column), $t=50$ (4th column).}\label{two circumscribed droplets 2}
   \end{figure}
   
   \begin{figure}[htp]
   \centering
   \begin{subfigure}{0.25\linewidth}
   \centering
   \includegraphics[trim=3.5cm 1cm 2.5cm 1cm,clip,width=1\linewidth]{two_close_circles/311_t50.eps}
   \end{subfigure}
   \begin{subfigure}{0.25\linewidth}
   \centering
   \includegraphics[trim=3.5cm 1cm 2.5cm 1cm,clip,width=1\linewidth]{two_close_circles/111_t50.eps}
   \end{subfigure}
   \begin{subfigure}{0.25\linewidth}
   \centering
   \includegraphics[trim=3.5cm 1cm 2.5cm 1cm,clip,width=1\linewidth]{two_close_circles/131_2_t50.eps}
   \end{subfigure}

   \centering
   \begin{subfigure}{0.25\linewidth}
   \centering
   \includegraphics[trim=0cm 0cm 0cm 0.23cm,clip,width=1\linewidth]{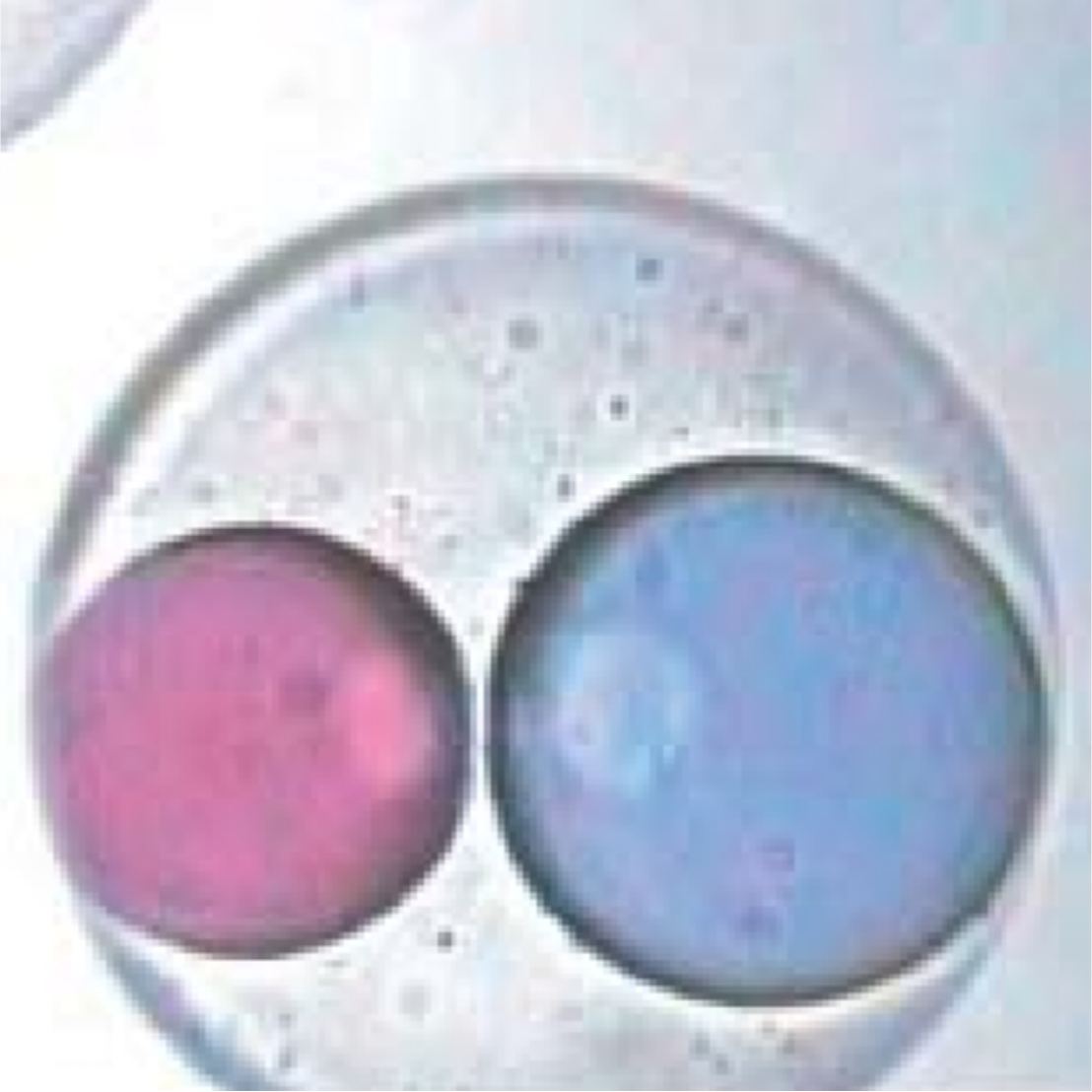}
   \caption*{(Nisisako et al, Soft Matter 2005 \cite{Nisisako2005})}
   \end{subfigure}
   \begin{subfigure}{0.25\linewidth}
   \centering
   \includegraphics[trim=0cm 0cm 0cm 0.23cm,clip,width=1\linewidth]{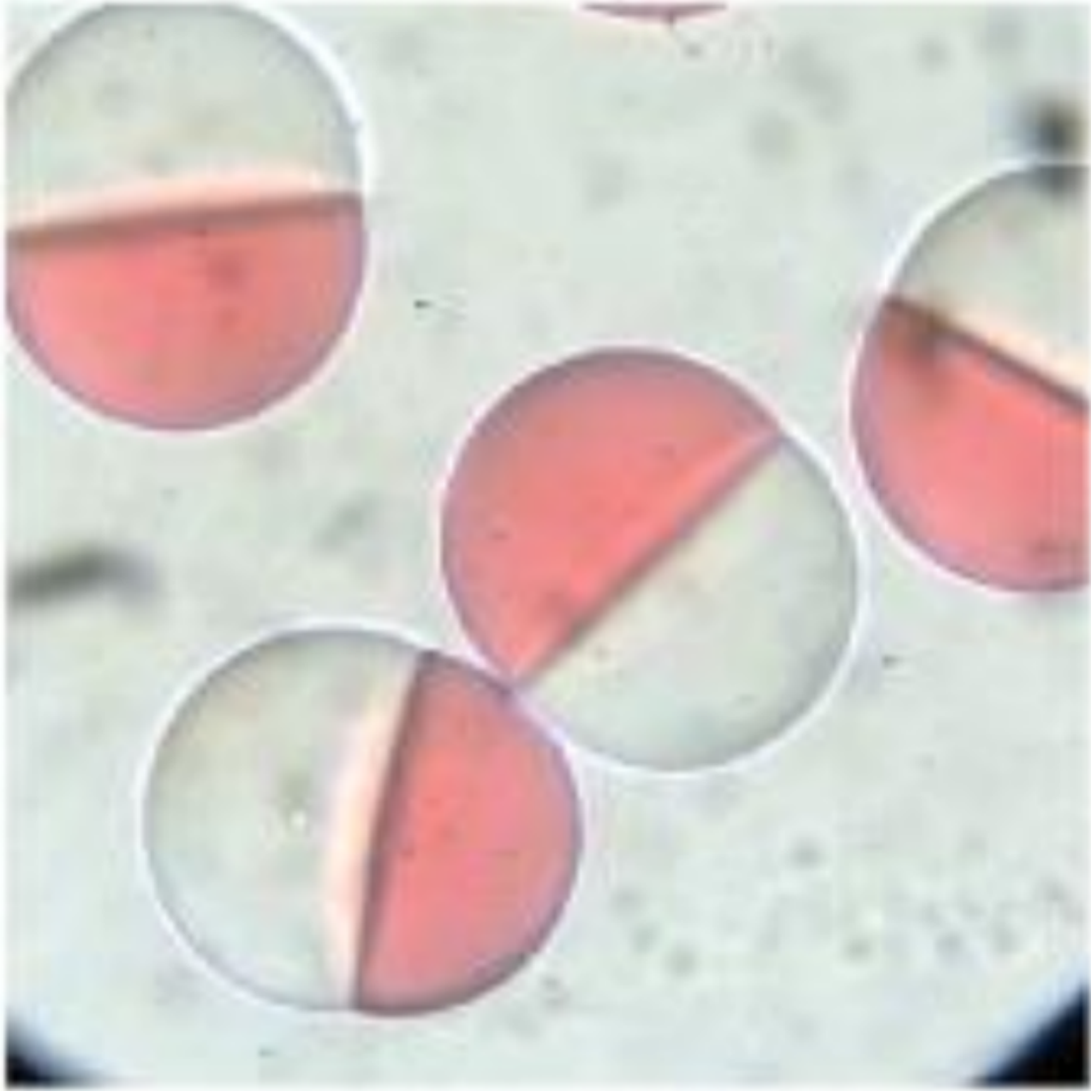}
   \caption*{(Xu and Nisisako, Micromachines 2021 \cite{Xu2021})}
   \end{subfigure}
   \begin{subfigure}{0.25\linewidth}
   \centering
   \includegraphics[trim=0.5cm 0.25cm 0cm 0.25cm,clip,width=1\linewidth]{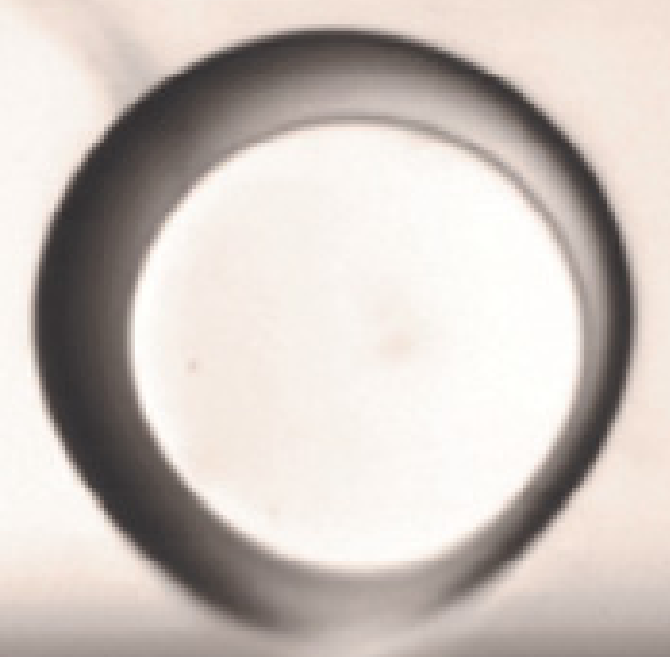}
   \caption*{(Chao et al, Appl. Phys. Lett. 2016 \cite{Chao2016})}
   \end{subfigure}
   \caption{The comparison between numerical final states and the experimental observations. Upper row: The numerical results with $(\sigma_{23},\sigma_{12},\sigma_{13})=(3,1,1), (1,1,1), (1,1,3)$ at $t=50$ from left to right; Lower row: The experimental references, from left to right, multi-core emulsions, Janus emulsions, and double emulsions.}\label{experimental works}
   \end{figure}
   
\section{Concluding remarks}\label{Sec6}
In this work, we employ a dichotomic  represent for multiphase problems and propose a dichotomy-based phase-field (DBPF) model for $N$-phase systems. In comparison to the classical characteristic-based phase-field (CBPF) model, the free energy in the DBPF model is simply an interpolation of the standard two-phase Ginzburg-Landau free energies, and thus can be easily generalized to $N$-phase problems through a recursive approach. We propose a systematic procedure to construct the surface tension functions, ensuring the mechanic, energetic, algebraic, and dynamic consistencies. Moreover, the newly proposed DBPF model eliminates the need for simplex constraints on phase-field variables, thus offering greater flexibility in the development of numerical schemes. 
The sharp-interface limit of the DBPF model is analyzed around the two-phase interfaces and the three-phase triple junctions, showing its asymptotic consistency with the corresponding sharp-interface model. In particular, we recover the Neumann triangle condition at the triple junction. 

By employing the mobility operator splitting (MOS) technique and the modified Crank-Nicolson method, we develop a second-order, linear, decoupled, and energy stable scheme for solving the ternary DBPF model, which preserves the energy decay property for the original free energy.  
Extensive numerical results are presented to validate the properties of the DBPF model and its numerical method. Benchmark simulations on the formation of liquid lenses between two stratified fluids and the generation of compound droplet emulsions show good agreement with experimental observations. The DBPF model demonstrates substantial potential for broad applications, particularly in the computational study of multiphase flow systems---a direction we will pursue in future work.

\section*{}
\appendix
\renewcommand{\thetheorem}{\Alph{section}.\arabic{theorem}}

\section{Proof of Lemma \ref{construction of f}}\label{proof of Lemma}
\begin{proof}
Firstly, we construct $f$ as follows:
\begin{align}\label{f}
f=\sum_{j=1}^N a_j h_j+
   \sum_{j=1}^{N-2}(-1)^j e_j,
\end{align}
where 
\begin{align}\label{a and b}
   \begin{cases}
   a_j=\big(\frac{1+z_j}{2}\big)^2(2-z_j),
   \ 1\leqslant j\leqslant N-1, 
   \\
   a_N=\big(\frac{1-z_{N-1}}{2}\big)^2(2+z_{N-1}),
   \end{cases}
\end{align}
and
\begin{align*}
   e_j=\sum_{1\leqslant\alpha_1<\ldots<\alpha_j
   \leqslant N-2}a_{\alpha_1}\ldots a_{\alpha_j}
   \sum_{k=\alpha_j+1}^{N}a_k\cdot h_k|_{x_{\alpha_1}
   =\cdots=x_{\alpha_j}=1}.
\end{align*}
Then, we prove that $f$ satisfies the boundary constraints \eqref{boundary of f} and first-order derivative conditions \eqref{1st_derivative}.

Without lose of generality, we set $z_m=1$, $1\leqslant m\leqslant N-2$. Noting that we want to show
   \begin{align}\label{f|z_m=1}
   f(\Pi_{m,1}^{N-1}\mathbf{Z})=
   h_m. \end{align}
Let $z_m=1$ in \eqref{f}, we have $a_m=1$ and
\begin{align*}
   f(\Pi_{m,1}^{N-1}\mathbf{Z})=
   \sum_{\substack{j=1\\ j\neq m}}^N a_j h_j|_{z_m=1}+h_m+
   \sum_{j=1}^{N-2}(-1)^j e_j|_{z_m=1}.
\end{align*}
To obtain \eqref{f|z_m=1}, we only need to show 
   \begin{align}\label{only need to show}
   \sum_{\substack{j=1\\ j\neq m}}^N a_j h_j|_{z_m=1}+
   \sum_{j=1}^{N-2}(-1)^j e_j|_{z_m=1}=0.
   \end{align}
Through some calculations, we can obtain
   \begin{align}\label{e_1}
        e_1|_{z_m=1}
        =&\Big(\sum_{1\leqslant\alpha_1\leqslant N-2}a_{\alpha_1}\sum_{k=\alpha_1+1}^{N}a_k\cdot h_k|_{z_{\alpha_1}=1}\Big)\Big\vert_{z_m=1}
        \notag\\
        =&\sum_{1\leqslant\alpha_1<m\leqslant N-2}a_{\alpha_1}\Big(\sum_{\substack{k=\alpha_1+1 \\k\neq m}}^{N}a_k\cdot h_k|_{z_{\alpha_1}=z_m=1}+ h_m|_{z_{\alpha_1}=1}\Big)
        \notag\\
        &+\sum_{k=m+1}^{N}a_k\cdot h_k|_{z_m=1}+\sum_{1\leqslant m<\alpha_1\leqslant N-2}a_{\alpha_1}\sum_{k=\alpha_1+1}^{N}a_k\cdot h_k|_{z_{\alpha_1}=z_m=1}
        \notag\\
        =&\sum_{\substack{j=1 \\ j\neq m}}
        ^N a_j h_j|_{z_m=1}+\sum_{\substack{1\leqslant\alpha_1\leqslant N-2\\ \alpha_1\neq m}}a_{\alpha_1}\sum_{\substack{k=\alpha_1+1 \\ k\neq m}}^{N}a_k\cdot h_k|_{z_{\alpha_1}=z_m=1}
        \notag\\
        :=&\sum_{\substack{j=1 \\ j\neq m}}
        ^N a_j h_j|_{z_m=1}+R_1,
   \end{align}
where the third equality holds due to the condition \eqref{continuity condition}.
Similarly, for $2\leqslant j\leqslant N-3$, performing a direct calculation yields
    \begin{align}\label{e_j}
        e_j|_{z_m=1}=&\Big(\sum_{1\leqslant\alpha_1<
        \ldots<\alpha_j
        \leqslant N-2}a_{\alpha_1}\ldots a_{\alpha_j}
        \sum_{k=\alpha_j+1}^{N}a_k\cdot h_k|_{x_{\alpha_1}
        =\cdots=x_{\alpha_j}=1}\Big)\Big\vert_{z_m=1}
        \notag\\
        =&\sum_{\substack{1\leqslant\alpha_1\leqslant\cdots\leqslant\alpha_{j-1}\leqslant N-2\\ \alpha_1,\cdots,\alpha_{j-1}\neq m}}a_{\alpha_1}\ldots a_{\alpha_{j-1}}\sum_{\substack{k=\alpha_{j-1}+1 \\ k\neq m}}^{N}a_k\cdot h_k|_{z_{\alpha_1}=\cdots=z_{\alpha_{j-1}}=z_m=1}
        \notag\\
        &+\sum_{\substack{1\leqslant\alpha_1
        \leqslant\cdots\leqslant\alpha_{j}\leqslant N-2\\ \alpha_1,\cdots,\alpha_{j}\neq m}}a_{\alpha_1}\ldots a_{\alpha_{j}}\sum_{\substack{k=\alpha_{j}+1 \\ k\neq m}}^{N}a_k\cdot h_k|_{z_{\alpha_1}=\cdots=z_{\alpha_{j}}=z_m=1}
        \notag\\
        :=&~R_{j-1}+R_j.
    \end{align}
For $j=N-2$, we have
    \begin{align}\label{e_{N-2}}
        e_{N-2}|_{z_m=1}=&~a_1\cdots a_{m-1}a_{m+1}\cdots a_{N-2}\sum_{k=N-1}^N(a_k\cdot h_k|_{z_1=\cdots=z_{N-2}=z_m=1})
        \notag\\
        =&~R_{N-3}.
    \end{align}
Combining \eqref{e_1}--\eqref{e_{N-2}}, we can obtain
    \begin{align*}
        \sum_{\substack{j=1\\ j\neq m}}^N a_j h_j|_{z_m=1}+
        \sum_{j=1}^{N-2}(-1)^j e_j|_{z_m=1}=&
        \sum_{\substack{j=1\\ j\neq m}}^N a_j h_j|_{z_m=1}-\sum_{\substack{j=1\\ j\neq m}}^N a_j h_j|_{z_m=1}-R_1
        \\
        &+\sum_{j=1}^{N-3}(-1)^j(R_{j-1}+R_j)+(-1)^{N-2}R_{N-3}
        \\
        =&~0.
    \end{align*}
Hence, $f$ satisfies the boundary constraints \eqref{boundary of f} for $1\leqslant m\leqslant N-2$. \

Next, we prove that $f$ satisfies the first-order derivative conditions \eqref{1st_derivative}. It is worth noting that $a_j'(\pm1)=0, \ 1\leqslant j\leqslant N$. Taking the derivative of $e_j$ with respect to $z_m$ and letting $z_m=1$ yields
    \begin{align*}
        \frac{\partial e_j}{\partial z_m}\Big|_{z_m=1}
        =&\sum_{\substack{1\leqslant\alpha_1\leqslant
        \cdots\leqslant\alpha_j\leqslant N-2\\ m\in\{\alpha_1,\cdots,\alpha_j\}}}^N a_{\alpha_1}\cdots a_m'(1)\cdots a_{\alpha_j}
        \sum_{k=\alpha_j+1}^Na_k\cdot h_k|_{z_{\alpha_1}=\cdots=z_{\alpha_j}=z_m=1}
        \\
        &+\sum_{\substack{1\leqslant\alpha_1\leqslant
        \cdots\leqslant\alpha_j\leqslant N-2\\ \alpha_1,\cdots,\alpha_j\neq m}}^N a_{\alpha_1}\cdots a_{\alpha_j}
        \sum_{\substack{k=\alpha_j+1\\ k\neq m}}^Na_k\cdot \frac{\partial h_k}{\partial z_m}|_{z_{\alpha_1}=\cdots=z_{\alpha_j}=z_m=1}
        \\
        &+\sum_{\substack{1\leqslant\alpha_1\leqslant
        \cdots\leqslant\alpha_j\leqslant N-2\\ \alpha_1,\cdots,\alpha_j\neq m}}^N a_{\alpha_1}\cdots a_{\alpha_j}
        a_m'(1)\cdot h_m|_{z_{\alpha_1}=\cdots=
        z_{\alpha_j}=z_m=1}
        \\
        =&~0,
    \end{align*}
where  the last equality is obtained from condition \eqref{1st derivative condition of h_i}. Taking the derivative of $f$ and applying condition \eqref{1st derivative condition of h_i} again gives
    \begin{align*}
        \frac{\partial f}{\partial z_m}\Big|_{z_m=1}=&
        \sum_{\substack{j=1\\ j\neq m}}^N a_j \frac{\partial h_j}{\partial z_m}|_{z_m=1}+a_m'(1)h_m+\sum_{j=1}^{N-2} (-1)^j\frac{\partial e_j}{\partial z_m}\Big|_{z_m=1}
        =0.
    \end{align*}
Hence, $f$ satisfies the first-order derivative conditions \eqref{1st_derivative} for $1\leqslant m\leqslant N-2$. The proof for the remaining case 
$z_m=\pm 1 \ (m=N\!-\!1)$ follows a similar approach. 
\end{proof}

\section{Construction of $\gamma_i^N$}\label{Construction of gamma_i^N}
For convenience of construction,  we let $\gamma_i^N({\mathbf{\Phi}}^{N})$ be independent of $\phi_i$. In this case, we consider $\phi_i$ to be an arbitrary real number in $[-1,1]$. As long as $\gamma_i^{N-1}$ is given that satisfies the consistencies \eqref{Mechanic consistency}-\eqref{Dynamic consistency},  we can give an effective approach to construct $\gamma_i^N$, $i=1,\ldots,N\!-\!1$. 

Assume that $\gamma_i^N$ has the form
\begin{align*}
\gamma_i^N({\mathbf{\Phi}}^{N})=\widehat{\gamma}_i^N({\mathbf{\Phi}}^{N})+c_i^N({\mathbf{\Phi}}^{N}),
\end{align*}
where $\widehat{\gamma}_i^N$ is introduced to ensure that $\gamma_i^N$ satisfies consistencies \eqref{Mechanic consistency}--\eqref{Algebraic consistency} with $c_i^N$ has no effect on the consistencies \eqref{Mechanic consistency}-- \eqref{Algebraic consistency}. $c_i^N$ only ensures the consistency \eqref{Dynamic consistency}. Thus, to obtain a simple form, we let 
\begin{align}\label{condition of c_i^N}
   \begin{cases}
   c_i^N({\mathbf{\Phi}}^{N})=0,
   \quad \forall\ {\mathbf{\Phi}}^{N}
   \in
   \begin{cases}
   \displaystyle\bigcup_{k}\mathcal{I}^{ik},&i=1,\\
   \Big(\displaystyle\bigcup_{k}\mathcal{I}^{ik}\Big)\bigcup
   \Big(\displaystyle\bigcup_{n=1}^{i-1}\mathcal{B}^n\Big), &2\leqslant i\leqslant N-2,
   \end{cases}
   \vspace{1mm}\\
   \dfrac{\partial c_i^N}{\partial \phi_j}\bigg|_{{\mathbf{\Phi}}^{N}\in\Pi_{j,b}^{N-1}Q_{N-1}}=0, \quad (j,b)\in I_a, \ 1\leqslant i\leqslant N-1, \ i\neq j.
   \end{cases}
\end{align}
Assume that 
\begin{align*}
   \widehat{\gamma}_i^N(\mathbf{\Phi}^{N})=
   \begin{cases}
   h_j^i, &{\mathbf{\Phi}}^{N}\in\Pi_{j,1}^{N-1}Q_{N-1},\ 1\leqslant i, j\leqslant N-1, \ i\neq j\\ 
   h_{N}^i, &{\mathbf{\Phi}}^{N}\in\Pi_{N-1,-1}^{N-1}Q_{N-1},\ 1\leqslant i\leqslant N-2,
   \end{cases}
\end{align*}
where 
\begin{align*}
   h_j^i
   =
   \begin{cases}
   \gamma_{i-1}^{N-1}(\mathbf{\Phi}_{(-j)}^{N};
   \mathbf{\Sigma}^{N,i}), &i>j\geqslant 1,\\
   \gamma_{i}^{N-1}(\mathbf{\Phi}_{(-j)}^{N};
   \mathbf{\Sigma}_{(-j)}^{N,i}), &i<j\leqslant N-1, 
   \end{cases}
\end{align*}
and 
   \begin{align*}
   h_{N}^i=\gamma_i^{N-1}(\mathbf{\Phi}^{N-1}
   ;\mathbf{\Sigma}^{N-1,i})), \quad 1\leqslant i\leqslant N-2. 
   \end{align*}
For $N=2$, we let $h_2^1=\frac{3}{2\sqrt2}\sigma_{13}$, $h_3^1=\frac{3}{2\sqrt2}\sigma_{12}$, and $h_1^2=\frac{3}{2\sqrt2}\sigma_{23}$. Since $\gamma_i^{N-1} (1\leqslant i\leqslant N-2)$ satisfies consistency \eqref{Mechanic consistency}, it is easy to verify that $\widehat{\gamma}_i^N(\mathbf{\Phi}^{N})$ already satisfies consistency \eqref{Mechanic consistency} under the assumption. It is obvious that $W^N|_{{\mathbf{\Phi}}^{N}\in\Pi_{N-1,-1}^{N-1}Q_{N-1}}=W^{N-1}(\mathbf{\Phi}^{N-1};\bigcup_{i<N-1}\mathbf{\Sigma}^{N-1,i})$. Moreover, for $1\leqslant j\leqslant N-1$, we have
\begin{align*}
   W^N|_{{\mathbf{\Phi}}^{N}\in\Pi_{j,1}^{N-1}Q_{N-1}}=&\int_{\Omega}\sum_{i=1,i\neq j}^{N-1}h_j^i g(\phi_i)\mathrm{d}x
   \\
   =&\int_{\Omega}\sum_{i=1}^{j-1}\gamma_i^{N-1}
   (\mathbf{\Phi}_{(-j)}^{N};\mathbf{\Sigma}_{(-j)}^{N,i}))g(\phi_i)
   \\
   &+\sum_{i=j+1}^{N-1}\gamma_{i-1}^{N-1}
   (\mathbf{\Phi}_{(-j)}^{N};\mathbf{\Sigma}^{N,i}))g(\phi_i)
   \mathrm{d}x\\
   =&~W^{N-1}(\mathbf{\Phi}_{(-j)}^{N};\bigcup_{i\neq j}\mathbf{\Sigma}_{(-j)}^{N,i}),
\end{align*}
where $g(\phi_i)=\frac{\varepsilon}{2}|\nabla\phi_i|^{2}+\frac{1}{\varepsilon}F(\phi_i) $.  Noting that ${\mathbf{\Phi}}^{N}\in[-1,1]^{N-1}$ and $h_1^i,\ldots,h_N^i$ are the boundary conditions of the unknown function $\widehat{\gamma}_i^N(\mathbf{\Phi}^{N})$. we can easily check that $h_1^i,\ldots,h_N^i$ are continuous at their intersections (points, lines, surfaces, volumes, etc.). Thus, the problem now translates to how to utilize the boundary information of $\widehat{\gamma}_i^N$ to construct the interior of $\widehat{\gamma}_i^N$.

For $1\leqslant i\leqslant N-2$, noting that $\widehat\gamma_i^N({\mathbf{\Phi}}^{N})$ is independent of $\phi_i$, and following the proof of Lemma \ref{construction of f} and Corollary \ref{corollary 1}, we can construct $\widehat{\gamma}_i^N$ as follows:
\begin{align}\label{widehat_gamma}
\widehat{\gamma}_i^N=\sum_{j=1}^N a_j h_j^i+
   \sum_{j=1}^{N-2}(-1)^j e_j,
\end{align}
where $h_j^i=0$ if $i=j$,
\begin{align*}
   a_j=
   \begin{cases}
   \big(\frac{1+x_j}{2}\big)^2(2-x_j), &1\leqslant j\leqslant N-1, i\neq j,\\
   0, &1\leqslant i=j\leqslant N-1,\\
   \big(\frac{1-x_{j-1}}{2}\big)^2(2+x_{j-1}), &j=N,
   \end{cases}
\end{align*}
and
\begin{align*}
   e_j=\sum_{1\leqslant\alpha_1<\ldots<\alpha_j
   \leqslant N-2}a_{\alpha_1}\ldots a_{\alpha_j}
   \sum_{k=\alpha_j+1}^{N}a_k\cdot h_k^i|_{x_{\alpha_1}
   =\cdots=x_{\alpha_j}=1}.
\end{align*}   
Then $\widehat{\gamma}_i^N$ satisfies the consistencies \eqref{Energetic consistency} and \eqref{Algebraic consistency}. 
In addition, according to \eqref{condition of c_i^N}, we can let $c_i^N$ have the following form:
\begin{align*}
   c_i^N=\Lambda\prod_{\substack{j=1\\j\neq i}}^{N-1}(1-\phi_j^2)^2,
\end{align*}
where $\Lambda$ is a positive constant. Finally, in order to ensure the dynamic consistency \eqref{Dynamic consistency},  we can determine the constant $\Lambda$ by requiring that 
\begin{align}\label{condition of c}
   \frac{\partial^2}{\partial \phi_j^2}(\widehat{\gamma}_i^N+c_i^N)\big|_{{\mathbf{\Phi}}^{N}\in\Pi_{j,b}^{N-1}Q_{N-1}}> 0, \quad (j,b)\in I_a, \ 1\leqslant i\leqslant N-1, \ i\neq j.
\end{align}
In summary, we can obtain $\gamma_i^{N}$, $1\leqslant i\leqslant N-2$ as follows:
\begin{align}\label{formula of gamma_i^N}
\gamma_i^N=\sum_{j=1}^N a_j h_j^i+
   \sum_{j=1}^{N-2}(-1)^j e_j+
   \Lambda\prod_{\substack{j=1\\j\neq i}}^{N-1}(1-\phi_j^2)^2.
\end{align}

For $i=N\!-\!1$, since $\gamma_{N-1}^N$ is independent of $\phi_{N-1}$ and $\phi_N$, we can additionally set $a_{N-1}=a_N=h_{N-1}^{N-1}=h_N^{N-1}=0$ in \eqref{widehat_gamma} to obtain $\gamma_{N-1}^N$ by the similar method.

\section{Proof of Theorem \ref{Thm_energy_stable}}\label{proof_of_Thm}
\begin{proof}
Taking the inner product of \eqref{second order scheme3 1} with $(\tau/2)\mu_{\varphi}^{n+\frac{1}{4}}$ and \eqref{second order scheme3 2} with $(\varphi^{n+\frac{1}{2}}-\varphi^n)$, respectively. Combining the results, we get
   \begin{align}\label{second order3 inner product}
   &-\frac{\tau}{2}\int_{\Omega}M_2(\psi^n)|\nabla\mu_{\varphi}^{n+\frac{1}{4}}|
   ^2\mathrm{d}\mathbf{x}\notag\\=&\int_{\Omega}\gamma_2(\psi^n)\Big(\frac{\varepsilon}{2}|\nabla\varphi^{n+\frac{1}{2}}|^2
   -\frac{\varepsilon}{2}|\nabla\varphi^{n}|^2 \Big)\mathrm{d}\mathbf{x} 
   +\frac{1}{\varepsilon}\int_{\Omega}\gamma_2(\psi^n)f(\varphi^{n}) (\varphi^{n+\frac{1}{2}}-\varphi^n)\mathrm{d}\mathbf{x}
   \notag\\
   &+\frac{1}{2\varepsilon}\int_{\Omega}\gamma_2(\psi^n)f'(\varphi^n)(\varphi^{n+\frac{1}{2}}-\varphi^n)^2\mathrm{d}\mathbf{x}
   +\frac{A_1\tau}{\varepsilon}\int_{\Omega}\gamma_2(\psi^n)(\varphi^{n+\frac{1}{2}}-\varphi^{n})^2\mathrm{d}\mathbf{x}
   \notag\\
   &+\int_{\Omega}\gamma_1'(\varphi^n)(\varphi^{n+\frac{1}{2}}-\varphi^{n})g(\psi^n)\mathrm{d}\mathbf{x}
 +\frac{1}{2}\int_{\Omega}\gamma_1''(\varphi^n)g(\psi^{n})(\varphi^{n+\frac{1}{2}}-\varphi^{n})^2\mathrm{d}\mathbf{x}
   \notag\\
   &+B_1\tau\int_{\Omega}g(\psi^n)(\varphi^{n+\frac{1}{2}}-\varphi^{n})^2\mathrm{d}\mathbf{x}.
   \end{align}
For the second term on the right side  of \eqref{second order3 inner product}, by Taylor expansions, we have
   \begin{align}\label{second order3 Term A}
   \frac{1}{\varepsilon}\int_{\Omega}\gamma_2(\psi^n)f(\varphi^{n}) (\varphi^{n+\frac{1}{2}}-\varphi^n)\mathrm{d}\mathbf{x}\geqslant&~\frac{1}{\varepsilon}\int_{\Omega}\gamma_2(\psi^n)
   \big(F(\varphi^{n+\frac{1}{2}})-F(\varphi^n)\big)\mathrm{d}\mathbf{x}
   \notag\\
   &-\frac{L_{F}}{2\varepsilon}\int_{\Omega}\gamma_2(\psi^n)(\varphi^{n+\frac{1}{2}}-\varphi^n)^2\mathrm{d}\mathbf{x}.
   \end{align}
Similarly, we can get
   \begin{align}\label{second order3 Term C}
   \int_{\Omega}\gamma_1'(\varphi^n)(\varphi^{n+\frac{1}{2}}-\varphi^{n})g(\psi^n)\mathrm{d}\mathbf{x}\geqslant&
   \int_{\Omega}\big(\gamma_1(\varphi^{n+\frac{1}{2}})-\gamma_1(\varphi^n)\big)g(\psi^n)\mathrm{d}\mathbf{x} \notag\\
   &-\frac{L_{\gamma_1}}{2}\int_{\Omega}g(\psi^n)(\varphi^{n+\frac{1}{2}}-\varphi^n)^2\mathrm{d}\mathbf{x},
  \end{align}
  
   \begin{align}\label{second order3 Term B}
   \frac{1}{2\varepsilon}\int_{\Omega}\gamma_2(\psi^n)f'(\varphi^n)(\varphi^{n+\frac{1}{2}}-\varphi^n)^2\mathrm{d}\mathbf{x}\geqslant
   -\frac{L_{F}}{2\varepsilon}\int_{\Omega}\gamma_2(\psi^n)(\varphi^{n+\frac{1}{2}}-\varphi^n)^2\mathrm{d}\mathbf{x},
   \end{align}
and
   \begin{align}\label{second order3 Term D}
   \frac{1}{2}\int_{\Omega}\gamma_1''(\varphi^n)g(\psi^{n})(\varphi^{n+\frac{1}{2}}-\varphi^{n})^2\mathrm{d}\mathbf{x}\geqslant
   -\frac{L_{\gamma_1}}{2}\int_{\Omega}g(\psi^{n})(\varphi^{n+\frac{1}{2}}-\varphi^{n})^2\mathrm{d}\mathbf{x}.
   \end{align}
Combining \eqref{second order3 inner product}--\eqref{second order3 Term D}, we obtain
   \begin{align*}
   &~W(\psi^n,\varphi^{n+\frac{1}{2}})-W(\psi^n,\varphi^n)
   +\Big(\frac{A_1\tau}{\varepsilon}-\frac{L_{F}}{\varepsilon}\Big)\int_{\Omega}\gamma_2(\psi^n)(\varphi^{n+\frac{1}{2}}-\varphi^n)^2\mathrm{d}\mathbf{x}
   \notag\\
   &+(B_1\tau-L_{\gamma_1})\int_{\Omega}g(\psi^n)(\varphi^{n+\frac{1}{2}}-\varphi^n)^2\mathrm{d}\mathbf{x}\notag\\
   \leqslant&   -\frac{\tau}{2}\int_{\Omega}M_2(\psi^n)|\nabla\mu_{\varphi}^{n+\frac{1}{4}}|^2\mathrm{d}\mathbf{x}.
   \end{align*}
Under the conditions $A_1\tau\geqslant L_{F}$ and $B_1\tau\geqslant L_{\gamma_1}$, we have
   \begin{align}\label{ineuqality_1}
   W(\psi^n,\varphi^{n+\frac{1}{2}})-W(\psi^n,\varphi^n)\leqslant
   -\frac{\tau}{2}\int_{\Omega}M_2(\psi^n)|\nabla\mu_{\varphi}^{n+\frac{1}{4}}|^2\mathrm{d}\mathbf{x}.
   \end{align}
Similarly, for \eqref{second order scheme3 3} and \eqref{second order scheme3 4}, we have
   \begin{align}\label{ineuqality_2}
   W(\psi^{n+1},\varphi^{n+\frac{1}{2}})-W(\psi^n,\varphi^{n+\frac{1}{2}})\leqslant
   -\tau\int_{\Omega}M_1|\nabla\mu_{\psi}^{n+\frac{1}{2}}|^2\mathrm{d}\mathbf{x}.
   \end{align}
For \eqref{second order scheme3 5} and \eqref{second order scheme3 6}, we have  
   \begin{align}\label{ineuqality_3}
   W(\psi^{n+1},\varphi^{n+1})-W(\psi^{n+1},\varphi^{n+\frac{1}{2}})\leqslant
   -\frac{\tau}{2}\int_{\Omega}M_2(\psi^{n+1})|\nabla\mu_{\varphi}^{n+\frac{3}{4}}|^2\mathrm{d}\mathbf{x}.
   \end{align}
Combining \eqref{ineuqality_1}, \eqref{ineuqality_2} and \eqref{ineuqality_3}, we obtain
   \begin{align*}
   W(\psi^{n+1},\varphi^{n+1})
   \leqslant&~W(\psi^n,\varphi^{n})-\frac{\tau}{2}\int_{\Omega}M_2(\psi^n)|\nabla\mu_{\varphi}^{n+\frac{1}{4}}|^2\mathrm{d}\mathbf{x}
   \notag\\
   &-\tau\int_{\Omega}M_1|\nabla\mu_{\psi}^{n+\frac{1}{2}}|^2\mathrm{d}\mathbf{x} 
   -\frac{\tau}{2}\int_{\Omega}M_2(\psi^{n+1})|\nabla\mu_{\varphi}^{n+\frac{3}{4}}|^2\mathrm{d}\mathbf{x}.
   \notag\\
   \end{align*}
Thus, we complete the proof.
\end{proof}

\section*{Acknowledgments}
The work of Zhen Zhang was partially supported by National Key R\&D Program of China (2023YFA1011403) and the NSFC grant (92470112). The work of Chenxi Wang was partially supported by the NSFC grant (NO. 12401524).

\bibliography{sample} 
\bibliographystyle{ws-m3as}

\end{document}